\documentclass[11pt]{article}

\usepackage{amsmath, amsthm,amsfonts,amssymb,mathtools,mathdots,scalerel}

\usepackage{enumitem, tocloft}
\usepackage[hidelinks]{hyperref}
\usepackage{tikz, stmaryrd, cmll}
\usepackage[british]{babel}
\usepackage[a4paper]{geometry}

\newtheorem{thm}{Theorem}[section]
\newtheorem{prop}[thm]{Proposition}
\newtheorem{cor}[thm]{Corollary}
\newtheorem{lem}[thm]{Lemma}

\newtheorem{conj}[thm]{Conjecture}
\theoremstyle{definition}
\newtheorem{dfn}[thm]{Definition}
\newtheorem{cons}[thm]{Construction}
\theoremstyle{remark}
\newtheorem{remark}[thm]{Remark}
\newtheorem{exm}[thm]{Example}

\usetikzlibrary{arrows, decorations.markings, shapes.geometric, decorations.pathmorphing, decorations.pathreplacing, intersections, patterns,calc, backgrounds}

\tikzset{
	0c/.style={circle, draw, fill, inner sep=1.5pt},
	1c/.style={->, thick, shorten <=2pt, shorten >=2pt},
	1clong/.style={->, thick},
	edge/.style={thick, shorten <=2pt, shorten >=2pt},
	equal/.style={thick, shorten <=2pt, shorten >=2pt, double},
	1cdash/.style={->, densely dashed, thick, shorten <=2pt, shorten >=2pt},
	1cdot/.style={->, dotted, thick, shorten <=2pt, shorten >=2pt},
	1cinc/.style={right hook->, thick, shorten <=2pt, shorten >=2pt},
	1csurj/.style={->>, thick, shorten <=2pt, shorten >=2pt},
	1cincl/.style={left hook->, thick, shorten <=2pt, shorten >=2pt},
	2c/.style={double, thick, shorten <=6pt, shorten >=8pt, decoration={markings,mark=at position -6pt with {\arrow[scale=1.75]{>}}}, preaction={decorate}},
	3c1/.style={thick, double, double distance=3pt, shorten <=9pt, shorten >=11pt},
    	3c2/.style={thick, shorten <=9pt, shorten >=10pt},
	3c3/.style={shorten <=9pt, shorten >=10pt, decoration={markings,mark=at position -8pt with {\arrow[scale=3]{>}}},preaction={decorate}},
	follow/.style={->, >=stealth, ultra thick, shorten <=3pt, shorten >=3pt, color=gray!70},
	every node/.style={scale=.8},
}

\newcommand\cp[1]{\,{\scriptstyle\#}_{#1}\,}
\newcommand\cpm[1]{*_{#1}}
\newcommand\comp[1]{\triangleright_{#1}}
\newcommand\compos[1]{\llbracket#1\rrbracket}
\newcommand\bord[2]{\partial_{#1}^{#2}}
\newcommand\idd[1]{\varepsilon_{#1}}
\newcommand\idcat[1]{\mathrm{id}_{#1}}
\newcommand\thin[1]{1_{#1}}
\newcommand\opp[1]{{#1}^\mathrm{op}}
\newcommand\coo[1]{{#1}^\mathrm{co}}
\newcommand\oppall[1]{{#1}^\circ}
\newcommand\oppn[2]{D_{#1}({#2})}

\newcommand\invrs[1]{#1^{-1}}
\newcommand\skel[2]{\sigma_{\leq #1}#2}
\newcommand\coskel[2]{\tau_{\leq #1}#2}
\newcommand\pfun{\rightharpoonup}

\newcommand\homset[1]{\mathrm{Hom}_{#1}}
\newcommand\nbd{\nobreakdash-\hspace{0pt}}

\newcommand\sbord[2]{\Delta_{#1}^{#2}}
\newcommand\dmn[1]{\mathrm{dim}(#1)}

\newcommand\clos[1]{\mathrm{cl}#1}
\newcommand\hass[1]{\mathcal{H}#1}

\newcommand\nondeg[2]{\mathcal{A}_{#1}{#2}}
\newcommand\molec[2]{\mathcal{M}o\ell_{#1}{#2}}
\newcommand\shell[1]{\mathcal{S}{#1}}

\newcommand\submol{\sqsubseteq}
\newcommand\subsph{\sqsubseteq_s}
\newcommand\supsph{\sqsupseteq_s}
\newcommand\tensor{\,{\scaleobj{0.75}{\boxtimes}}\,}
\newcommand\join{\,{\star}\,}
\newcommand\face[1]{\mathfrak{F}{#1}}
\newcommand\incl{\hookrightarrow}
\newcommand\incliso{\stackrel{\sim}{\hookrightarrow}}
\newcommand\surj{\twoheadrightarrow}
\newcommand\fnct[1]{\underline{#1}}

\newcommand\homoplax[2]{[#1,#2]_l}
\newcommand\homlax[2]{[#1,#2]_r}

\newcommand\slice[2]{{#1}/{\raisebox{-2pt}{$#2$}}}

\newcommand\equi[2]{\mathcal{E}q_{#1}{#2}}
\newcommand\degg[1]{\mathcal{D}eg{#1}}
\newcommand\invrt[1]{\mathcal{I}nv{#1}}
\newcommand\hatto[1]{{#1_{\omega}}}
\newcommand\nerve[1]{\mathcal{N}#1}
\newcommand\nervealg[1]{\mathcal{N}_\textit{alg}#1}
\newcommand\celto{\Rightarrow}
\newcommand\pin[3]{\pi_{#1}^{#2}(#3)}
\newcommand\compglob[1]{\Phi^{#1}}
\newcommand\extr[2]{\mathcal{E}_{#1}(\Delta^{#2})}
\newcommand\extrtil[2]{\tilde{\mathcal{E}}_{#1}(\Delta^{#2})}
\newcommand\subdiv{\mathrm{Sd}}
\newcommand\exfun{\mathrm{Ex}}
\newcommand\realis[1]{|#1|}
\newcommand\horn[1]{\mathcal{D}iv{(#1)}}

\newcommand\satur[2]{\mathcal{T}^{#1}{(#2)}}
\newcommand\infl[1]{O{(#1)}}

\newcommand\cat[1]{\mathbf{#1}}
\newcommand\psh[2]{\mathbf{PSh}_{#1}(#2)}
\newcommand\globset{\omega\cat{Gph}}
\newcommand\globsetrefl{\omega\cat{Gph}_\textit{refl}}

\newcommand\rpol{\cat{RPol}}
\newcommand\dgmset{\cat{DgmSet}}
\newcommand\rdgmset{\cat{RDgmSet}}
\newcommand\kandgmset{\cat{DgmSet}_\textit{Kan}}
\newcommand\rdgmsetalg{\cat{RDgmSet}_\textit{alg}}
\newcommand\pol{\cat{Pol}}
\newcommand\pos{\cat{Pos}}
\newcommand\ogpos{\cat{ogPos}}
\newcommand\rdcpx{\cat{RDCpx}}
\newcommand\rdcpxin{\cat{RDCpx}_\textit{in}}
\newcommand\atomin{\cat{RAtom}_\textit{in}}
\newcommand\molecin{\cat{RMolec}_\textit{in}}

\newcommand\atom{\cat{RAtom}}
\newcommand\omegacat{\omega\cat{Cat}}
\newcommand\cghaus{\cat{cgHaus}}
\newcommand\sset{\cat{sSet}}
\newcommand\semisset{\cat{ssSet}}
\newcommand\pomegacat{\cat{p}\omega\cat{Cat}}

\newcommand\deltacat{\cat{\Delta}}
\newcommand\deltain{\cat{\Delta}_\textit{in}}

\newcommand\delres[1]{#1_\Delta}
\newcommand\bicatun{\cat{BiCat}_\textit{su}}
\newcommand\bicatst{\cat{BiCat}_\textit{s}}

\newcommand\restr[2]{{
  \left.\kern-\nulldelimiterspace 
  #1 
  \vphantom{|} 
  \right|_{#2}
  }}

\setitemize{itemsep=0pt,topsep=1ex}
\setenumerate{itemsep=0pt,topsep=1ex}

\title{\Large\bfseries Representable diagrammatic sets \\ as a model of weak higher categories}

\author{Amar Hadzihasanovic\\[3pt] RIMS, Kyoto University}

\date{September 2019}

\begin{document}
\maketitle

\vspace{-10pt}

\begin{minipage}{0.9\linewidth}
\small Developing an idea of Kapranov and Voevodsky, we introduce a model of weak $\omega$\nbd categories based on directed complexes, combinatorial presentations of pasting diagrams. We propose this as a convenient framework for higher-dimensional rewriting. 

We define diagrammatic sets to be presheaves on a category of directed complexes presenting pasting diagrams with spherical boundaries. Diagrammatic sets have structural face and degeneracy operations, but no structural composition. We define a notion of equivalence cell in a diagrammatic set, and say that a diagrammatic set is representable if all pasting diagrams with spherical boundaries are connected to individual cells --- their weak composites --- by a higher-dimensional equivalence cell. We develop the basic theory of representable diagrammatic sets (RDSs), and prove that equivalence cells satisfy the expected properties in an RDS. We study nerves of strict $\omega$\nbd categories as RDSs and prove that 2-truncated RDSs are equivalent to bicategories. 

Finally, we connect the combinatorics of diagrammatic sets and simplicial sets, and prove a version of the homotopy hypothesis for ``groupoidal'' RDSs: there exists a geometric realisation functor for diagrammatic sets, inducing isomorphisms between combinatorial and classical homotopy groups, which has a homotopical right inverse.
\end{minipage}

\tableofcontents

\section{Introduction}

This article develops the basic theory of representable diagrammatic sets (RDSs), a model of weak $\omega$\nbd categories. RDSs belong in the same family of semi-algebraic models as weak complicial sets \cite{verity2008weak}: they have structural degeneracies, or units, preserved by all morphisms, while the existence of weak composites is stated as a property of the underlying structure. Particular weak composites may then be fixed as structure, which is always weakly preserved, and can be required to be preserved strictly (compare with Kan complexes vs algebraic Kan complexes \cite{nikolaus2011algebraic}). 

The intended application of the model, and its underlying structure --- diagrammatic sets --- is \emph{higher-dimensional rewriting and (universal) algebra}, in particular, explicit presentations and computations with pasting or string diagrams. The development of ``yet another model'' is justified by the absence, in the current landscape, of a model which is both \emph{proved to be general enough}, and has a ``sufficiently strong'' \emph{pasting theorem}. 

We take the validity of the homotopy hypothesis for weak $\omega$\nbd groupoids as a minimal benchmark of generality of a model of weak $\omega$\nbd categories. In particular, we consider the following, minimal version of the homotopy hypothesis, inspired by C.\ Simpson \cite{simpson2009homotopy}: 
\begin{center}
\begin{minipage}{.9\textwidth}
\emph{There exists a realisation functor for weak $\omega$\nbd groupoids ($n$\nbd groupoids) which is essentially surjective on homotopy types (homotopy $n$\nbd types), together with natural isomorphisms from the combinatorial homotopy groups of the weak $\omega$\nbd groupoid ($n$\nbd groupoid) to the homotopy groups of its realisation.}
\end{minipage}
\end{center}
We prove in Section \ref{sec:homotopy} a slightly stronger statement for ``groupoidal'' RDSs, namely, the existence of a homotopical right inverse to the realisation functor. 

Still stronger benchmarks, yet unproven for RDSs, include the promotion of the realisation to a Quillen equivalence, or the verification of the Barwick---Schommer-Pries axioms \cite{barwick2011unicity}. Models for which the homotopy hypothesis is proven, even in this stronger form, include Barwick's $n$\nbd fold complete Segal spaces [\emph{ibid.}], Rezk's $\Theta_n$-spaces \cite{rezk2010cartesian}, and weak complicial sets (announced in \cite{riehl2018complicial}). These were developed as models for the homotopy theory of higher categories, which is ultimately agnostic about them: the tendency is towards \emph{synthetic}, model-independent mathematics, as exemplified by homotopy type theory \cite{hottbook}.

This tendency is orthogonal to the concerns of higher-dimensional rewriting, which deals with \emph{presented} higher categories, almost tautologically not invariant under equivalence. The models we mentioned rely on simplicial or cellular combinatorics, which cover a very limited range of pasting diagrams; to present and reason about $n$\nbd dimensional theories by generators and relations, and interpret them in $n$\nbd categories, the options are currently limited.
\begin{enumerate}
	\item One can restrict themselves to the low-dimensional cases for which a manageable, explicit algebraic theory of weak $n$\nbd categories exists, as with the presentations of monoidal bicategories studied in \cite{schommer2009classification}. 
	
	\item One can be satisfied with presenting \emph{strict} $\omega$\nbd categories, for which a variety of strong pasting theorems exists. This is the route primarily taken in the polygraph approach to rewriting \cite{burroni1993higher,guiraud2016polygraphs}. However the homotopy hypothesis is false already for strict 3\nbd categories \cite[Theorem 4.4.2]{simpson2009homotopy}.
	
	\item One can use a model whose generality is unknown. Most recently, Dorn introduced a semistrict model, associative $n$\nbd categories \cite{dorn2018associative}, as a foundation for the \texttt{homotopy.io} proof assistant \cite{reutter2019high}. Its homotopy hypothesis is proven up to dimension 3.
\end{enumerate}

We propose diagrammatic sets as a framework for higher-dimensional rewriting. These extend the constructible polygraphs of \cite{hadzihasanovic2018combinatorial} with more shapes and with algebraic units and degeneracies, allowing generators with ``nullary'' inputs or outputs. Diagrammatic sets are based on Steiner's combinatorics of \emph{directed complexes}, together with a regularity constraint modelled on S.\ Henry's regular polyplexes \cite{henry2018regular}. Regular directed complexes with a greatest element, called \emph{atoms}, encompass a great variety of diagram shapes: globes, oriented simplices, cubes, and, conjecturally, all positive opetopes \cite{zawadowski2017positive}. 

Like cubes, atoms are closed under lax Gray products, and like simplices they are closed under joins. They support the typical operations of higher-dimensional rewriting, such as surgery and reversal of cells. They have good geometric realisations, homeomorphic to closed balls of the appropriate dimension. Moreover, they are described by quite simple data structures --- finite posets with an edge-labelling of their Hasse diagram --- which means that, in principle, they lend themselves to formal implementation.

Representable diagrammatic sets are a model of weak higher categories in which theories presented as diagrammatic sets can be interpreted. In addition to the proof of C.\ Simpson's form of the homotopy hypothesis, as arguments for adequacy, we present a nerve functor for strict $\omega$\nbd categories, which conditionally to a conjecture on the $\omega$\nbd categories presented by regular directed complexes, becomes a full and faithful functor into the category of RDSs and ``strict morphisms'' (those that preserve a choice of weak composites); and we describe an explicit equivalence between bicategories and ``2\nbd truncated'' RDSs.

Based on our limited experimenting, turning informal diagrammatic proofs into formal constructions of cells in diagrammatic sets is quite straightforward, and only requires some tinkering with degeneracies to regularise diagrams and, if needed, form braidings or interchangers. Thus, we hope that (representable) diagrammatic sets may be an entry point into weak higher categories with a gentler learning curve for those who come from rewriting theory or ``applied category theory'', as opposed to homotopy theory or algebraic geometry.

\subsubsection*{The definition in brief}

In one sentence in the style of \cite[Appendix A.1]{cheng2004higher}, our model of weak $\omega$\nbd category is
\begin{center}
\setlength{\fboxsep}{1em}
\noindent\fbox{\begin{minipage}{.7\textwidth}A diagrammatic set in which every composition horn has an equivalence filler.\end{minipage}}
\end{center}
We will give an intuition of what each term means.

A \textbf{diagrammatic set} is a presheaf on a shape category, $\atom$, whose objects are called \emph{atoms}. The boundaries of an $(n+1)$\nbd atom are $n$\nbd molecules \emph{with spherical boundary}, which are particular pasting diagrams of $n$\nbd atoms: their essential feature is that, forgetting the orientation of cells, they correspond to regular CW decompositions of the topological $n$\nbd ball. Each $n$\nbd atom Yoneda\nbd embeds as a diagrammatic set, and this extends to general $n$\nbd molecules. If $X$ is a diagrammatic set, we call a morphism from an atom to $X$ (equivalently, an element of $X$) a \emph{cell} of $X$, and a morphism from a generic molecule with spherical boundary to $X$ a \emph{spherical diagram} in $X$. 

Like the simplex category, $\atom$ has an orthogonal factorisation system, by which we can separate operations on a diagrammatic set into \emph{faces} and \emph{degeneracies}. Compared to simplicial sets, the combinatorics of degeneracies is very rich: in Section \ref{sec:lowdim}, we will see that in a diagrammatic set we can construct spherical \emph{braiding diagrams} between any pair of $2$\nbd cells with 0-dimensional boundaries, only using the structural degeneracies.

The only thing lacking in a diagrammatic set is the ability to compose diagrams, that is, turn diagrams into individual cells. In this sense, diagrammatic sets without any additional property are already a good context for higher-dimensional rewriting theory, where one wants to separate cells (that is, generators) from diagrams (that is, non-generators). 

A \textbf{composition horn} is, essentially, a spherical $n$\nbd diagram. \emph{Filling} the composition horn is finding a single $n$\nbd cell with the same boundary as the diagram, and an $(n+1)$\nbd cell, a \emph{compositor}, connecting them; the $n$\nbd cell should be seen as a \emph{weak composite} of the diagram. Thus, asking for fillers of composition horns means asking for weak composites of cells forming a spherical diagram.

In dimension 2, among

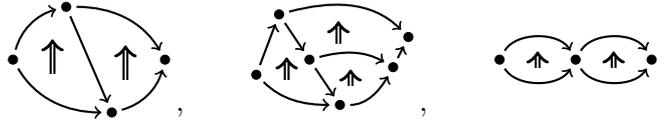
\begin{equation*}
\begin{tikzpicture}[baseline={([yshift=-.5ex]current bounding box.center)},scale=.8]
\begin{scope}
	\node[0c] (a) at (-1.25,0) {};
	\node[0c] (b) at (.375,-.875) {};
	\node[0c] (c) at (-.375,.875) {};
	\node[0c] (d) at (1.25,0) {};
	\draw[1c, out=-60, in=180] (a) to (b);
	\draw[1c, out=0, in=120] (c) to (d);
	\draw[1c, out=15, in=-105] (b) to (d);
	\draw[1c, out=75, in=-165] (a) to (c);
	\draw[1c] (c) to (b);
	\draw[2c] (-.6,-.5) to (-.6,.6);
	\draw[2c] (.6,-.6) to (.6,.5);
	\node[scale=1.25] at (1.5,-.875) {,};
\end{scope}
\begin{scope}[shift={(4,0)}]
	\node[0c] (i2) at (1,-.125) {};
	\node[0c] (0) at (-1.25,-.25) {};
	\node[0c] (1) at (1.25,.375) {};
	\node[0c] (o1) at (-.875,.75) {};
	\node[0c] (i1) at (.125,-.75) {};
	\node[0c] (m1) at (-.375,0) {};
	\draw[1c] (0) to (o1);
	\draw[1c, out=15, in=135] (o1) to (1);
	\draw[1c] (i2) to (1);
	\draw[1c, out=0, in=-120] (i1) to (i2);
	\draw[1c, out=-45, in=180] (0) to (i1);
	\draw[1c] (o1) to (m1);
	\draw[1c] (m1) to (i1);
	\draw[1c, out=15, in=150] (m1) to (i2);
	\draw[2c] (-.75,-.6) to (-.75,.3);
	\draw[2c] (.3,-.7) to (.3,.1);
	\draw[2c] (.1,0) to (.1,.9);
	\node[scale=1.25] at (1.5,-.875) {,};
\end{scope}
\begin{scope}[shift={(8,0)}]
	\node[0c] (0) at (-1.25,0) {};
	\node[0c] (m) at (0,0) {};
	\node[0c] (1) at (1.25,0) {};
	\draw[1c, out=60,in=120] (0) to (m);
	\draw[1c, out=60,in=120] (m) to (1);
	\draw[1c, out=-60,in=-120] (0) to (m);
	\draw[1c, out=-60,in=-120] (m) to (1);
	\draw[2c] (-.625,-.4) to (-.625,.4);
	\draw[2c] (.625,-.4) to (.625,.4);
	\node[scale=1.25] at (1.5,-.875) {,};
\end{scope}
\end{tikzpicture}
\end{equation*}

\noindent the first two diagrams are regular, hence weakly composable, but the third is not; nevertheless, it can be ``regularised'' by inserting degeneracies appropriately, for example

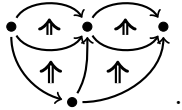
\begin{equation*}
\begin{tikzpicture}[baseline={([yshift=-.5ex]current bounding box.center)},scale=.8]
\begin{scope}
	\node[0c] (0) at (-1.25,0) {};
	\node[0c] (m) at (0,0) {};
	\node[0c] (1) at (1.25,0) {};
	\node[0c] (b) at (-.25,-1.25) {};
	\draw[1c, out=60,in=120] (0) to (m);
	\draw[1c, out=60,in=120] (m) to (1);
	\draw[1c, out=-60,in=-120] (0) to (m);
	\draw[1c, out=-60,in=-120] (m) to (1);
	\draw[1c, out=-90, in=165] (0) to (b);
	\draw[1c, out=0, in=-90] (b) to (1);
	\draw[1c, out=60, in=-90] (b) to (m);
	\draw[2c] (-.625,-.4) to (-.625,.4);
	\draw[2c] (.625,-.4) to (.625,.4);
	\draw[2c] (-.6,-1.2) to (-.6,-.3);
	\draw[2c] (.5,-1.2) to (.5,-.3);
	\node[scale=1.25] at (1.5,-1.25) {.};
\end{scope}
\end{tikzpicture}
\end{equation*}

\noindent In general, one can always compose after regularisation.

Not every filler exhibits a weak composite: it needs to be an \textbf{equivalence}. An $n$\nbd cell $e$ is an $n$\nbd equivalence if any other $n$\nbd cell ``factors through $e$ up to an $(n+1)$\nbd equivalence'' whenever possible. More precisely, $e$ is an equivalence if, whenever the input or output boundary of another $n$\nbd cell $x$ contains the input or output boundary of $e$, $x$ can be factorised as a weak composite of $e$ with another cell $x'$. 

Notice that the definition of $n$\nbd equivalence calls the definition of $(n+1)$\nbd equivalence: this is a coinductive definition. This is another point in which the richer combinatorics of shapes pays off. Our notion of equivalences is similar to the universal cells of the opetopic model \cite{baez1998higher}, where there is also a ``weak uniqueness'' requirement on the factorisation, forced by the asymmetry of opetopic cells: we do not require any uniqueness, but do require \emph{two-sided} factorisation. Compare the following two ways of characterising isomorphisms in a category:
\begin{enumerate}
	\item $f: x \to y$ is an isomorphism if every $g: x \to z$ factors as $f;h$ for a \emph{unique} $h: y \to z$, \emph{or}, dually, if every $k: z \to y$ factors as $\ell;f$ for a \emph{unique} $\ell: z \to x$;
	\item $f: x \to y$ is an isomorphism if every $g: x \to z$ factors as $f;h$ for some $h: y \to z$ \emph{and} every $k: z \to y$ factors as $\ell;f$ for some $\ell: z \to x$.
\end{enumerate}
The definition of universal cells in an opetopic set generalises the first, and the definition of equivalences in a diagrammatic set generalises the second. 

The uniqueness requirement becomes, logically, a universal quantification over higher equivalences, which does not translate into a valid coinductive definition: this is why the opetopic model only covers weak $n$\nbd categories for finite $n$, whereas RDSs are naturally a model of weak $\omega$\nbd categories.

\subsubsection*{Related work}

This work is strongly related to S.\ Henry's \cite{henry2018regular}. The notion of \emph{regularity} for polygraphs was developed independently by Henry and the author (I initially focussed on a stronger notion, which I later relabelled as \emph{constructibility}), and if Conjecture \ref{conj:freegen} is correct, then diagrammatic sets with face but not degeneracy maps are almost certainly equivalent to Henry's regular polygraphs, and the Kan diagrammatic sets of Section \ref{sec:homotopy} are fibrant in his weak model structure on regular polygraphs.

The last part of the article can be read as a companion and complement to Henry's work. The result that Henry achieves with model-category-theoretic methods, a weak Quillen equivalence \cite{henry2018weak} between spaces and regular polygraphs, is stronger than what we achieve here. However, we rely on purely combinatorial constructions, relating the combinatorics of diagrammatic sets to those of simplicial sets, thus furthering the connection between poset topology and the theory of pasting diagrams, which we started exploring in \cite{hadzihasanovic2018combinatorial}. Moreover, our definition of ``weak $\omega$\nbd groupoid'' is a special case of a definition of weak $\omega$\nbd category: finding such a generalisation  was left as an open problem for fibrant regular polygraphs. We have no doubt that a synthesis of the two approaches can be achieved.

The presence of structural degeneracies in diagrammatic sets goes somewhat contrary to Henry's goal of proving C.\ Simpson's conjecture --- that one can strictify associativity and interchange constraints in higher-categorical composition, without loss of generality --- since any such strictification would have to ignore, at least partially, this structure. On the other hand, degeneracies are necessary in the practice of higher-dimensional algebra, where one deals constantly with ``nullary'' inputs or outputs (think of the unitality equations for monoids, or the equational theory of adjunctions). This is the main application of the theory of polygraphs, and our main aim is to develop a practical framework.

The name \emph{diagrammatic set} is borrowed from Kapranov and Voevodsky \cite{kapranov1991infty}, which, in spite of its main result being incorrect, is a major influence of both this work and \cite{henry2018regular}. Kapranov and Voevodsky's diagrammatic sets were based on Johnson's composable pasting schemes \cite{johnson1989combinatorics}, a bad choice for reasons discussed in \cite[Appendix A.2]{henry2019non}, and are \emph{not} the same as our diagrammatic sets, based on regular directed complexes. However, the essential idea is the same, and the name was never used again, so we do not expect any confusion to arise.

Despite a number of works on the combinatorics of higher-categorical pasting in the late 80s and early 90s \cite{power1991pasting,steiner1993algebra}, \cite{kapranov1991infty} seems to be the only attempt to use them as the basis of a model of higher categories. Even so, diagrammatic sets were only used as a stepping stone towards the erroneous proof of the homotopy hypothesis for strict $\omega$\nbd groupoids, which explains, perhaps, why they were not picked up in any subsequent work.

Finally, in its intent of developing a practical, general framework for \emph{presented} higher-dimensional theories, our work is close to the ``quasistrict'' and ``associative'' $n$\nbd categories of Vicary, Dorn, and others \cite{vicary2016globular,dorn2018associative,reutter2019high}, tied to the development of the \texttt{homotopy.io} proof assistant. The technical foundation, however, is very different, and for now we do not know how the models compare.

\subsubsection*{Structure of the article}

We start in Section \ref{sec:basic} by giving the basic definitions of oriented graded posets, molecule (with spherical boundary), and regular directed complex. We show that the latter are closed under lax Gray products and joins. In Section \ref{sec:maps}, we define maps of regular directed complexes and their category, and look at some limits, colimits, and monoidal structures. Section \ref{sec:constructions} is devoted to various combinatorial constructions that are of interest in rewriting, or will be used in the rest of the paper. In Section \ref{sec:presheaves} we define diagrammatic sets, their morphisms, monoidal structures, and sequences of skeleta and coskeleta. 

Section \ref{sec:equivalences} introduces the notion of equivalence cell in a diagrammatic set, which is then used to define representable diagrammatic sets. Section \ref{sec:closure} is devoted to the proof that equivalence cells in an RDS have the expected properties: they include all degenerate cells, and are closed under higher equivalence, weak composition, and (in the appropriate cases) weak division. In Section \ref{sec:morphrep} we present an alternative characterisation of equivalences as weakly invertible cells in an RDS, and use it to prove that all morphisms of RDSs preserve equivalences. In Section \ref{sec:nerves}, we define a realisation functor of diagrammatic sets as $\omega$\nbd categories, together with its right adjoint, the diagrammatic nerve. We prove that, conditional to a conjecture on regular directed complexes, the diagrammatic nerve is full and faithful into the category of RDSs and morphisms that preserve a choice of compositors. In Section \ref{sec:lowdim}, we define $n$\nbd truncated RDSs, our model of weak $n$\nbd category for finite $n$, and prove that 2\nbd truncated RDSs and morphisms are equivalent to bicategories and morphisms that preserve unitors. 

In Section \ref{sec:simplices}, we look at the embedding of simplicial sets as a full subcategory of diagrammatic sets, and introduce various combinatorial constructions relating oriented simplices to more general molecules. In Section \ref{sec:combinatorial}, we introduce Kan diagrammatic sets, show that they are a special case of RDSs, and that they restrict to Kan complexes. We define the combinatorial homotopy groups of a Kan diagrammatic set in the expected way, and construct explicit isomorphisms with the combinatorial homotopy groups of the induced Kan complex. We deduce that the \emph{simplicial} geometric realisation of a Kan diagrammatic set is also naturally compatible with homotopy groups. Finally, in Section \ref{sec:realisation}, we define another geometric realisation of diagrammatic sets, and show that its right adjoint turns a space into a Kan diagrammatic set, and is a homotopical right inverse to the simplicial geometric realisation. We deduce C.\ Simpson's form of the homotopy hypothesis for Kan diagrammatic sets, together with its $n$\nbd truncated versions.

\subsubsection*{Outlook and open problems}

The contents of this article are meant to showcase the strengths of diagrammatic sets as a framework for higher-dimensional rewriting and algebra, which is why we favoured combinatorial proofs based on explicit manipulation of pasting diagrams. As a consequence, some results may be weaker than what we could achieve with a more abstract approach. Most notably, in the light of \cite{henry2018regular}, it should be possible to promote our form of the homotopy hypothesis to a (weak) Quillen equivalence of (weak) model structures with only a modicum of original work. Still better, as mentioned earlier, would be a verification of the Barwick---Schommer-Pries axioms for RDSs.

Another outstanding problem is Conjecture \ref{conj:freegen}, whose proof would remove the conditional from the results of Section \ref{sec:nerves}, and confirm that strict $\omega$\nbd categories embed into our model in the expected way. An alternative would be to recalibrate the combinatorial substrate as discussed in Remark \ref{rmk:ifconjfalse}. 

There is, then, the question of explicitly comparing our model to other models of weak higher categories, not only at the level of their homotopy theory. This is something on which very little is known in general. On the grounds that atoms encompass so many cell shapes, we think that diagrammatic sets may be a good place for comparing ``geometric'' definitions of higher categories. In particular, we conjecture in Section \ref{sec:combinatorial} that the restriction of Kan diagrammatic sets to Kan complexes generalises to a restriction of RDSs to complicial sets; and since our notions of equivalences and representability are inspired by the opetopic or multitopic model, it is plausible that RDSs also restrict to some version of opetopic higher categories. In Section \ref{sec:lowdim}, we also vaguely outline a translation to algebraic models with underlying $\omega$\nbd graphs. As for models based on iterated enrichment or Segal-type conditions, at the moment we do not know how to compare.

The theory of equivalences and representability in diagrammatic sets is made much simpler by the presence of degeneracies. There remains the problem of defining a model of weak $\omega$\nbd categories without degeneracies, generalising fibrant regular polygraphs, to which the strictification argument of \cite{henry2018regular} would apply. In \cite[Appendix B]{hadzihasanovic2018combinatorial}, we proposed a notion of universal cell and representability for constructible polygraphs, which could be extended to regular polygraphs; as discussed there, there are indications from low dimensions that this could give a good model of weak higher categories, but the coinductive arguments required for proving any property of universal cells seem much more complicated. 

Finally, in a certain sense, all this work stems from the problem of formalising the ``smash product of algebraic theories'' from \cite[Section 2.3]{hadzihasanovic2017algebra}, and kickstarting the programme of compositional universal algebra which we outlined there. We believe that diagrammatic sets answer that problem, and plan to return to it in the near future.

\subsubsection*{Acknowledgements}
This work was supported by a JSPS Postdoctoral Research Fellowship and by JSPS KAKENHI Grant Number 17F17810. Thanks to Pierre-Louis Curien, Yves Guiraud, Simon Henry, Alex Kavvos, Chaitanya Leena-Subramaniam, and Jamie Vicary for helpful discussions on aspects of the article.

\section{Regular directed complexes}

\subsection{Basic definitions} \label{sec:basic}

We begin by recapitulating some definitions from \cite{hadzihasanovic2018combinatorial}, and generalising some others. The following are standard notions of order theory and poset topology.

\begin{dfn}
Let $P$ be a finite poset with order relation $\leq$. For all elements $x, y \in P$, we say that $y$ \emph{covers} $x$ if $x < y$ and, for all $y' \in X$, if $x < y' \leq y$, then $y' = y$. 

The \emph{Hasse diagram} of $P$ is the finite directed graph $\hass{P}$ with $\hass{P}_0 := P$ as set of vertices, and $\hass{P}_1 := \{c_{y,x}: y \to x \,|\, y \text{ covers } x\}$ as set of edges. We can reconstruct the partial order on $P$ from its Hasse diagram, letting $x \leq y$ in $P$ if and only if there is a path from $y$ to $x$ in $\hass{P}$.

Let $P_\bot$ be $P$ extended with a least element $\bot$. We say that $P$ is \emph{graded} if, for all $x \in P$, all paths from $x$ to $\bot$ in $\hass{P}_\bot$ have the same length. If $P$ is graded, for each $x \in P$, let $n+1$ be the length of paths from $x$ to $\bot$; then we define $\dmn{x} := n$, the \emph{dimension} of $x$, and let $P^{(n)} := \{x \in P \,|\, \dmn{x} = n\}$ and $\skel{n}{P} := \{x \in P \,|\, \dmn{x} \leq n\}$. We call the latter the \emph{$n$\nbd skeleton} of $P$.

For all $x, y \in P$ such that $x \leq y$, the \emph{interval} $[x,y]$ from $x$ to $y$ is the subset $\{z \in P \,|\, x \leq z \leq y\}$. If $P$ is graded, all paths from $y$ to $x$ in $\hass{P}$ have length $\dmn{y} - \dmn{x}$; this is the \emph{length} of the interval $[x,y]$.
\end{dfn}

In what follows, we assume the usual ``sign rule'' multiplication on $\{+,-\}$. We will often let variables $\alpha, \beta, \ldots$ range implicitly over $\{+,-\}$.

\begin{dfn}
An \emph{oriented graded poset} is a finite graded poset $P$ together with an edge-labelling $o: \hass{P}_1 \to \{+,-\}$ of the Hasse diagram of $P$ (an \emph{orientation}). 

Given a graded poset $P$ with orientation $o$, we extend $o$ to $P_\bot$, by setting $o(c_{x,\bot}) := +$ for all minimal elements $x$ of $P$. An \emph{oriented thin poset} is an oriented graded poset $P$ with the following property: each interval $[x,y]$ of length 2 in $P_\bot$ is of the form
\begin{equation} \label{eq:oriented_thin}
\begin{tikzpicture}[baseline={([yshift=-.5ex]current bounding box.center)}]
	\node[scale=1.25] (0) at (0,2) {$y$};
	\node[scale=1.25] (1) at (-1,1) {$z_1$};
	\node[scale=1.25] (1b) at (1,1) {$z_2$};
	\node[scale=1.25] (2) at (0,0) {$x$};
	\draw[1c] (0) to node[auto,swap] {$\alpha_1$} (1);
	\draw[1c] (0) to node[auto] {$\alpha_2$} (1b);
	\draw[1c] (1) to node[auto,swap] {$\beta_1$} (2);
	\draw[1c] (1b) to node[auto] {$\beta_2$} (2);
\end{tikzpicture}
\end{equation}
in the labelled Hasse diagram $\hass{P}_\bot$, with $\alpha_1\beta_1 = -\alpha_2\beta_2$.
\end{dfn}

The following concerns special subsets of (oriented, graded) posets.

\begin{dfn}
Let $P$ be a poset, and $U \subseteq P$. The \emph{closure} of $U$ is the subset $\clos{U} := \{x \in P \,|\, \exists y \in U \,x \leq y\}$ of $P$. We say that $U$ is \emph{closed} if $U = \clos{U}$. 

Suppose $P$ is graded, and $U \subseteq P$ is closed; $U$ is also graded with the partial order inherited from $P$. We write $\dmn{U} := \mathrm{max}\{\dmn{x} \,|\, x \in U\}$ if $U$ is inhabited, and $\dmn{\emptyset} = -1$; in particular, $\dmn{\clos{\{x\}}} = \dmn{x}$. We say that $U$ is \emph{pure} if all its maximal elements have dimension $n = \dmn{U}$, equivalently, if $U = \clos{(U^{(n)})}$.

Let $P$ be an oriented graded poset, $U \subseteq P$ a closed subset; $U$ inherits an orientation from $P$ by restriction. For $\alpha \in \{+,-\}$ and $n \in \mathbb{N}$, we define
\begin{align*}
	\sbord{n}{\alpha} U & := \{x \in U \,|\, \dmn{x} = n \text{ and, for all $y \in U$, if $y$ covers $x$, then $o(c_{y,x}) = \alpha$} \}, \\
	\bord{n}{\alpha} U & := \clos{(\sbord{n}{\alpha} U)} \cup \{ x \in U \,|\, \text{for all $y \in U$, if $x \leq y$, then $\dmn{y} \leq n$} \}, \\
	\sbord{n}{} U & := \sbord{n}{+}U \cup \sbord{n}{-}U, \quad \qquad \quad \bord{n}{}U := \bord{n}{+}U \cup \bord{n}{-} U.
\end{align*}
If $U$ is $n$\nbd dimensional, we write $\sbord{}{\alpha}U := \sbord{n}{\alpha}U$ and $\bord{}{\alpha}U := \bord{n}{\alpha}U$. We call $\bord{n}{-}U$ the \emph{input $n$\nbd boundary}, and $\bord{n}{+}U$ the \emph{output $n$\nbd boundary} of $U$. For all $x \in P$, we will use the short-hand notation $\sbord{n}{\alpha}x := \sbord{n}{\alpha}\clos{\{x\}}$ and $\bord{n}{\alpha}x := \bord{n}{\alpha}\clos{\{x\}}$.
\end{dfn}

\begin{remark}
In particular, if $U$ is $n$\nbd dimensional, then $\bord{m}{\alpha}U = U$ for all $m \geq n$. If $U$ is also pure, then $\bord{k}{\alpha}U = \clos{(\sbord{k}{\alpha} U)}$ for all $k < n$.
\end{remark}

\begin{dfn}
Let $U_1, U_2 \subseteq P$ be closed subsets of an oriented graded poset. If $U_1 \cap U_2 = \bord{n}{+}U_1 = \bord{n}{-}U_2$, let
\begin{equation*}
	U_1 \cp{n} U_2 := U_1 \cup U_2;
\end{equation*}
this defines partial $n$\nbd composition operations on the closed subsets of $P$, for all $n$. 

Let $P$ be an oriented graded poset. For each $n \in \mathbb{N}$, we define a family $\molec{n}{P}$ of closed subsets of $P$, the \emph{$n$\nbd molecules} of $P$, together with a partial order $\submol$ on each $\molec{n}{P}$, to be read ``is a submolecule of''.

Let $U \subseteq P$ be closed. Then $U \in \molec{n}{P}$ if and only if $\dmn{U} \leq n$, and, inductively on proper subsets of $U$, either
\begin{itemize}
	\item $U$ has a greatest element, in which case we call it an \emph{atom}, or
	\item there exist $n$\nbd molecules $U_1, U_2$ properly contained in $U$, and $k < n$ such that $U_1 \cap U_2 = \bord{k}{+}U_1 = \bord{k}{-}U_2$, and $U = U_1 \cp{k} U_2$.
\end{itemize}
We define $\submol$ to be the smallest partial order relation on $\molec{n}{P}$ such that $U_1, U_2 \submol U$ for all triples $U, U_1, U_2$ in the latter situation.

We say that $P$ itself is an $n$\nbd molecule if $P \in \molec{n}{P}$.
\end{dfn}

The following is a simple property of molecules in an oriented graded poset.
\begin{lem} \label{lem:molecbasic}
Let $U$ be an $n$\nbd dimensional molecule in an oriented graded poset, $n > 0$. Then:
\begin{itemize}
	\item each $x \in \sbord{}{+}U \cap \sbord{}{-}U$ is not covered by any element,
	\item each $x \in \sbord{}{}U \setminus (\sbord{}{+}U \cap \sbord{}{-}U)$ is covered by a single element, and 
	\item each $x \in U^{(n-1)}\setminus \sbord{}{}U$ is covered by exactly two elements with opposite orientations.
\end{itemize}
\end{lem}
\begin{proof}
We proceed by induction on submolecules of $U$. When $U$ is an atom, $\sbord{}{+}U \cap \sbord{}{-}U$ and $x \in U^{(n-1)}\setminus \sbord{}{}U$ are both empty, and obviously any $x \in \sbord{}{}U$ is covered only by the greatest element. 

Suppose $U$ has a proper decomposition $U_1 \cp{k} U_2$. If $k < n-1$, then $U_1 \cap U_2$ contains no $(n-1)$\nbd dimensional elements, so any $(n-1)$\nbd dimensional element of $U$ is covered only by elements of $U_1$, or only by elements of $U_2$, and the statement follows from the inductive hypothesis. If $k = n-1$, the only additional possibility is that an $(n-1)$\nbd dimensional element is covered both by an element of $U_1$ and an element of $U_2$, in which case it belongs to 
\begin{equation*}
	(U_1 \cap U_2)^{(n-1)} = \sbord{n-1}{+}U_1 = \sbord{n-1}{-}U_2,
\end{equation*}
and by the inductive hypothesis it is covered by a single element of $U_1$ with orientation $+$ and by a single element of $U_2$ with orientation $-$. 
\end{proof}

\begin{dfn}
A \emph{directed complex} is an oriented graded poset $P$ such that, for all $x \in P$, with $\dmn{x} = n > 0$, and all $\alpha, \beta$,
\begin{enumerate}
	\item $\bord{}{\alpha}x$ is a molecule, and
	\item $\bord{}{\alpha}(\bord{}{\beta}x) = \bord{n-2}{\alpha}x$.
\end{enumerate}
\end{dfn}

\begin{remark}
Compared to Steiner's definition of directed complex in \cite{steiner1993algebra}, ours has in addition the built-in constraint that $\sbord{}{+}x$ and $\sbord{}{-}x$ are disjoint for all $x$.
\end{remark}

\begin{exm}
For each $n \in \mathbb{N}$, let $O^n$ be the poset with a pair of elements $\fnct{k}^+, \fnct{k}^-$ for each $k < n$ and a greatest element $\fnct{n}$, with the partial order defined by $\fnct{j}^\alpha \leq \fnct{k}^\beta$ if and only if $j \leq k$. This is a graded poset, with $\dmn{\fnct{n}} = n$ and $\dmn{\fnct{k}^\alpha} = k$ for all $k < n$. 

With the orientation $o(c_{y,x}) := \alpha$ if $x = \fnct{k}^\alpha$, and $\alpha \in \{+,-\}$, it becomes a directed complex; in fact, it is the smallest $n$\nbd dimensional directed complex. We call $O^n$ the \emph{$n$\nbd globe}.

There is a category $\cat{O}$ whose objects are the $n$\nbd globes, and for all $n$ and $k < n$ there are exactly two morphisms $\imath^+, \imath^-: O^k \hookrightarrow O^n$, defined by $\imath^\alpha(\fnct{k}) = \fnct{k}^\alpha$, and $\imath^\alpha(\fnct{j}^\beta) = \fnct{j}^\beta$ for all $j < k$. 

The category $\globset$ of presheaves on $\cat{O}$ is the category of $\omega$\nbd graphs \cite[Section 1.4]{leinster2004higher}, also known as globular sets. We represent an $\omega$\nbd graph $X$ as a diagram
\begin{equation*}
\begin{tikzpicture}[baseline={([yshift=-.5ex]current bounding box.center)}]
	\node[scale=1.25] (0) at (0,0) {$X_0$};
	\node[scale=1.25] (1) at (2,0) {$X_1$};
	\node[scale=1.25] (2) at (4,0) {$\ldots$};
	\node[scale=1.25] (3) at (6,0) {$X_n$};
	\node[scale=1.25] (4) at (8,0) {$\ldots$};
	\draw[1c] (1.west |- 0,.15) to node[auto,swap] {$\bord{}{+}$} (0.east |- 0,.15);
	\draw[1c] (1.west |- 0,-.15) to node[auto] {$\bord{}{-}$} (0.east |- 0,-.15);
	\draw[1c] (2.west |- 0,.15) to node[auto,swap] {$\bord{}{+}$} (1.east |- 0,.15);
	\draw[1c] (2.west |- 0,-.15) to node[auto] {$\bord{}{-}$} (1.east |- 0,-.15);
	\draw[1c] (3.west |- 0,.15) to node[auto,swap] {$\bord{}{+}$} (2.east |- 0,.15);
	\draw[1c] (3.west |- 0,-.15) to node[auto] {$\bord{}{-}$} (2.east |- 0,-.15);
	\draw[1c] (4.west |- 0,.15) to node[auto,swap] {$\bord{}{+}$} (3.east |- 0,.15);
	\draw[1c] (4.west |- 0,-.15) to node[auto] {$\bord{}{-}$} (3.east |- 0,-.15);
\end{tikzpicture}
\end{equation*}
of sets and functions, where $X_n := X(O^n)$, and $\bord{}{\alpha}: X_n \to X_{n-1}$ is the image of $\imath^\alpha: O^{n-1} \hookrightarrow O^n$.
\end{exm}

Next, we recall the definition of strict $\omega$\nbd categories.

\begin{dfn}
Let $X$ be an $\omega$\nbd graph, and for all $x \in X_n$ and $k < n$, let
\begin{equation*}
	\bord{k}{\alpha}x = \underbrace{\bord{}{\alpha}(\ldots(\bord{}{\alpha}}_{n-k} x)).
\end{equation*}  
We call the elements $x \in X_n$ the \emph{$n$\nbd cells} of $X$. Given two $n$\nbd cells $x$, $y$ of $X$, and $k < n$, we say that $x$ and $y$ are \emph{$k$\nbd composable}, and write $x \comp{k} y$, if $\bord{k}{+}x = \bord{k}{-}y$.

We write $X_n \comp{k} X_n \subseteq X_n \times X_n$ for the set of pairs of $k$\nbd composable $n$\nbd cells of $X$.
\end{dfn}

\begin{dfn}
A \emph{partial $\omega$\nbd category} is an $\omega$\nbd graph $X$ together with operations
\begin{equation*}
	\idd{}: X_n \to X_{n+1}, \qquad \cp{k}: X_n \comp{k} X_n \pfun X_n,
\end{equation*}
called \emph{unit} and \emph{$k$\nbd composition}, for all $n \in \mathbb{N},$ and $k < n$, where $\idd{}$ is a total function and the $\cp{k}$ are partial functions. For all $k$\nbd cells $x$, and $n > k$, let
\begin{equation*}
	\idd{n}x := \underbrace{(\idd{}\ldots\idd{})}_{n-k} x,
\end{equation*}
an $n$\nbd cell of $X$. The operations are required to satisfy the following conditions:
\begin{enumerate}
	\item for all $n$\nbd cells $x$, and all $k < n$, 
	\begin{align*}
	 	\bord{}{\alpha}(\varepsilon x) & = x, \\
		x \cp{k} \idd{n}(\bord{k}{+}x) & = x = \idd{n}(\bord{k}{-}x) \cp{k} x,
	\end{align*}
	where the two $k$\nbd compositions are always defined;
	\item for all $(n+1)$\nbd cells $x, y$, and all $k < n$, whenever the left-hand side is defined,
	\begin{align*}
		\bord{}{-}(x \cp{n} y) & = \bord{}{-} x, \\
		\bord{}{+}(x \cp{n} y) & = \bord{}{+} y, \\
		\bord{}{\alpha}(x \cp{k} y) & = \bord{}{\alpha} x \cp{k} \bord{}{\alpha} y;
	\end{align*}
	\item for all cells $x, y, x', y'$, and all $n$ and $k < n$, whenever the left-hand side is defined, 
	\begin{align*}
		\idd{}(x \cp{n} y) & = \idd{}x \cp{n} \idd{}y; \\
		(x \cp{n} x') \cp{k} (y \cp{n} y') & = (x \cp{k} y) \cp{n} (x' \cp{k} y');
	\end{align*}
	\item for all cells $x, y, z$, and all $n$, whenever either side is defined,
	\begin{equation*}
		(x \cp{n} y) \cp{n} z = x \cp{n} (y \cp{n} z).
	\end{equation*}
\end{enumerate} 
A partial $\omega$\nbd category is an \emph{$\omega$\nbd category} if the $\cp{k}$ are total functions.

A \emph{functor} of partial $\omega$\nbd categories is a morphism of the underlying $\omega$\nbd graphs that commutes with units and compositions. A functor is an \emph{inclusion} if it is injective on cells of each dimension. Partial $\omega$\nbd categories and functors form a category $\pomegacat$, with a full subcategory $\omegacat$ on $\omega$\nbd categories.

The inclusion of $\omegacat$ into $\pomegacat$ has a left adjoint $(-)^*: \pomegacat \to \omegacat$. Given a partial $\omega$\nbd category $X$, we call $X^*$ the $\omega$\nbd category \emph{generated} by $X$.
\end{dfn}

\begin{prop}\emph{\cite[Proposition 2.9 and 2.13]{steiner1993algebra}}
Let $P$ be a directed complex. 
\begin{enumerate}
	\item The diagram
\begin{equation*}
\begin{tikzpicture}
	\node[scale=1.25] (0) at (0,0) {$\molec{0}{P}$};
	\node[scale=1.25] (1) at (2.5,0) {$\molec{1}{P}$};
	\node[scale=1.25] (2) at (5,0) {$\ldots$};
	\node[scale=1.25] (3) at (7.5,0) {$\molec{n}{P}$};
	\node[scale=1.25] (4) at (10,0) {$\ldots$};
	\draw[1c] (1.west |- 0,.15) to node[auto,swap] {$\bord{}{+}$} (0.east |- 0,.15);
	\draw[1c] (1.west |- 0,-.15) to node[auto] {$\bord{}{-}$} (0.east |- 0,-.15);
	\draw[1c] (2.west |- 0,.15) to node[auto,swap] {$\bord{}{+}$} (1.east |- 0,.15);
	\draw[1c] (2.west |- 0,-.15) to node[auto] {$\bord{}{-}$} (1.east |- 0,-.15);
	\draw[1c] (3.west |- 0,.15) to node[auto,swap] {$\bord{}{+}$} (2.east |- 0,.15);
	\draw[1c] (3.west |- 0,-.15) to node[auto] {$\bord{}{-}$} (2.east |- 0,-.15);
	\draw[1c] (4.west |- 0,.15) to node[auto,swap] {$\bord{}{+}$} (3.east |- 0,.15);
	\draw[1c] (4.west |- 0,-.15) to node[auto] {$\bord{}{-}$} (3.east |- 0,-.15);
\end{tikzpicture}
\end{equation*}
is an $\omega$\nbd graph $\molec{}{P}$. 

\item For any $n$\nbd molecule $U$, let $\idd{}(U) := U$ as an $(n+1)$\nbd molecule, and for any pair $U_1, U_2$ of $n$\nbd molecules, let $U_1 \cp{k} U_2$ be defined if and only if $U_1 \cap U_2 = \bord{k}{+}U_1 = \bord{k}{-}U_2$, and in that case be equal to $U_1 \cup U_2$. With this assignment, $\molec{}{P}$ is a partial $\omega$\nbd category. 

\item The unit $\molec{}{P} \to (\molec{}{P})^*$ is an inclusion of partial $\omega$\nbd categories.
\end{enumerate}
\end{prop}

\begin{remark}
In particular, the \emph{globularity} property $\bord{j}{\alpha}(\bord{k}{\beta}U) = \bord{j}{\alpha}U$ holds for all $n$\nbd molecules $U$, all $\alpha \in \{+,-\}$, and $j < k < n$.
\end{remark}

\begin{lem} \label{lem:composition_form}
Let $U$ be a molecule in a directed complex $P$. Then $U$ is an atom, or
\begin{equation*}
	U = U_1 \cp{k} \ldots \cp{k} U_m,
\end{equation*}
for some molecules $U_1,\ldots,U_m$, where $m > 1$, each $U_i$ has exactly one element of dimension greater than $k$, and at most one of them has an element of dimension greater than $k + 1$.
\end{lem}
\begin{proof}
The proof of \cite[Proposition 4.2]{steiner2004omega}, referring to $\omega$\nbd categories with a set of composition-generators, works also for the partial $\omega$\nbd category $\molec{}{P}$, which is generated by the atoms of $P$.
\end{proof}

\begin{prop} \label{prop:boundary_submol}
Let $U$ be an $n$\nbd dimensional molecule in a directed complex. Then the $n$\nbd dimensional elements of $U$ can be listed as $x_1, \ldots, x_m$ in such a way that $\bord{}{-}x_1 \submol \bord{}{-}U$, $\bord{}{+}x_m \submol \bord{}{+}U$, and, for all $0 < j < m$, if 
\begin{equation} \label{eq:tildeu1k}
	\tilde{U_1} := \bord{}{-}U \cup \bigcup_{i=1}^j \clos\{x_i\}, \quad \quad \tilde{U_2} := \bord{}{+}U \cup \bigcup_{i=j+1}^m \clos\{x_i\},
\end{equation}
then $U = \tilde{U_1} \cp{n-1} \tilde{U_2}$, $\bord{}{+}x_j \submol \bord{}{+}\tilde{U_1}$ and $\bord{}{-}x_{j+1} \submol \bord{}{-}\tilde{U_2}$.
\end{prop}
\begin{proof}
By Lemma \ref{lem:composition_form}, $U$ has an expression of the form $U_1 \cp{n-1} \ldots \cp{n-1} U_m$ where each of the $U_i$ contains exactly one $n$\nbd dimensional element $x_i$. For each $j$, we have that $\tilde{U_1}$ and $\tilde{U_2}$ as in (\ref{eq:tildeu1k}) are equal to $U_1 \cp{n-1} \ldots \cp{n-1} U_j$ and $U_{j+1} \cp{n-1} \ldots \cp{n-1} U_m$, respectively.

Iterating Lemma \ref{lem:composition_form}, we then find that $U_1$ has an expression as a composite of atoms using only $\cp{k}$ for $k < n-1$. By induction on submolecules of $U_1$: if $U_1$ is the atom $\clos\{x_1\}$, then clearly $\bord{}{\alpha}x_1 = \bord{}{\alpha}U_1$. Otherwise, suppose $U_1 = V_1 \cp{k} V_2$ for some $k < n-1$ and, without loss of generality, $\clos\{x_1\} \submol V_1$. By the inductive hypothesis, $\bord{}{\alpha}x_1 \submol \bord{}{\alpha}V_1$, and $\bord{}{\alpha}U_1 = \bord{}{\alpha}V_1 \cp{k} \bord{n-1}{\alpha}V_2$, hence $\bord{}{\alpha}x_1 \submol \bord{}{\alpha}U_1$. 

This proves the statement for $j = 1$; for $j > 1$, it suffices to apply the first part of the proof to the molecule $U_{j} \cp{n-1} \ldots \cp{n-1} U_m$. 
\end{proof}

As presentations of pasting diagrams, molecules in directed complexes are quite unconstrained. In \cite{hadzihasanovic2018combinatorial}, we considered \emph{constructible} molecules as a better-behaved class. The following constraint, based on \cite{henry2018regular}, is weaker than constructibility.

\begin{dfn}
Let $U$ be an $n$\nbd dimensional molecule in a directed complex. We say that $U$ \emph{has spherical boundary} if $n = 0$, or if $n > 0$ and
\begin{enumerate}
	\item $\bord{}{-}U$ and $\bord{}{+}U$ are $(n-1)$\nbd molecules with spherical boundary, and
	\item $\bord{}{-}U \cap \bord{}{+}U = \bord{}{}(\bord{}{+}U) = \bord{}{}(\bord{}{-}U)$. 
\end{enumerate}
A directed complex is \emph{regular} if, for all $x \in P$, the atom $\clos{\{x\}}$ has spherical boundary. 

If $U$ and $V$ are $n$\nbd molecules, $U$ has spherical boundary, and $U \submol V$, we write $U \subsph V$, and say that $U$ is a \emph{spherical submolecule} of $V$.
\end{dfn}

\begin{exm}
Every constructible $n$\nbd molecule in an oriented thin poset is regular and has spherical boundary by \cite[Theorem 3.14]{hadzihasanovic2018combinatorial}. In particular, every constructible directed complex is a regular directed complex. However, the opposite is not true: see [Remark 7.12, \emph{ibid.}] for a counterexample.
\end{exm}

\begin{lem} \label{lem:kbound_spherical}
Let $U$ be an $n$\nbd molecule with spherical boundary, $k < n$. Then $\bord{k}{+}U$ and $\bord{k}{-}U$ are $k$\nbd molecules with spherical boundary.
\end{lem}
\begin{proof}
An easy induction.
\end{proof}

\begin{remark}
In fact, unraveling the induction on $n$ in the definition, we can define an $n$\nbd molecule with spherical boundary to be one in which $\bord{k}{+}U \cap \bord{k}{-}U = \bord{k-1}{}U$ for all $k < n$.
\end{remark}

\begin{lem} \label{lem:basic_reg}
Let $U$ be an $n$\nbd molecule with spherical boundary. Then:
\begin{enumerate}[label=(\alph*)]
	\item $U$ is pure and $n$\nbd dimensional;
	\item $\sbord{}{+}U$ and $\sbord{}{-}U$ are disjoint and inhabited, and each $x \in \sbord{}{}U$ is covered by a single element.
\end{enumerate}
\end{lem}
\begin{proof}
If $n = 0$, this is obvious. Let $n > 0$, and suppose $U$ has a maximal element $x$ of dimension $k < n$; then $x \in \bord{k}{+}U \cap \bord{k}{-}U$. By globularity, $\bord{k}{\alpha}U = \bord{}{\alpha}(\bord{k+1}{\beta}U)$, and by Lemma \ref{lem:kbound_spherical} $\bord{k+1}{\beta}U$ has spherical boundary. Thus
\begin{equation*}
	\bord{k}{+}U \cap \bord{k}{-}U = \bord{}{}(\bord{k}{\alpha}U)
\end{equation*}
which is $(k-1)$\nbd dimensional. This contradicts $x \in \bord{k}{+}U \cap \bord{k}{-}U$, so $U$ is pure.

It follows from purity of $U$ that $\sbord{}{+}U$ and $\sbord{}{-}U$ are disjoint, and we conclude by Lemma \ref{lem:molecbasic}.
\end{proof}

\begin{prop} \label{prop:reg_inhabit}
Let $P$ be a regular directed complex. For all $x \in P$, both $\sbord{}{+}x$ and $\sbord{}{-}x$ are inhabited. 
\end{prop}
\begin{proof}
Follows from Lemma \ref{lem:basic_reg} since $\clos\{x\}$ has spherical boundary for all $x \in P$.
\end{proof}

\begin{prop}
Let $P$ be a regular directed complex. Then $P$ is an oriented thin poset.
\end{prop}
\begin{proof}
Let $[x,y]$ be an interval of length 2 in $P_\bot$. If $x = \bot$, then $\clos\{y\}$ is a 1-globe, so $y$ covers exactly two 0-dimensional elements with opposite orientations, which cover $\bot$ with the same orientation. Otherwise, suppose $x < z < y$, with $z \in \sbord{}{\alpha}y$. Then $x \in \bord{}{\alpha}y$, which is a regular molecule with spherical boundary, so there are two possibilities.
\begin{itemize}
	\item $x \in \bord{}{}(\bord{}{\alpha}y)$, in which case by Lemma \ref{lem:basic_reg} it is only covered by $z$ in $\bord{}{\alpha}y$. Because $\clos\{y\}$ has spherical boundary, also $x \in \bord{}{}(\bord{}{-\alpha}y)$, and $x$ is covered with the same orientation by a single other element $z' \in \sbord{}{-\alpha}y$.
	\item $x \in \bord{}{\alpha}y \setminus \bord{}{}(\bord{}{\alpha}y)$. Then by Lemma \ref{lem:molecbasic} $x$ is covered by a single other element $z' \in \sbord{}{\alpha}y$ with opposite orientation.
\end{itemize}
In both cases, the interval is of the form (\ref{eq:oriented_thin}).
\end{proof}

\begin{cons}
Let $P$ be an oriented graded poset. The \emph{suspension} $\Sigma P$ of $P$ is the oriented graded poset whose elements are $\{\Sigma x \,|\, x \in P \} + \{\bot^-, \bot^+\}$, with the partial order defined by
\begin{itemize}
	\item $\bot^\alpha < \Sigma x$ for all $x \in P$, and
	\item $\Sigma x \leq \Sigma y$ if and only if $x \leq y$ in $P$,
\end{itemize}
and the orientation $o(c_{\Sigma y,\bot^\alpha}) := \alpha$ if $y \in P^{(0)}$, and $o(c_{\Sigma y,\Sigma x}) := o(c_{y,x})$ for all pairs $x, y$ such that $y$ covers $x$ in $P$. If $\dmn{x} = n$ in $P$, then $\dmn{\Sigma x} = n+1$ in $\Sigma P$.

If $U$ is an $n$\nbd molecule, then $\Sigma U$ is an $(n+1)$\nbd molecule, and if $U$ has spherical boundary, then so does $\Sigma U$. We deduce that if $P$ is a (regular) directed complex then $\Sigma P$ is a (regular) directed complex.
\end{cons}

\begin{cons}
Let $P, Q$ be oriented graded posets. The \emph{lax Gray product} $P \tensor Q$ of $P$ and $Q$ is the product poset $P \times Q$, which is graded, with the following orientation: write $x \tensor y$ for an element $(x,y)$ of $P \times Q$, seen as an element of $P \tensor Q$; then, for all $x'$ covered by $x$ in $P$, and all $y'$ covered by $y$ in $Q$, let
\begin{align*}
	o(c_{x \tensor y, x' \tensor y}) & := o_P(c_{x,x'}), \\
	o(c_{x \tensor y, x \tensor y'}) & := (-)^{\dmn{x}}o_Q(c_{y,y'}),
\end{align*}
where $o_P$ and $o_Q$ are the orientations of $P$ and $Q$, respectively. It is straightforward to prove that this operation is associative and has $1 := O^0$ as a unit up to isomorphism.
\end{cons}

In the following statement, $n \lor m$ denotes the greatest and $n \land m$ the least of the natural numbers $n, m$.

\begin{prop}\emph{\cite[Theorem 7.4]{steiner1993algebra}} \label{prop:molec-compose}
Let $U$ be an $n$\nbd molecule and $V$ an $m$\nbd molecule in a directed complex. Then $U \tensor V$ is an $(n+m)$\nbd molecule, and for all $k \in \mathbb{N}$
\begin{equation*}
	\bord{k}{\alpha}(U \tensor V) = \bigcup_{i=(k-m) \lor 0}^{n \land k} \bord{i}{\alpha}U \tensor \bord{k-i}{(-)^i\alpha}V
\end{equation*}
is a $k$\nbd molecule. 
\end{prop}

\begin{prop} \label{prop:grayspher}
Suppose $U$ and $V$ are molecules with spherical boundary. Then $U \tensor V$ has spherical boundary.
\end{prop}
\begin{proof}
Let $n = \dmn{U}, m = \dmn{V}$. By Proposition \ref{prop:molec-compose}, for all $k < n+m$, we have 
\begin{align*}
	& \bord{k}{+}(U \tensor V) \cap \bord{k}{-}(U \tensor V) = \left(\bigcup_i \bord{i}{+}U \tensor \bord{k-i}{(-)^i}V\right) \cap \left(\bigcup_j \bord{j}{-}U \tensor \bord{k-j}{-(-)^j}V\right) = \\
	& = \left(\bigcup_{k-m<i<n} \bord{i-1}{}U \tensor \bord{k-i-1}{}V \right) \cup \left(\bigcup_{i\geq n} U \tensor \bord{k-i-1}{}V\right) \cup \left(\bigcup_{i \leq k-m} \bord{i-1}{}U \tensor V\right) \cup \\
	& \quad \quad \cup \left(\bigcup_{i < j} \bord{i}{+}U \tensor \bord{k-j}{-(-)^j}V \right) \cup \left(\bigcup_{j < i} \bord{j}{-}U \tensor \bord{k-i}{(-)^i}V \right)
\end{align*}
using the fact that
\begin{equation*}
	\bord{i}{+}U \cap \bord{i}{-}U = \bord{i-1}{}U, \quad \quad \bord{j}{+}V \cap \bord{j}{-}V = \bord{j-1}{}V
\end{equation*}
when $i < n$ and $j < m$ because $U$ and $V$ have spherical boundary, and that $\bord{i}{\alpha}U \subseteq \bord{j}{\beta}U$ and $\bord{i}{\alpha}V \subseteq \bord{j}{\beta}V$ if $i < j$. For the same reason, we have 
\begin{equation*}
	\bord{i}{+}U \tensor \bord{k-j}{-(-)^j}V \subseteq \bord{i}{+}U \tensor \bord{k-i-1}{(-)^i}V
\end{equation*}
for all $j > i+1$, and similarly $\bord{j}{-}U \tensor \bord{k-i}{(-)^i}V \subseteq \bord{j}{-}U \tensor \bord{k-j-1}{-(-)^j}V$ for all $i > j+1$. Therefore, the last two terms are equal to
\begin{equation*}
	\left(\bigcup_{i} \bord{i}{+}U \tensor \bord{k-i-1}{(-)^i}V \right) \cup \left(\bigcup_j \bord{j}{-}U \tensor \bord{k-j-1}{-(-)^j}V \right),
\end{equation*}
and it is easy to see that the first three terms are included in this union. Therefore,
\begin{align*}
	\bord{k}{+}(U \tensor V) \cap \bord{k}{-}(U \tensor V) & = \left(\bigcup_{i} \bord{i}{+}U \tensor \bord{k-i-1}{(-)^i}V \right) \cup \left(\bigcup_j \bord{j}{-}U \tensor \bord{k-j-1}{-(-)^j}V \right) = \\
	& = \bord{k-1}{+}(U \tensor V) \cup \bord{k-1}{-}(U \tensor V),
\end{align*}
which proves that $U \tensor V$ has spherical boundary.
\end{proof}

\begin{cor}
Let $P$ and $Q$ be regular directed complexes. Then $P \tensor Q$ is a regular directed complex.
\end{cor}

\begin{cons}
Let $P, Q$ be oriented graded posets. The \emph{join} $P \join Q$ of $P$ and $Q$ is the unique oriented graded poset such that $(P \join Q)_\bot$ is isomorphic to $P_\bot \tensor Q_\bot$. We use the following notation for elements of $P \join Q$: 
\begin{itemize}
	\item for all $x \in P^{(n)}$, let $x \in (P \join Q)^{(n)}$ correspond to $x \tensor \bot$ in $P_\bot \tensor Q_\bot$;
	\item for all $y \in Q^{(m)}$, let $y \in (P \join Q)^{(m)}$ correspond to $\bot \tensor y$ in $P_\bot \tensor Q_\bot$;
	\item for all $x \in P^{(n)}$, $y \in Q^{(m)}$, let $x \join y \in (P \join Q)^{(n+m+1)}$ correspond to $x \tensor y$ in $P_\bot \tensor Q_\bot$.
\end{itemize}
The join is clearly associative and has the empty oriented graded poset $\emptyset$ as unit up to isomorphism.
\end{cons}

\begin{prop}\emph{\cite[Theorem 7.8]{steiner1993algebra}}
Let $U$ be an $n$\nbd molecule and $V$ an $m$\nbd molecule. Then $U \join V$ is an $(n+m+1)$\nbd molecule, and for all $k < n+m+1$ and $\alpha \in\{+,-\}$, $\bord{k}{\alpha}(U \join V)$ is a $k$\nbd molecule.
\end{prop}

The definition of the join using $(-)_\bot$ is simple and makes it obvious that the operation is associative and unital. However, a different definition using suspensions can be useful in proofs about directed complexes, since, unlike $(-)_\bot$, suspensions preserve their class.

\begin{cons}
The join and the suspension are related in the following way. There is an injective function from $P_\bot$ to $\Sigma P$, sending $\bot$ to $\bot^+$ and $x \in P$ to $\Sigma x$. Therefore there is injective function of underlying sets from $P_\bot \tensor Q_\bot$ to $\Sigma P \tensor \Sigma Q$, which through the definition of $P \join Q$ determines an injective function $j: P \join Q \to \Sigma P \tensor \Sigma Q$. 

By Proposition \ref{prop:grayspher}, if $U$ and $V$ are molecules with spherical boundary, $\Sigma U \tensor \Sigma V$ has spherical boundary. It is straightforward to see that this implies the same property for $U \join V = \invrs{j}(\Sigma U \tensor \Sigma V)$. We can thus state the following.
\end{cons}

\begin{prop}
Suppose $U$ and $V$ are molecules with spherical boundary. Then $U \join V$ has spherical boundary.
\end{prop}

\begin{cor}
Let $P$ and $Q$ be regular directed complexes. Then $P \join Q$ is a regular directed complex.
\end{cor}

\subsection{Maps of directed complexes} \label{sec:maps}

In \cite{hadzihasanovic2018combinatorial} we focussed exclusively on a category of oriented graded posets and injective maps, because those correspond to cellular maps of polygraphs. Here, we consider an extension encompassing surjective maps, which will give a choice of degeneracies.

\begin{dfn} \label{dfn:inclclps}
Let $P, Q$ be oriented graded posets. A \emph{map} $f: P \to Q$ of oriented graded posets is a function from $P$ to $Q$ satisfying, for all $x \in P$, $n \in \mathbb{N}$, and $\alpha \in \{+,-\}$,
\begin{equation} \label{dfn:mapdef}
	\bord{n}{\alpha}f(x) = f(\bord{n}{\alpha}x).
\end{equation}
An \emph{inclusion} $\imath: P \incl Q$ is an injective map of oriented graded posets. An inclusion is an \emph{isomorphism} if it is also surjective.

We write $\ogpos$ for the category of oriented graded posets and maps. We write $\rdcpx$ for its full subcategory on regular directed complexes, $\atom$ for its full subcategory on atoms, and $\rdcpxin$ and $\atomin$ for their respective subcategories of inclusions.
\end{dfn}

\begin{lem} \label{lem:closedmap}
Let $f: P \to Q$ be a map of oriented graded posets. Then $f$ is a closed, order-preserving, dimension-non-increasing function of the underlying graded posets.
\end{lem}
\begin{remark}
It follows that there are forgetful functors from $\rdcpx$ and each of its subcategories to $\pos$, the category of posets and order-preserving functions.
\end{remark}
\begin{proof}
Let $x \in P$, and let $m$ be larger than the dimensions of both $x$ and $f(x)$. Then $\clos\{f(x)\} = \bord{m}{\alpha}f(x) = f(\bord{m}{\alpha}x) = f(\clos\{x\})$. This proves that $f$ is both closed and order-preserving, since $y \leq x$ if and only if $y \in \clos\{x\}$.

The dimension of an element $x$ of an oriented graded poset can be characterised as the smallest $n$ such that $\bord{n}{+}x = \bord{n}{-}x = \clos{\{x\}}$. Suppose $x \in P$ is $n$\nbd dimensional; then $\bord{n}{\alpha}f(x) = f(\bord{n}{\alpha}x) = f(\clos{\{x\}}) = \clos{\{f(x)\}}$. It follows that the dimension of $f(x)$ is at most $n$.
\end{proof}

\begin{lem} \label{lem:sameinclusions}
Let $\imath: P \incl Q$ be an inclusion of oriented graded posets. Then $\imath$ is order-reflecting and preserves the covering relation compatibly with the orientations, in the sense that $o_Q(c_{\imath(y),\imath(x)}) = o_P(c_{y,x})$ for all $y, x \in P$ such that $y$ covers $x$.
\end{lem}
\begin{proof}
Let $x, y$ be such that $\imath(x) \leq \imath(y)$. Then $\imath(x) \in \clos\{\imath(y)\} = \imath(\clos\{y\})$, so $\imath(x) = \imath(x')$ for some $x' \leq y$. Since $\imath$ is injective, $x = x'$. It follows that $\imath$ is a closed embedding of graded posets; in particular it preserves the covering relation and dimensions. It follows that $f(\sbord{}{}x) = \sbord{}{}f(x)$ for all $x \in P$, and from $f(\bord{}{\alpha}x) = \bord{}{\alpha}f(x)$ we obtain $f(\sbord{}{\alpha}x) = \sbord{}{\alpha}f(x)$, which is equivalent to $\imath$ being compatible with orientations.
\end{proof}

\begin{remark}
In \cite{hadzihasanovic2018combinatorial} we defined inclusions of oriented graded posets as closed embeddings of posets compatible with the orientations. It follows from Lemma \ref{lem:sameinclusions} that inclusions in the sense of Definition \ref{dfn:inclclps} are also inclusions in the former sense; the converse is straightforward to prove.
\end{remark}

\begin{prop} \label{lem:factor_clpsincl}
Every map $f: P \to Q$ of oriented graded posets factors as a surjective map $P \surj \widehat{P}$ followed by an inclusion $\widehat{P} \incl Q$. This factorisation is unique up to isomorphism.
\end{prop}
\begin{proof}
By Lemma \ref{lem:closedmap}, the image $f(P)$ of $f$ is a closed subset of $Q$. Then $f: P \surj f(P)$ is a surjective map, and the subset inclusion $f(P) \subseteq Q$ is an inclusion of oriented graded posets.
Uniqueness up to isomorphism is a consequence of the uniqueness of the epi-mono factorisation of the underlying function of $f$, together with the fact that a closed embedding into the underlying poset of an oriented graded poset is compatible with a unique orientation on the domain. 
\end{proof}

Next, we specialise to maps of regular directed complexes. The following proves that molecules in a regular directed complex are quite rigid: they have no non-trivial automorphisms.

\begin{lem} \label{lem:molecnoauto}
Let $U$ be a molecule in a regular directed complex, and $\imath: U \incliso U$ an isomorphism. Then $\imath$ is the identity.
\end{lem}
\begin{proof}
We proceed by induction on the dimension and submolecules of $U$. If $n = 0$, this is obvious. Suppose $n > 0$. Then $\imath(\bord{}{\alpha}U) = \bord{}{\alpha}U$, so $\imath$ restricts to an automorphism of $\bord{}{\alpha}U$, a molecule of lower dimension, which by the inductive hypothesis must be the identity. By the same reasoning, if $\imath(x) = x$ for an element $x$, then $\imath$ is also the identity on $\clos\{x\}$, so it suffices to prove that $\imath$ fixes the maximal elements of $U$. If $U$ is an atom, this is obvious.

Otherwise, if $x$ is maximal and $\dmn{x} < n$, then $x \in \bord{}{}U$, so we have already established that $\imath(x) = x$. Suppose $\dmn{x} = n$. By Lemma \ref{lem:molecbasic} and Proposition \ref{prop:reg_inhabit}, we can construct a sequence $x = x_0 \to y_0 \to x_1 \to \ldots \to x_m \to y_m$ of elements of $U$, where the $y_i$ are $(n-1)$\nbd dimensional, the $x_i$ are $n$\nbd dimensional, $y_i \in \sbord{}{+}x_{i} \cap \sbord{}{-}x_{i+1}$ for $i < m$, and $y_m \in \sbord{}{+}x_m \cap \sbord{}{+}U$: at every $x_i$, we can always pick $y_{i} \in \sbord{}{+}x_i$, and if $y_i \in \sbord{}{+}U$ we stop, otherwise we take $x_{i+1}$ such that $y_i \in \sbord{}{-}x_{i+1}$.

Any such sequence is mapped by $\imath$ to a sequence with the same property, that is, $\imath(x) \to \imath(y_0) \to \ldots \to \imath(x_m) \to \imath(y_m) = y_m$. By Lemma \ref{lem:molecbasic}, $y_m$ is only covered by $x_m$ in $U$, so $\imath(x_m) = x_m$, hence also $\imath(y_{m-1}) = y_{m-1}$. Then $y_{m-1}$ is only covered by $x_{m-1}$ with orientation $+$, and proceeding backwards we find that $\imath(x) = x$.
\end{proof}

Let $\molecin$ be the full subcategory of $\rdcpxin$ on molecules of any dimension, and, for each regular directed complex $P$, let $\slice{\molecin}{P}$ be the comma category of inclusions of molecules into $P$, that is, the category whose objects are inclusions $U \hookrightarrow P$ of molecules into $P$, and morphisms are commutative triangles
\begin{equation*}
\begin{tikzpicture}[baseline={([yshift=-.5ex]current bounding box.center)}]
	\node[scale=1.25] (0) at (-1.25,1.25) {$U$};
	\node[scale=1.25] (1) at (0,0) {$P$};
	\node[scale=1.25] (2) at (1.25,1.25) {$V$};
	\draw[1cinc] (0) to (2);
	\draw[1cincl] (0) to (1);
	\draw[1cinc] (2) to (1);
\end{tikzpicture}
\end{equation*}
of inclusions.

\begin{prop} \label{prop:preordermolec}
Let $P$ be a regular directed complex. The category $\slice{\molecin}{P}$ is a preorder.
\end{prop}
\begin{proof}
By Lemma \ref{lem:molecnoauto}, two inclusions $\imath_1, \imath_2: U \incl P$ are equal if and only if they have the same image in $P$. This establishes an equivalence between $\slice{\molecin}{P}$ and a subposet of the subset lattice of $P$.
\end{proof}

Let us look at some basic limits and colimits in $\rdcpx$.

\begin{prop}
The directed complex $1$ with a single element is the terminal object of $\rdcpx$.
\end{prop}
\begin{proof}
Let $P$ be a regular directed complex. There is a unique function from $P$ to $1$, and it is trivially a map.
\end{proof}

\begin{prop} \label{prop:pushouts_exist}
The category $\rdcpxin$ has an initial object and pushouts, created by the forgetful functor to $\cat{Set}$, and preserved by the inclusion of subcategories $\rdcpxin \hookrightarrow \rdcpx$.
\end{prop}

\begin{remark}
It follows that both $\rdcpxin$ and $\rdcpx$ also have all finite coproducts. 
\end{remark}
\begin{proof}
The empty directed complex $\emptyset$ is clearly initial in both categories.

Let $\imath_1: Q \incl P_1, \imath_2: Q \incl P_2$ be a span of inclusions. We let $P_1 \cup_Q P_2$ be the pushout of the underlying span of sets, that is, the quotient of the disjoint union $P_1 + P_2$ of sets by the relation $\imath_1(x) \sim \imath_2(x)$ for all $x \in Q$. This comes with injective functions $j_1: P_1 \incl P_1 \cup_Q P_2$ and $j_2: P_2 \incl P_1 \cup_Q P_2$, and becomes a poset by letting $j_i(x) \leq j_i(y)$ if and only if $x \leq y$ in $P_i$. Now $\clos\{j_i(x)\} \simeq \clos\{x\}$ for all $x \in J_i$, and every element of $P_1 \cup_Q P_2$ is of the form $j_i(x)$ for some $i$ and $x \in J_i$, so $P_1 \cup_Q P_2$ is graded, and inherits an orientation from $P_1$ and $P_2$, compatibly on $Q$. 

For the same reason, with this orientation $P_1 \cup_Q P_2$ is a regular directed complex. The universal property is a simple check.
\end{proof}

\begin{cor} \label{cor:globpos_is_colimit}
Any regular directed complex is the colimit of the diagram of inclusions of its atoms.
\end{cor}
\begin{proof}
This is true of the underlying sets and functions, and the colimit can be constructed by pushouts and finite coproducts.
\end{proof}

\begin{exm} \label{exm:pushoutcomposition}
Let $U, V$ and $W$ be regular molecules together with necessarily unique isomorphisms $W \incliso \bord{k}{+}U$ and $W \incliso \bord{k}{-}V$. We can take the pushout
\begin{equation*}
\begin{tikzpicture}[baseline={([yshift=-.5ex]current bounding box.center)}]
	\node[scale=1.25] (0) at (0,1.5) {$W$};
	\node[scale=1.25] (1) at (2.5,0) {$U\cp{k}V$};
	\node[scale=1.25] (2) at (0,0) {$U$};
	\node[scale=1.25] (3) at (2.5,1.5) {$V$};
	\draw[1cinc] (0) to (3);
	\draw[1cincl] (0) to (2);
	\draw[1cinc] (2) to node[auto,swap] {$j_1$} (1);
	\draw[1cincl] (3) to node[auto] {$j_2$} (1);
	\draw[edge] (1.6,0.2) to (1.6,0.7) to (2.3,0.7);
\end{tikzpicture}
\end{equation*}
in $\rdcpxin$; then $U \cp{k} V$ is a molecule, equal to the $k$\nbd composite $j_1(U) \cp{k} j_2(V)$. 

Conversely, if $P$ is a molecule, decomposing as $U \cp{k} V$, then $P$ is the pushout of the span of inclusions $(U \cap V \subseteq U$, $U \cap V \subseteq V)$. In this sense, the two interpretations of $\cp{k}$ --- a decomposition of subsets of a regular directed complex, and an operation composing different regular directed complexes --- are compatible with each other. 

Moreover, thanks to Proposition \ref{prop:preordermolec}, if an equation between expressions built from the $\cp{k}$ and the $\bord{k}{\alpha}$ holds for molecules in all regular directed complexes, then it holds \emph{up to unique isomorphism} for the same expressions seen as operations on regular directed complexes. 

For example, for all $n$\nbd molecules $U_1, U_2 \subseteq P$ and all $k < n-1$, it holds that $\bord{n-1}{\alpha}(U_1 \cp{k} U_2) = \bord{n-1}{\alpha}U_1 \cp{k} \bord{n-1}{\alpha}U_2$ when $U_1 \cap U_2 = \bord{k}{+}U_1 = \bord{k}{-}U_2$. It follows that for all $n$\nbd molecules $U$ and $V$ with $\bord{k}{+}U$ isomorphic to $\bord{k}{-}V$, there is a unique isomorphism $\bord{n-1}{\alpha}(U \cp{k} V) \incliso \bord{n-1}{\alpha}U \cp{k} \bord{n-1}{\alpha}V$.
\end{exm}

We will now prove that lax Gray products and joins are compatible with maps of directed complexes, so in particular they define monoidal structures on $\rdcpx$.

\begin{prop}
Let $f: P \to P'$ and $g: Q \to Q'$ be two maps of directed complexes. Then there is a map $f \tensor g: P \tensor Q \to P' \tensor Q'$ of directed complexes whose underlying function is $f \times g$.
\end{prop}
\begin{proof}
Let $x \in P^{(n)}$ and $y \in Q^{(m)}$. By Proposition \ref{prop:molec-compose} and the fact that $f, g$ are maps of oriented graded posets, we have
\begin{align*}
	\bord{k}{\alpha}(f \tensor g(x \tensor y)) & = \bord{k}{\alpha}(f(x) \tensor g(y)) = \bigcup_i \bord{i}{\alpha}f(x) \tensor \bord{k-i}{(-)^i\alpha}g(y) = \\
	& = \bigcup_i f(\bord{i}{\alpha}x) \tensor g(\bord{k-i}{(-)^i\alpha}y) = \\
	& = f \tensor g\left(\bigcup_i \bord{i}{\alpha}x \tensor \bord{k-i}{(-)^i\alpha}y\right) = f \tensor g(\bord{k}{\alpha}(x\tensor y)).
\end{align*}
This proves that $f \tensor g$ is a map of oriented graded posets.
\end{proof}

\begin{cor}
The lax Gray product determines a monoidal structure on $\rdcpx$, whose unit is the terminal object $1$, restricting to a monoidal structure on $\atom$.

The forgetful functor is a monoidal functor $(\rdcpx, -\tensor-, 1) \to (\pos, -\times-, 1)$. 
\end{cor}

\begin{prop} \label{prop:laxgray_preserve_colim}
The lax Gray product preserves the initial object and pushouts of inclusions in each variable.
\end{prop}
\begin{proof}
As shown in Proposition \ref{prop:pushouts_exist}, these colimits are created by the forgetful functor to $\cat{Set}$, and the statement is true for the cartesian product of sets.
\end{proof}

Next, we move on to joins of regular directed complexes. To prove functoriality, we will use the functoriality of suspensions.

\begin{cons}
Let $f: P \to Q$ be a map of oriented graded posets. Then $\Sigma f: \Sigma P \to \Sigma Q$, defined by $\bot^\alpha \mapsto \bot^\alpha$ and $\Sigma x \mapsto \Sigma f(x)$ is also a map of oriented graded posets: for all $x \in P$, by construction we have $\bord{k}{\alpha}(\Sigma x) = \Sigma(\bord{k-1}{\alpha}x)$ for all $k > 0$, and $\bord{0}{\alpha}(\Sigma x) = \{\bot^\alpha\}$. Hence, 
\begin{equation*}
	\Sigma f(\bord{k}{\alpha}(\Sigma x)) = \Sigma(f(\bord{k-1}{\alpha}x)) = \Sigma(\bord{k-1}{\alpha}f(x)) = \bord{k}{\alpha}\Sigma f(\Sigma x)
\end{equation*}
when $k > 0$, and the few remaining cases can be easily checked. 

This assignment respects composition and identities, so it defines an endofunctor $\Sigma$ on $\ogpos$, which restricts to an endofunctor on $\rdcpx$.
\end{cons}

\begin{cons}
If $f: P \to P'$ and $g: Q \to Q'$ are two maps of directed complexes, then $\Sigma f \tensor \Sigma g$ sends the image of $j: P \join Q \to \Sigma P \tensor \Sigma Q$ to the image of $j: P' \join Q' \to \Sigma P' \tensor \Sigma Q'$. By injectivity, $j$ has a partial inverse $\invrs{j}$ on its image, so it makes sense to define
\begin{equation} \label{eq:join_maps}
	f \join g := j;(\Sigma f \tensor \Sigma g); \invrs{j} : P \join Q \to P' \join Q'.
\end{equation}
\end{cons}

\begin{prop}
The join determines a monoidal structure on $\rdcpx$, whose unit is the empty directed complex $\emptyset$, restricting to a monoidal structure on $\atom$.
\end{prop}
\begin{proof}
Functoriality follows from the functoriality of $\Sigma(-) \tensor \Sigma(-)$, so it suffices to prove that $f \join g$ is a map of oriented graded posets. For all $z \in P \join Q$, we have 
\begin{equation*}
	\invrs{j}(\bord{k+1}{\alpha}j(z)) = \bord{k}{\alpha}z, \quad \quad \quad	\clos j(\bord{k}{\alpha}z) = \bord{k+1}{\alpha}j(z). 
\end{equation*}
Now, say that $W \subseteq \Sigma P \tensor \Sigma Q$ is \emph{closed relative to $j$} if $w \in W$ and $j(z) \leq w$ for some $z \in P \join Q$ implies that $j(z) \in W$. Then:
\begin{enumerate}
	\item for all closed $U \subseteq P \join Q$, the subset $j(U)$ is closed relative to $j$,
	\item if $W$ is closed relative to $j$, then $\invrs{j}(W) = \invrs{j}(\clos W)$, and
	\item $\Sigma f \tensor \Sigma g$ preserves the property of closure relative to $j$.
\end{enumerate}
Because for all $k$ and $z \in P \join Q$, the subset $j(\bord{k}{\alpha}z)$ is closed relative to $j$, we have
\begin{equation*}
	\invrs{j}(\Sigma f \tensor \Sigma g(\clos j(\bord{k}{\alpha}z))) = \invrs{j}(\Sigma f \tensor \Sigma g(j(\bord{k}{\alpha}z))) = f\join g (\bord{k}{\alpha}z)
\end{equation*}
and
\begin{align*}
	\invrs{j}(\Sigma f \tensor \Sigma g(\clos j(\bord{k}{\alpha}z))) & = \invrs{j}(\Sigma f \tensor \Sigma g(\bord{k+1}{\alpha}j(z))) \\ 
	& = \invrs{j}(\bord{k+1}{\alpha}(\Sigma f \tensor \Sigma g(j(z))) = \bord{k}{\alpha}(f\join g)(z).
\end{align*}
So $f \join g$ is compatible with boundaries: it is a map of oriented graded posets.
\end{proof}

\begin{prop} \label{lem:join_pushout_incl}
The join preserves pushouts of inclusions in each variable.
\end{prop}
\begin{proof}
The $\Sigma$ endofunctor preserves pushout diagrams of inclusions. The statement then follows from Proposition \ref{prop:laxgray_preserve_colim} and the definition of joins.
\end{proof}

\begin{remark}
Notice, however, that $\Sigma$ does not preserve the initial object, nor any coproducts, and neither does the join operation.
\end{remark}

\begin{remark}
By definition, there are injective functions $P \incl P \join Q$ and $Q \incl P \join Q$ for all oriented graded posets $P, Q$. These are in fact inclusions of oriented graded posets, natural in $P$ and $Q$.
\end{remark}

To conclude this section, we rapidly treat duals of directed complexes and their interaction with lax Gray products and joins.

\begin{cons}
Let $P$ be an oriented graded poset, and $J \subseteq \mathbb{N}\setminus\{0\}$. Then $\oppn{J}{P}$, the \emph{$J$\nbd dual} of $P$, is the oriented graded poset with the same underlying poset as $P$, and the orientation $o'$ defined by 
\begin{equation*}
o'(c_{x,y}) := \begin{cases}
		-o(c_{x,y}), & \dmn{x} \in J, \\
		o(c_{x,y}), & \dmn{x} \not\in J,
	\end{cases}
\end{equation*}
for all elements $x, y \in P$ such that $x$ covers $y$. 

We write $\opp{P}$, $\coo{P}$, and $\oppall{P}$ for $\oppn{J}{P}$ in the cases, respectively, $J = \{2n-1\}_{n>0}$, $J = \{2n\}_{n>0}$, and $J = \mathbb{N} \setminus \{0\}$.
\end{cons}

\begin{prop}
Let $P$ be a (regular) directed complex, $J \subseteq \mathbb{N}^+$. Then $\oppn{J}{P}$ is a (regular) directed complex. The assignment $P \mapsto \oppn{J}{P}$ extends to an involutive endofunctor of $\rdcpx$.
\end{prop}
\begin{proof}
If $U$ is an $n$\nbd molecule, then $\oppn{J}{U}$ is an $n$\nbd molecule, by a straightforward induction on the structure of $U$. Moreover $\oppn{J}{U}$ has spherical boundary if and only if $U$ does. It follows that $\oppn{J}{P}$ is a (regular) directed complex if and only if $P$ is. Functoriality on $\rdcpx$ is an easy check (a map $f$ has the same underlying function as $\oppn{J}{f}$).
\end{proof}

The following is a simple fact about atoms and duals.

\begin{lem} \label{lem:reverse_surj}
Let $U$ and $V$ be atoms, $n = \dmn{U} > \dmn{V}$, and let $p: U \surj V$ be a surjective map. Then there is a map $p': \oppn{n}{U} \surj V$ such that $\restr{p}{\bord{}{\alpha}U} = \restr{p'}{\bord{}{-\alpha}\oppn{n}{U}}$ for $\alpha \in \{+,-\}$.
\end{lem}
\begin{proof}
The statement defines $p'$ on $\bord{}{}\oppn{n}{U}$; it suffices to extend it by sending the greatest element of $\oppn{n}{U}$ to the greatest element of $V$. This is well-defined because $p(\bord{}{+}U) = p(\bord{}{-}U) = V$.
\end{proof}

\begin{prop}
Let $P, Q$ be two oriented graded posets. Then:
\begin{enumerate}[label=(\alph*)]
	\item $x \tensor y \mapsto y \tensor x$ defines an isomorphism between $\opp{(P \tensor Q)}$ and $\opp{Q} \tensor \opp{P}$ and between $\coo{(P \tensor Q)}$ and $\coo{Q} \tensor \coo{P}$;
	\item $x \join y \mapsto y \join x$ defines an isomorphism between $\opp{(P \join Q)}$ and $\opp{Q} \join \opp{P}$.
\end{enumerate}
Consequently, $x \tensor y \mapsto x \tensor y$ defines an isomorphism between $\oppall{(P \tensor Q)}$ and $\oppall{P} \tensor \oppall{Q}$. These isomorphisms are natural for maps of oriented graded posets.
\end{prop}
\begin{proof}
Identical to the proof of \cite[Proposition 4.23]{hadzihasanovic2018combinatorial}.
\end{proof}

\subsection{Some constructions of molecules} \label{sec:constructions}

In this section, we give a few simple constructions of molecules and maps that we will use in the rest of the article. 

\begin{cons} \label{cons:substitution}
Let $U, V, W$ be regular $n$\nbd molecules such that 
\begin{enumerate}
	\item $V, W$ have spherical boundary, $V \subsph U$, and
	\item there are (necessarily unique) isomorphisms $\bord{}{-}W \incliso \bord{}{-}V$ and $\bord{}{+}W \incliso \bord{}{+}V$ of $(n-1)$\nbd molecules.
\end{enumerate}
The isomorphisms restrict to isomorphisms of the lower-dimensional boundaries, and because $V$ and $W$ have spherical boundary, this suffices to induce an isomorphism $\imath: \bord{}{}W \incliso \bord{}{}V$. 

The \emph{substitution} $U[W/V]$ of $W$ for $V \subsph U$ is the result of replacing $V$ with $W$ in $U$, identifying $\bord{}{}V$ and $\bord{}{}W$ through $\imath$: that is, $U[W/V]$ is the set $(U\setminus V)+ W$, with $x \leq y$ if and only if
\begin{itemize}
	\item $x, y \in U\setminus V$ and $x \leq y$ in $U$, or $x, y \in W$ and $x \leq y$ in $W$, or
	\item $x \in U\setminus V$, $y \in W$, and for some $z \in \bord{}{}W$ we have $x \leq \imath(z)$ in $U$ and $z \leq y$ in $W$, or
	\item $x \in W$, $y \in U\setminus V$, and for some $z \in \bord{}{}W$ we have $x \leq z$ in $W$ and $\imath(z) \leq y$ in $U$.
\end{itemize}
The poset $U[W/V]$ is still a pure graded poset, and inherits an orientation from those of $U$ and $W$. Notice that $\bord{}{\alpha}U[W/V]$ is isomorphic to $\bord{}{\alpha}U$.
\end{cons}

\begin{prop} \label{prop:substitution}
$U[W/V]$ is a regular $n$\nbd molecule, $W \subsph U[W/V]$. If $V \subsph V' \submol U$ for another $n$\nbd molecule $V'$, then $W \subsph V'[W/V] \submol U[W/V]$. Moreover, if $U$ has spherical boundary, so does $U[W/V]$.
\end{prop}
\begin{proof}
By induction on increasing $V'$ with $V \subsph V' \submol U$: if $V' = V$, then $V'[W/V] = W$, which is a regular $n$\nbd molecule by assumption. 

Otherwise, $V'$ has a proper decomposition $V'_1 \cp{k} V'_2$, with $V \subsph V'_i$; without loss of generality, let $i = 1$. Then $V'_1[W/V]$ is a regular $n$\nbd molecule with boundaries isomorphic to those of $V'_1$, so $V'_1[W/V] \cp{k} V'_2$ is well-defined, a regular molecule, and equal to $V'[W/V]$. Moreover $W \subsph V'_1[W/V] \submol V'[W/V]$. We conclude using the fact that chains $V \sqsubset \ldots \sqsubset U$ are finite.

Finally, because substitutions do not affect boundaries, if $U$ has spherical boundary, so does $U[W/V]$. 
\end{proof}

The following construction formalises the fact that, given two regular $n$\nbd molecules $U, V$ with isomorphic spherical boundaries, we can form an $(n+1)$\nbd atom with $U$ as its input boundary, and $V$ as its output boundary. 

\begin{cons}
Let $U$ and $V$ be regular $n$\nbd molecules with spherical boundary and (necessarily unique) isomorphisms $\bord{}{-}U \incliso \bord{}{-}V$ and $\bord{}{+}U \incliso \bord{}{+}V$.

Form the pushout of $\bord{}{}U \subseteq U$ and $\bord{}{}U \incliso \bord{}{}V \subseteq V$ in $\rdcpx$; then, let $U \celto V$ be the oriented graded poset obtained by adjoining a single $(n+1)$\nbd dimensional element $\top$ with $\bord{}{-}\top := U$ and $\bord{}{+}\top := V$. Then $U \celto V$ is an $(n+1)$\nbd atom with spherical boundary.
\end{cons}

In particular, for all regular $n$\nbd molecules $U$ with spherical boundary, $\bord{}{-}U$ and $\bord{}{+}U$ have isomorphic boundaries; thus $\bord{}{-}U \celto \bord{}{+}U$ is well-defined, and has boundaries isomorphic to those of $U$.

\begin{dfn} 
For all regular molecules $U$ with spherical boundary, we write $\compos{U}$ for the atom $\bord{}{-}U \celto \bord{}{+}U$.
\end{dfn}

\begin{cons} \label{cons:ou}
Let $U$ be a molecule in a directed complex. We define $\infl{U}$ to be the poset $(O^1 \tensor U)/\sim$, where $\sim$ is the equivalence relation generated by
\begin{equation*}
	\fnct{0}^- \tensor x \sim \fnct{1} \tensor x \sim \fnct{0}^+ \tensor x \text{ for all } x \in \bord{}{}U.
\end{equation*}
Because $\fnct{0}^-\tensor y$ covers $\fnct{0}^-\tensor x$ with orientation $\alpha$ if and only if $\fnct{0}^+ \tensor y$ covers $\fnct{0}^+ \tensor x$ with orientation $\alpha$, $\infl{U}$ inherits an orientation from $O^1 \tensor U$; this is the only orientation making the quotient $q: O^1 \tensor U \surj \infl{U}$ a map of oriented graded posets.

Intuitively, $\infl{U}$ is a cylinder on $U$ with its sides squashed. The natural map 
\begin{equation*}
	! \tensor \idcat{U}: O^1 \tensor U \surj 1 \tensor U \incliso U,
\end{equation*}
where $!: O^1 \surj 1$ is the unique map onto the terminal object, descends to the quotient, factoring as
\begin{equation*}
\begin{tikzpicture}[baseline={([yshift=-.5ex]current bounding box.center)}]
	\node[scale=1.25] (0) at (-1.5,1.25) {$O^1 \tensor U$};
	\node[scale=1.25] (1) at (0,0) {$\infl{U}$};
	\node[scale=1.25] (2) at (1.5,1.25) {$U$};
	\draw[1csurj] (0) to node[auto] {$! \tensor \idcat{U}$} (2);
	\draw[1csurj] (0) to node[auto,swap] {$q$} (1);
	\draw[1csurj] (1) to node[auto,swap] {$p_U$} (2);
\end{tikzpicture}
\end{equation*}
for a unique $p_U: \infl{U} \surj U$. This has the property that
\begin{equation*}
\begin{tikzpicture}[baseline={([yshift=-.5ex]current bounding box.center)}]
	\node[scale=1.25] (0) at (-1.25,1.25) {$U$};
	\node[scale=1.25] (1) at (0,0) {$O(U)$};
	\node[scale=1.25] (2) at (1.25,1.25) {$U$};
	\draw[1c] (0) to node[auto] {$\idcat{U}$} (2);
	\draw[1cinc] (0) to node[auto,swap] {$\imath^{\alpha}$} (1);
	\draw[1csurj] (1) to node[auto,swap] {$p_U$} (2);
\end{tikzpicture}
\end{equation*}
commutes for $\alpha \in \{+,-\}$, where $\imath^\alpha$ is the isomorphic inclusion of $U$ into $\bord{}{\alpha}O(U)$.
\end{cons}

We want to prove that, if $U$ is a regular molecule with spherical boundary, then so is $O(U)$. It is convenient to prove the following, more general result. Let $C \subseteq \bord{}{}U$ be a closed subset, and let $U_C$ be the quotient of $O^1 \tensor U$ by the equivalence relation generated by
\begin{equation*}
	\fnct{0}^- \tensor x \sim \fnct{1} \tensor x \sim \fnct{0}^+ \tensor x \text{ for all } x \in C.
\end{equation*}
Then $U_C$ admits a unique orientation making the quotient $q_C: O^1 \tensor U \surj U_C$ a map of oriented graded posets. When $C = \bord{}{}U$, we recover $U_C = O(U)$. 

\begin{lem} \label{lem:uc_regular}
Let $U$ be an $n$\nbd molecule in a directed complex and $C \subseteq \bord{}{}U$ a closed subset. Then $U_C$ is an $(n+1)$\nbd molecule. If $U$ has spherical boundary, so does $U_C$, and if $U$ is regular, so is $U_C$.
\end{lem}
\begin{proof}
We know that $O^1 \tensor U$ is an $(n+1)$\nbd molecule; using the fact that $q_C$ is compatible with boundaries, we can obtain an expression of $U_C$ as a composite of atoms by induction on expressions of $O^1 \tensor U$ as a composite of atoms.

For $k \leq n$, we have
\begin{equation*}
	\bord{k}{\alpha}(O^1 \tensor U) = \{\fnct{0}^\alpha\} \tensor \bord{k}{\alpha}U \cup O^1 \tensor \bord{k-1}{-\alpha}U,
\end{equation*}
and it is easy to see that $\bord{k}{\alpha}U_C = q_C(\bord{k}{\alpha}(O^1 \tensor U))$. Then $\bord{k}{+}U_C \cap \bord{k}{-}U_C$ is equal to
\begin{align*}
	& \left(q_C(\{\fnct{0}^+\} \tensor \bord{k}{+}U) \cap q_C(\{\fnct{0}^-\} \tensor \bord{k}{-}U)\right) \cup \left(q_C(\{\fnct{0}^+\} \tensor \bord{k}{+}U) \cap q_C(O^1 \tensor \bord{k-1}{+}U)\right) \cup \\
	& \cup \left(q_C(O^1 \tensor \bord{k-1}{-}U) \cap q_C(\{\fnct{0}^-\} \tensor \bord{k}{-}U)\right) \cup \left(q_C(O^1 \tensor \bord{k-1}{-}U) \cap q_C(O^1 \tensor \bord{k-1}{+}U)\right),
\end{align*}
which is equal to
\begin{align*}
	& q_C(O^1 \tensor (C \cap \bord{k}{+}U \cap \bord{k}{-}U)) \cup q_C(\{\fnct{0}^+\} \tensor \bord{k-1}{+}U) \; \cup \\
	& \cup q_C(\{\fnct{0}^-\} \tensor \bord{k-1}{-}U) \cup q_C(O^1 \tensor (\bord{k-1}{+}U \cap \bord{k-1}{-}U)).
\end{align*}
If $U$ has spherical boundary, the first term is included in the union of the following two by definition of $C$, and because $\bord{k}{+}U \cap \bord{k}{-}U = \bord{k-1}{}U$ when $k < n$, and $C \subseteq \bord{}{}U$ when $k = n$; while the last term is equal to $q_C(O^1 \tensor \bord{k-2}{}U)$. It follows that 
\begin{equation*}
	\bord{k}{+}U_C \cap \bord{k}{-}U_C = \bord{k-1}{+}U_C \cup \bord{k-1}{-}U_C,
\end{equation*}
and $U_C$ has spherical boundary.

Finally, suppose that $U$ is regular. Elements $x \in U_C$ are either the image of some $\fnct{0}^\alpha \tensor x'$ for $x' \in U$, in which case $\clos\{x\}$ in $U_C$ is isomorphic to $\clos\{x'\}$ in $U$, or they are the image of $\fnct{1} \tensor x'$ for some $x' \in U \setminus C$. But then $\clos\{x\}$ is isomorphic to $\clos\{x'\}_D$, where $D := C \cap \bord{}{}x'$, which is a molecule with spherical boundary by what we proved earlier.
\end{proof}

\begin{prop} \label{prop:ou_regular}
Let $U$ be a regular $n$\nbd molecule with spherical boundary. Then $\infl{U}$ is a regular $(n+1)$\nbd molecule with spherical boundary.
\end{prop}
\begin{proof}
Follows from Lemma \ref{lem:uc_regular} with $C := \bord{}{}U$.
\end{proof}

\begin{remark}
If $U$ is an $n$\nbd atom, then $O(U)$ is an $(n+1)$\nbd atom isomorphic to $U \celto U$. Letting $O^0(U) := U$ and $O^n(U) := O(O^{n-1}(U))$ for $n > 0$, in particular, $O^{n-1}(1)$ is isomorphic to $O^n$.
\end{remark}

\begin{remark}
While the construction $O(-)$ is not functorial in general, it satisfies a restricted naturality over \emph{surjective} maps of $n$\nbd atoms of the same dimension. For a surjective map $f: U \to V$ of $n$\nbd atoms, since $f(\bord{}{\alpha}U) = \bord{}{\alpha}V$, the commutative square
\begin{equation*}
\begin{tikzpicture}[baseline={([yshift=-.5ex]current bounding box.center)}]
	\node[scale=1.25] (0) at (0,1.5) {$O^1 \tensor U$};
	\node[scale=1.25] (1) at (3,0) {$V$};
	\node[scale=1.25] (2) at (0,0) {$U$};
	\node[scale=1.25] (3) at (3,1.5) {$O^1 \tensor V$};
	\draw[1csurj] (0) to node[auto] {$\idcat{O^1} \tensor f$} (3);
	\draw[1csurj] (0) to node[auto,swap] {$! \tensor \idcat{U}$} (2);
	\draw[1csurj] (2) to node[auto,swap] {$f$} (1);
	\draw[1csurj] (3) to node[auto] {$! \tensor \idcat{V}$} (1);
\end{tikzpicture}
\end{equation*}
descends to a commutative square
\begin{equation*}
\begin{tikzpicture}[baseline={([yshift=-.5ex]current bounding box.center)}]
	\node[scale=1.25] (0) at (0,1.5) {$O(U)$};
	\node[scale=1.25] (1) at (2.5,0) {$V$.};
	\node[scale=1.25] (2) at (0,0) {$U$};
	\node[scale=1.25] (3) at (2.5,1.5) {$O(V)$};
	\draw[1csurj] (0) to node[auto] {$O(f)$} (3);
	\draw[1csurj] (0) to node[auto,swap] {$p_U$} (2);
	\draw[1csurj] (2) to node[auto,swap] {$f$} (1);
	\draw[1csurj] (3) to node[auto] {$p_V$} (1);
\end{tikzpicture}
\end{equation*}
Because $O(f)$ is also a surjective map of $(n+1)$\nbd atoms, we can iterate this construction, letting $O^0(f) := f: U \surj V$ and $O^n(f) := O(O^{n-1}(f)): O^n(U) \surj O^n(V)$. 
\end{remark}

\begin{remark}
Construction \ref{cons:ou} will allow us to produce ``units'' on diagrams shaped as molecules with spherical boundary in a diagrammatic set. 

Its importance is the reason why we use regular, and not constructible directed complexes as in \cite{hadzihasanovic2018combinatorial}: Proposition \ref{prop:ou_regular} does not hold for constructible molecules, that is, there are constructible molecules $U$ such that $O(U)$ is not constructible.

For example, let $V$ be the following constructible 3-atom:
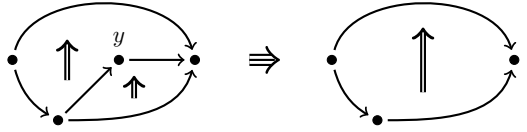
\begin{equation*}
\begin{tikzpicture}[baseline={([yshift=-.5ex]current bounding box.center)},scale=.8]
\begin{scope}
	\node[0c] (0) at (-1.5,0) {};
	\node[0c] (1) at (1.5,0) {};
	\node[0c] (a) at (-.75,-1) {};
	\node[0c] (m) at (.25,0) [label=above:$y$] {};
	\draw[1c, out=75,in=105] (0) to (1);
	\draw[1c, out=-75,in=150] (0) to (a);
	\draw[1c, out=0, in=-105] (a) to (1);
	\draw[1c] (a) to (m);
	\draw[1c] (m) to (1);
	\draw[2c] (-.6, -.6) to (-.6,.6);
	\draw[2c] (.5, -.9) to (.5,0);
	\draw[3c1] (2,0) to (3.25,0);
	\draw[3c2] (2,0) to (3.25,0);
	\draw[3c3] (2,0) to (3.25,0);
\end{scope}
\begin{scope}[shift={(5.25,0)}]
	\node[0c] (0) at (-1.5,0) {};
	\node[0c] (1) at (1.5,0) {};
	\node[0c] (a) at (-.75,-1) {};
	\draw[1c, out=75,in=105] (0) to (1);
	\draw[1c, out=-75,in=150] (0) to (a);
	\draw[1c, out=0, in=-105] (a) to (1);
	\draw[2c] (0, -.8) to (0,.8);
	\node[scale=1.25] at (1.75,-1) {,};
\end{scope}
\end{tikzpicture}
\end{equation*}
and let $C := \{y\}$. By Lemma \ref{lem:uc_regular}, $V_C$ is a regular 4-atom with spherical boundary, but $\bord{}{+}V_C$ has the following form:
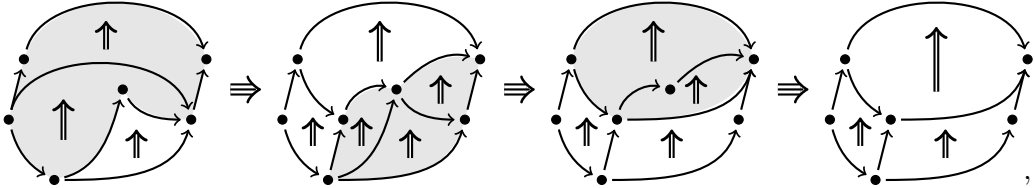
\begin{equation*}
\begin{tikzpicture}[baseline={([yshift=-.5ex]current bounding box.center)},scale=.8]
\begin{scope}
	\path[fill, color=gray!20] (-1.5,0) to [out=-75,in=150] (-.75,-1) to [out=15,in=-105,looseness=1.05] (.375,.5) to [out=-60,in=180] (1.5,0) to (1.75,1) to [out=105,in=75,looseness=1.1] (-1.25,1) to (-1.5,0);
	\node[0c] (0) at (-1.5,0) {};
	\node[0c] (1) at (1.5,0) {};
	\node[0c] (a) at (-.75,-1) {};
	\node[0c] (m) at (.375,.5) {};
	\node[0c] (0b) at (-1.25,1) {};
	\node[0c] (1b) at (1.75,1) {};
	\draw[1c, out=75,in=105] (0) to (1);
	\draw[1c, out=75,in=105] (0b) to (1b);
	\draw[1c, out=-75,in=150] (0) to (a);
	\draw[1c, out=0, in=-105] (a) to (1);
	\draw[1c] (0) to (0b);
	\draw[1c] (1) to (1b);
	\draw[1c,out=15, in=-105] (a) to (m);
	\draw[1c,out=-60,in=180] (m) to (1);
	\draw[2c] (-.6, -.6) to (-.6,.6);
	\draw[2c] (.6, -.9) to (.6,.1);
	\draw[2c] (.1, .9) to (.1,1.9);
	\draw[3c1] (1.75,.5) to (3,.5);
	\draw[3c2] (1.75,.5) to (3,.5);
	\draw[3c3] (1.75,.5) to (3,.5);
\end{scope}
\begin{scope}[shift={(4.5,0)}]
	\path[fill, color=gray!20] (-.75,-1) to (-.5,0) to [out=75,in=165,looseness=1.05] (.375,.5) to [out=45,in=165] (1.75,1) to (1.5,0) to [out=-105,in=0,looseness=1.1] (-.75,-1);
	\node[0c] (0) at (-1.5,0) {};
	\node[0c] (1) at (1.5,0) {};
	\node[0c] (a) at (-.75,-1) {};
	\node[0c] (m) at (.375,.5) {};
	\node[0c] (0b) at (-1.25,1) {};
	\node[0c] (1b) at (1.75,1) {};
	\node[0c] (ab) at (-.5,0) {};
	\draw[1c, out=75,in=105] (0b) to (1b);
	\draw[1c, out=-75,in=150] (0) to (a);
	\draw[1c, out=-75,in=150] (0b) to (ab);
	\draw[1c, out=0, in=-105] (a) to (1);
	\draw[1c, out=75, in=165] (ab) to (m);
	\draw[1c, out=45, in=165] (m) to (1b);
	\draw[1c] (0) to (0b);
	\draw[1c] (1) to (1b);
	\draw[1c] (a) to (ab);
	\draw[1c,out=15, in=-105] (a) to (m);
	\draw[1c,out=-60,in=180] (m) to (1);
	\draw[2c] (-1, -.7) to (-1,.3);
	\draw[2c] (-.2, -.7) to (-.2,.3);
	\draw[2c] (1.1, 0) to (1.1,1);
	\draw[2c] (.1, .7) to (.1,1.9);
	\draw[2c] (.6, -.9) to (.6,.1);
	\draw[3c1] (1.75,.5) to (3,.5);
	\draw[3c2] (1.75,.5) to (3,.5);
	\draw[3c3] (1.75,.5) to (3,.5);
\end{scope}
\begin{scope}[shift={(9,0)}]
	\path[fill, color=gray!20] (-1.25,1) to [out=-75,in=150] (-.5,0) to [out=0,in=-105] (1.75,1) to [out=105,in=75,looseness=1.1] (-1.25,1);
	\node[0c] (0) at (-1.5,0) {};
	\node[0c] (1) at (1.5,0) {};
	\node[0c] (a) at (-.75,-1) {};
	\node[0c] (m) at (.375,.5) {};
	\node[0c] (0b) at (-1.25,1) {};
	\node[0c] (1b) at (1.75,1) {};
	\node[0c] (ab) at (-.5,0) {};
	\draw[1c, out=75,in=105] (0b) to (1b);
	\draw[1c, out=-75,in=150] (0) to (a);
	\draw[1c, out=-75,in=150] (0b) to (ab);
	\draw[1c, out=0, in=-105] (a) to (1);
	\draw[1c, out=0, in=-105] (ab) to (1b);
	\draw[1c, out=75, in=165] (ab) to (m);
	\draw[1c, out=45, in=165] (m) to (1b);
	\draw[1c] (0) to (0b);
	\draw[1c] (1) to (1b);
	\draw[1c] (a) to (ab);
	\draw[2c] (-1, -.7) to (-1,.3);
	\draw[2c] (.4, -.9) to (.4,.1);
	\draw[2c] (.8, 0) to (.8,1);
	\draw[2c] (.1, .7) to (.1,1.9);
	\draw[3c1] (1.75,.5) to (3,.5);
	\draw[3c2] (1.75,.5) to (3,.5);
	\draw[3c3] (1.75,.5) to (3,.5);
\end{scope}
\begin{scope}[shift={(13.5,0)}]
	\node[0c] (0) at (-1.5,0) {};
	\node[0c] (1) at (1.5,0) {};
	\node[0c] (a) at (-.75,-1) {};
	\node[0c] (0b) at (-1.25,1) {};
	\node[0c] (1b) at (1.75,1) {};
	\node[0c] (ab) at (-.5,0) {};
	\draw[1c, out=75,in=105] (0b) to (1b);
	\draw[1c, out=-75,in=150] (0) to (a);
	\draw[1c, out=-75,in=150] (0b) to (ab);
	\draw[1c, out=0, in=-105] (a) to (1);
	\draw[1c, out=0, in=-105] (ab) to (1b);
	\draw[1c] (0) to (0b);
	\draw[1c] (1) to (1b);
	\draw[1c] (a) to (ab);
	\draw[2c] (-1, -.7) to (-1,.3);
	\draw[2c] (.4, -.9) to (.4,.1);
	\draw[2c] (.25, .2) to (.25,1.8);
	\node[scale=1.25] at (1.75,-1) {,};
\end{scope}
\end{tikzpicture}
\end{equation*}
where the shaded area in each diagram is the input boundary of the following 3\nbd atom. This is not a constructible 3\nbd molecule, because no pair of 3\nbd atoms forms a constructible 3\nbd molecule: the union of the first two does not have spherical boundary, and the union of the next two is not a molecule.

It is easy to produce a constructible 3-molecule $U$ that contains $V$ as an atom, and such that $V \cap \bord{}{}U = \{y\}$. Then $O(U)$ contains $V_C$ as an atom, hence it is not constructible. 
\end{remark}

The following construction produces shapes of ``unitor'' cells.

\begin{cons} \label{cons:unitor_molec}
Let $U$ be an $n$\nbd atom, and $V \subsph \bord{}{-}U$ a regular $(n-1)$\nbd molecule with spherical boundary. There is an $(n+1)$\nbd atom $L^V_U := U \celto (O(V) \cup U)$, where the union is along $\bord{}{+}O(V) \incliso V \subsph \bord{}{-}U$, and a retraction $\ell^V_U: L^V_U \surj U$ 
\begin{itemize}
	\item sending the greatest element of $L^V_U$ to the greatest element of $U$,
	\item equal to the identity on both isomorphic copies of $U$ in $\bord{}{}L^V_U$, and
	\item equal to $p_V: O(V) \surj V$ on $O(V) \subsph \bord{}{+}L^V_U$.
\end{itemize}
Dually in dimension $n$, if $V' \subsph \bord{}{+}U$, there is an $(n+1)$\nbd atom $R^{V'}_U := U \celto (U \cup O(V'))$, together with a retraction $r^{V'}_U: R^{V'}_U \surj U$. 

Dually in dimension $n+1$, we also have two atoms 
\begin{equation*}
	\tilde{L}^V_U := \oppn{n+1}{L^V_U}, \quad \quad \tilde{R}^{V'}_U := \oppn{n+1}{R^{V'}_U}
\end{equation*}
with retractions $\tilde{\ell}^V_U: \tilde{L}^V_U \surj U$ and $\tilde{r}^{V'}_U: \tilde{R}^{V'}_U \surj U$. 
\end{cons}

\begin{cons} \label{cons:fattening}
Let $U$ be an $n$\nbd atom, $V$ an $(n-1)$\nbd atom, and $f: U \surj V$ a surjective map. There is a surjective map $f_\prec: U \surj O(V)$ defined as follows:
\begin{itemize}
	\item $f_\prec$ maps the greatest element of $U$ to the greatest element of $O(V)$;
	\item if $x \in \bord{}{\alpha}U$, then $f_\prec$ maps it to $f(x) \in V \incliso \bord{}{\alpha}O(V)$.
\end{itemize}
We can see $f_\prec$ as a ``fattened'' version of $f$, splitting the image of its two boundaries. If $p_V: O(V) \surj V$ is the map of Construction \ref{cons:ou}, we have that 
\begin{equation*}
\begin{tikzpicture}[baseline={([yshift=-.5ex]current bounding box.center)}]
	\node[scale=1.25] (0) at (-1.25,1.25) {$U$};
	\node[scale=1.25] (1) at (0,0) {$O(V)$};
	\node[scale=1.25] (2) at (1.25,1.25) {$V$};
	\draw[1csurj] (0) to node[auto] {$f$} (2);
	\draw[1csurj] (0) to node[auto,swap] {$f_\prec$} (1);
	\draw[1csurj] (1) to node[auto,swap] {$p_V$} (2);
\end{tikzpicture}
\end{equation*}
commutes. 
\end{cons}

Next, we generalise the notion of a merger tree \cite[Construction 2.18]{hadzihasanovic2018combinatorial} from constructible molecules to regular molecules with spherical boundary.

\begin{cons} \label{cons:mergertree}
Let $U$ be a regular $n$\nbd molecule with spherical boundary. A \emph{binary split} of $U$ is a pair $U_1, U_2 \subsph U$ of $n$\nbd molecules with spherical boundary such that
\begin{enumerate}
	\item $U_1 \cap U_2 \subseteq \bord{}{+}U_1 \cap \bord{}{-}U_2$, 
	\item $\tilde{U_1} := U_1 \cup \bord{}{-}U$ and $\tilde{U_2} := U_2 \cup \bord{}{+}U$ are molecules, $U_1 \subsph \tilde{U_1}$ and $U_2 \subsph \tilde{U_2}$, and
	\item $U = U_1 \cup U_2 = \tilde{U_1} \cp{n-1} \tilde{U_2}$.
\end{enumerate}
We say that $U$ is \emph{unsplittable} if it admits no binary split.

A \emph{merger tree} $\mathfrak{T}$ for $U$ is a rooted binary tree whose vertices are labelled with $n$\nbd molecules $V \subsph U$, such that
\begin{enumerate}
	\item the root is labelled $U$,
	\item the leaves are labelled with unsplittable molecules, and
	\item if the children of a vertex labelled $V$ are labelled $V_1$ and $V_2$, then $V = V_1 \cup V_2$ is a binary split of $V$.
\end{enumerate}
\end{cons}

\begin{exm}
All atoms are unsplittable. The only unsplittable constructible molecules are atoms: by \cite[Theorem 5.18]{hadzihasanovic2018combinatorial}, merger trees for constructible molecules are merger trees in the generalised sense, and every constructible molecule admits a merger tree whose leaves are labelled with atoms.

The 3-molecule of \cite[Remark 7.12]{hadzihasanovic2018combinatorial} is an unsplittable regular 3-molecule with three maximal elements.
\end{exm}

\begin{remark} \label{rmk:binarysplit}
If a regular $n$\nbd molecule $U$ with spherical boundary has two maximal elements, then it admits a binary split: using Lemma \ref{lem:composition_form}, we can write $U = \tilde{U_1} \cp{n-1} \tilde{U_2}$ where $\tilde{U_1}$ and $\tilde{U_2}$ contain a single $n$\nbd dimensional element $x_1 \in \tilde{U_1}$, $x_2 \in \tilde{U_2}$; since $U$ is pure, $U = \clos\{x_1\} \cup \clos\{x_2\}$, and necessarily $\clos\{x_1\} \subsph \tilde{U_1}$ and $\clos\{x_2\} \subsph \tilde{U_2}$. 
\end{remark}

\subsection{Diagrammatic sets} \label{sec:presheaves}

\begin{dfn}
A \emph{diagrammatic set} is a presheaf on $\atom$. Diagrammatic sets, together with their morphisms of presheaves, form a category $\dgmset$.
\end{dfn}

We will assume that $\atom$ is skeletal: isomorphism of directed complexes is decidable, so this is a reasonable requirement, and spares us the trouble of dealing with pairs of elements that are related by an isomorphism of atoms.

We have the Yoneda embedding $\atom \incl \dgmset$. By Corollary \ref{cor:globpos_is_colimit} and the universal property of presheaf categories as free cocompletions, this extends to an embedding $\rdcpx \incl \dgmset$. We will casually identify any regular directed complex with its embedding. 

\begin{cons}
Let $X$ be a diagrammatic set, and $J \subseteq \mathbb{N}^+$. The \emph{$J$\nbd dual} $\oppn{J}{X}$ of $X$ is the diagrammatic set defined by $\oppn{J}{X}(-) := X(\oppn{J}{-})$. For each $J$, this defines an endofunctor $\oppn{J}{-}$ on $\dgmset$.

In particular, we write $\oppall{X}$, $\opp{X}$, and $\coo{X}$ for $X(\oppall{(-)})$, $X(\opp{(-)})$, and $X(\coo{(-)})$, respectively.
\end{cons}

In \cite[Section 4]{hadzihasanovic2018combinatorial}, we considered the extension of the lax Gray product and join of constructible directed complexes to constructible polygraphs, defined as presheaves on a category of constructible atoms and inclusions. These extend with no effort to regular directed complexes and diagrammatic sets, and we bundle the relevant facts together into a single statement.

\begin{prop}
The following facts hold.
\begin{enumerate}
	\item The monoidal structure $(-\tensor-, 1)$ on $\rdcpx$ extends to a monoidal biclosed structure $(-\tensor-, \homoplax{-}{-}, \homlax{-}{-}, 1)$ on $\dgmset$.
	\item The monoidal structure $(-\join-, \emptyset)$ on $\rdcpx$ extends to $\dgmset$, with natural inclusions $X \incl X \join Y, Y \incl X \join Y$ for all diagrammatic sets $X, Y$. 
	
	This monoidal structure is locally biclosed, in the sense that, for all $X$, there are right adjoints to the functors $\dgmset \to \slice{X}{\dgmset}$ defined by
\begin{equation*}
	f: Y \to Z \;\; \mapsto
\begin{tikzpicture}[baseline={([yshift=-.5ex]current bounding box.center)}]
	\node[scale=1.25] (0) at (-1.25,-1.25) {$X \join Y$};
	\node[scale=1.25] (1) at (0,0) {$X$};
	\node[scale=1.25] (2) at (1.25,-1.25) {$X \join Z$};
	\draw[1cincl] (1) to (0);
	\draw[1cinc] (1) to (2);
	\draw[1c] (0) to node[auto] {$\mathrm{id}_X \join f$} (2);
	\node[scale=1.25] at (2,-1.4) {,}; 
\end{tikzpicture} \quad
	f: Y \to Z \;\; \mapsto
\begin{tikzpicture}[baseline={([yshift=-.5ex]current bounding box.center)}]
	\node[scale=1.25] (0) at (-1.25,-1.25) {$Y \join X$};
	\node[scale=1.25] (1) at (0,0) {$X$};
	\node[scale=1.25] (2) at (1.25,-1.25) {$Z \join X$};
	\draw[1cincl] (1) to (0);
	\draw[1cinc] (1) to (2);
	\draw[1c] (0) to node[auto] {$f \join \mathrm{id}_X$} (2);
	\node[scale=1.25] at (2,-1.4) {.};
\end{tikzpicture}
\end{equation*}
	\item There are canonical isomorphisms $\opp{(X \tensor Y)} \simeq \opp{Y} \tensor \opp{X}$, $\coo{(X \tensor Y)} \simeq \coo{Y} \tensor \coo{X}$, $\oppall{(X \tensor Y)} \simeq \oppall{X} \tensor \oppall{Y},$ and $\opp{(X \join Y)} \simeq \opp{Y} \join \opp{X}$, natural in the diagrammatic sets $X, Y$. 
\end{enumerate}
\end{prop}

Let us introduce some terminology.

\begin{dfn}
Let $X$ be a diagrammatic set. An $n$\nbd \emph{diagram} of shape $U$ in $X$ is a morphism $x: U \to X$, where $U$ is an $n$\nbd molecule. An $n$\nbd diagram of shape $U$ is \emph{spherical} if $U$ has spherical boundary, and it is an $n$\nbd \emph{cell} if $U$ is an $n$\nbd atom.

If $x$ is a diagram of shape $U$, the \emph{input $k$\nbd boundary} of $x$ is the diagram $\bord{k}{-}x := \imath^-_k;x$, where $\imath^-_k: \bord{k}{-}U \incl U$ is the inclusion of the input $k$\nbd boundary of $U$ into $U$. Similarly, the \emph{output $k$\nbd boundary} of $x$ is $\bord{k}{+}x := \imath^+_k;x$, where $\imath^+_k: \bord{k}{+}U \incl U$ is the inclusion of the output $k$\nbd boundary of $U$ into $U$. We also write $\bord{k}{}x := \imath_k;x$, where $\imath$ is the inclusion $\bord{k}{}U \incl U$. When $U$ is $n$\nbd dimensional and $k = n-1$, we omit the index.

If $x$ is an $n$\nbd diagram, we write $x: y^- \celto y^+$ to express that $\bord{n-1}{\alpha}x = y^\alpha$ for $\alpha \in \{+,-\}$. We say that two $n$\nbd diagrams $x, x': y^- \celto y^+$ are \emph{parallel}.

Let $x$ be an $n$\nbd diagram of shape $U$, and let $V \incl U$ be the inclusion of a directed complex $V$. We write $\restr{x}{V}$ for the restriction of $x$ to $V$, and $\restr{x}{V} \subseteq x$. We also write $y \subsph x$ if $y = \restr{x}{V}$ for some spherical submolecule $V \subsph U$, and say $y$ is a \emph{spherical subdiagram} of $x$.

If $x$ is a diagram of shape $U = U_1 \cp{k} U_2$, we write $x = x_1 \cp{k} x_2$ where $x_1, x_2$ are the restrictions of $x$ to $U_1$ and $U_2$, respectively. If $z$ is a spherical $n$\nbd diagram of shape $W$, parallel to $y \subsph x$ of shape $V$, we write $x[z/y]$ for the $n$\nbd diagram of shape $U[W/V]$ which is equal to $z$ on $W$, and to $x$ on $U \setminus V$. 
\end{dfn}

\begin{dfn}
Suppose $x_1$ and $x_2$ are $n$\nbd diagrams of shapes $U_1$ and $U_2$ with the following property: there is a directed complex $V$ and inclusions $V \incl \bord{}{+}U_1$ and $V \incl \bord{}{-}U_2$ such that 
\begin{enumerate}
	\item $\restr{x_1}{V} = \restr{x_2}{V} =: y$, and
	\item the pushout
\begin{equation*}
\begin{tikzpicture}[baseline={([yshift=-.5ex]current bounding box.center)}]
	\node[scale=1.25] (0) at (0,1.5) {$V$};
	\node[scale=1.25] (1) at (2.5,0) {$U$};
	\node[scale=1.25] (2) at (0,0) {$U_1$};
	\node[scale=1.25] (3) at (2.5,1.5) {$U_2$};
	\draw[1cinc] (0) to (3);
	\draw[1cincl] (0) to (2);
	\draw[1cinc] (2) to (1);
	\draw[1cincl] (3) to (1);
	\draw[edge] (1.6,0.2) to (1.6,0.7) to (2.3,0.7);
\end{tikzpicture}
\end{equation*}
is a regular $n$\nbd molecule with $U_1, U_2 \submol U$.
\end{enumerate}
Then we write $x_1 \cup x_2$ for the $n$\nbd diagram of shape $U$ defined by $\restr{(x_1 \cup x_2)}{U_i} = x_i$ for $i = 1,2$, and call it the \emph{union} of $x_1$ and $x_2$ \emph{along} $y \subseteq \bord{}{+}x_1, \bord{}{-}x_2$. We also write $y = x_1 \cap x_2$. We will often leave $y$ implicit, when clear from context.
\end{dfn}

The following is a general fact that we will use in a few occasions.

\begin{cons}
Let $\cat{D}$ be a small category. For each subcategory $\cat{C}$ of $\cat{D}$, there is a restriction functor $-_\cat{C}: \psh{}{\cat{D}} \to \psh{}{\cat{C}}$. This functor has a left adjoint $\imath_\cat{C}: \psh{}{\cat{C}} \to \psh{}{\cat{D}}$, with the following description: for each presheaf $X$ on $\cat{C}$, the elements of $\imath_\cat{C}X(d)$ are pairs $(x \in X(c), f: d \to c)$, where $c$ is an object of $\cat{C}$ and $f: d \to c$ a morphism of $\cat{D}$, quotiented by the relation
\begin{equation*}
	(g^*x \in X(c), f: d \to c) \sim (x \in X(c'), f;g: d \to c')
\end{equation*}
for all morphisms $g: c \to c'$ in $\cat{C}$. If $h: d' \to d$ is a morphism in $\cat{D}$, $h^*$ sends the equivalence class of $(x, f)$ in $\imath_\cat{C}X(d)$ to the equivalence class of $(x, h;f)$ in $\imath_\cat{C}X(d')$.

The functor $\imath_\cat{C}$ can equally be characterised as the left Kan extension of the inclusion $\cat{C} \incl \cat{D} \incl \psh{}{\cat{D}}$ along the Yoneda embedding $\cat{C} \incl \psh{}{\cat{C}}$.
\end{cons}

\begin{lem} \label{lem:leftadj_full}
Suppose $\cat{C}$ is a full subcategory of $\cat{D}$. Then $\imath_\cat{C}: \psh{}{\cat{C}} \to \psh{}{\cat{D}}$ is full and faithful.
\end{lem}
\begin{proof}
If $c$ is an object of $\cat{C}$, each equivalence class of $\imath_\cat{C}X(c)$ has a unique representative of the form $(x, \idcat{c})$. It follows that the components $X \to (\imath_\cat{C}X)_\cat{C}$ of the unit of the adjunction are isomorphisms, and we can apply \cite[Theorem IV.3.1]{maclane1971cats}.
\end{proof}

We will use adjunctions of the form $\imath_\cat{C} \dashv -_\cat{C}$ for subcategories of $\atom$. 

\begin{exm}
Let $\tilde{\cat{O}}$ be the full subcategory of $\atom$ on the globes $O^n$. Presheaves on $\tilde{\cat{O}}$ can be identified with \emph{reflexive} $\omega$\nbd graphs, for which see \cite[Example 10.1.2]{leinster2004higher}. 

Letting $\globsetrefl$ denote the category of reflexive $\omega$\nbd graphs, the restriction functor $-_{\tilde{\cat{O}}}: \dgmset \to \globsetrefl$ has a full and faithful left adjoint $\imath_{\tilde{\cat{O}}}: \globsetrefl \to \dgmset$. Thus reflexive $\omega$\nbd graphs can be identified with a full subcategory of diagrammatic sets.
\end{exm}
In some cases, there is a canonical choice of representatives for elements of $\imath_\cat{C}X$.

\begin{exm} \label{exm:freedeg}
Let $\rpol := \psh{}{\atomin}$; we conjecture that this is equivalent to the category of \emph{regular polygraphs} of \cite{henry2018regular}, which should follow from Conjecture \ref{conj:freegen}.

Let $u: \rpol \to \dgmset$ be left adjoint to the restriction $\dgmset \to \rpol$; this can be seen as the functor that ``freely adds degeneracies'' to a regular polygraph.

It follows from Lemma \ref{lem:factor_clpsincl} that, if $X$ is a regular polygraph, each cell of shape $U$ in $uX$ has a unique representative of the form $(x \in X(V), f: U \surj V)$, where $f$ is a surjective map of atoms. This explicit description can be used to prove easily that the adjunction is monadic: that is, diagrammatic sets can be seen as regular polygraphs with algebraic structure.
\end{exm}

\begin{cons} \label{cons:nskeleton}
The category $\atom$ has a filtration given by the full subcategories $\atom_n$ for $n \geq 0$, whose objects are the atoms of dimension $k \leq n$. We call a presheaf on $\atom_n$ an \emph{$n$\nbd diagrammatic set}, and write $n\dgmset$ for their category.

The restrictions $-_{\leq n}: \dgmset \to n\dgmset$ have full and faithful left adjoints,
\begin{equation*}
	\imath_n: n\dgmset \to \dgmset.
\end{equation*}
If $X$ is a diagrammatic set, we write $\skel{n}{X} := \imath_nX_{\leq n}$, and call the counit $\skel{n}{X} \to X$ the \emph{$n$\nbd skeleton} of $X$. 

The restrictions factor as a sequence of restriction functors 
\begin{equation*}
	\ldots \to n\dgmset \to (n-1)\dgmset \to \ldots \to 0\dgmset,
\end{equation*}
all with full and faithful left adjoints. By universal properties, the skeleta of $X$ form a sequence of morphisms
\begin{equation*}
	\skel{0}{X} \to \ldots \to \skel{n-1}{X} \to \skel{n}{X} \to \ldots
\end{equation*}
over $X$, whose colimit is $X$. 
\end{cons}

\begin{cons} \label{cons:ncoskel}
The restrictions $-_{\leq n}: \dgmset \to n\dgmset$ also have right adjoints $\imath^!_n: n\dgmset \to \dgmset$, by right Kan extension of $\atom_n \incl \dgmset$ along the Yoneda embedding $\atom_n \incl n\dgmset$. 

If $X$ is a presheaf on $\atom_n$, we have for all atoms $U$
\begin{equation*}
	\imath^!_nX(U) \simeq \homset{\dgmset}(U, \imath^!_n X) \simeq \homset{n\dgmset}(\skel{n}{U}, X),
\end{equation*}
that is, $\imath^!_nX$ has one cell of shape $U$ for every way of mapping the directed complex $\skel{n}{U}$ into $X$. If $U$ has dimension $k \leq n$, these are the same as the cells of $X$. If $k > n$, since $\skel{n}{\bord{}{}U} \incl \skel{n}{U}$, any two parallel $k$\nbd cells in $\imath^!_n X$ must be equal, and between any two parallel spherical $(k-1)$\nbd diagrams in $\imath^!_n X$ there must be a $k$\nbd cell.

If $X$ is a diagrammatic set, we write $\coskel{n}{X} := \imath_n^! X_{\leq n}$, and call the counit $X \to \coskel{n}{X}$ the \emph{$n$\nbd coskeleton} of $X$. The coskeleta of $X$ form a tower of morphisms 
\begin{equation} \label{eq:coskel-tower}
	\ldots \to \coskel{n}{X} \to \coskel{n-1}{X} \to \ldots \to \coskel{0}{X}
\end{equation}
under $X$, whose limit is $X$.
\end{cons}

\begin{dfn}
We say that $X$ is \emph{$n$\nbd coskeletal} if its $n$\nbd coskeleton is an isomorphism.
\end{dfn}

\begin{dfn}
Let $x: U \to X$ be an $n$\nbd cell in a diagrammatic set. We say that $x$ is \emph{irreducible} if, whenever $x = p^*y$ for some surjective map $p: U \surj V$ and $y \in X(V)$, it follows that that $U = V$ and $p$ is the identity. We write $\nondeg{n}{X}$ for the set of irreducible $n$\nbd cells of $X$.

We say that $x$ is \emph{degenerate} if there exist a $k$\nbd cell $y: V \to X$, with $k < n$, and a surjective map $p: U \surj V$ such that $x = p^*y$. We write $\degg{X}$ for the set of degenerate cells of $X$. 
\end{dfn}

\begin{remark}
Every irreducible cell is non-degenerate. A non-degenerate cell, however, may not be irreducible, because of the existence of surjective maps between atoms of the same dimension. This should be contrasted with simplicial sets, where the two corresponding notions coincide.
\end{remark}

Because atoms have finitely many elements, and surjective maps decrease the number of elements, every cell of a diagrammatic set must be equal to $p^*y$ for some irreducible cell $y$. However, this expression may not in general be unique.

\begin{exm}
Let $X$ be the quotient $O^2/\bord{}{}O^2$, that is, the pushout
\begin{equation*}
\begin{tikzpicture}[baseline={([yshift=-.5ex]current bounding box.center)}]
	\node[scale=1.25] (0) at (0,1.5) {$\bord{}{}O^2$};
	\node[scale=1.25] (1) at (2.5,0) {$X$};
	\node[scale=1.25] (2) at (0,0) {$1$};
	\node[scale=1.25] (3) at (2.5,1.5) {$O^2$};
	\draw[1cinc] (0) to (3);
	\draw[1c] (0) to (2);
	\draw[1cinc] (2) to (1);
	\draw[1c] (3) to node[auto] {$x$} (1);
	\draw[edge] (1.6,0.2) to (1.6,0.7) to (2.3,0.7);
\end{tikzpicture}
\end{equation*}
in $\dgmset$; then $x$ is the only irreducible 2\nbd cell in $X$. 

Let $U := (O^1 \cp{0} O^1) \celto O^1$; there are two surjective maps $p_1, p_2: U \surj O^2$, the first of which collapses the first copy of $O^1$, and the second of which collapses the second copy of $O^1$ in the input boundary of $U$. The 2-cells $p_1^*x$ and $p_2^*x$ have equal boundaries in $X$, and we can let $X'$ be the quotient of $X$ by $p_1^*x \sim p_2^*x$, that is, the coequaliser of the pair $p_1^*x, p_2^*x: U \to X$. In $X'$, the cell $x$ is still irreducible, and we have $p_1^*x = p_2^*x$, but $p_1 \neq p_2$.
\end{exm}

\begin{dfn}
We say that $X$ has the \emph{Eilenberg-Zilber property} if, for all atoms $U$ and cells $x: U \to X$, there is a unique irreducible cell $y: V \to X$ and a unique surjective map $p: U \surj V$ such that $x = p^*y$. 
\end{dfn}

\begin{exm}
By the remark in Example \ref{exm:freedeg}, for all regular polygraphs $X$, the diagrammatic set $uX$ has the Eilenberg-Zilber property. 
\end{exm}

\begin{prop} \label{prop:ez_polygraph}
Let $X$ be a diagrammatic set. For each irreducible cell $x \in \nondeg{n}{X}$, let $U(x)$ be the shape of $x$. The following are equivalent:
\begin{enumerate}[label=(\alph*)]
	\item $X$ has the Eilenberg-Zilber property;
	\item for all $n > 0$, the diagram
\begin{equation} \label{eq:ez_pushout}
\begin{tikzpicture}[baseline={([yshift=-.5ex]current bounding box.center)}]
	\node[scale=1.25] (0) at (0,2) {$\displaystyle \coprod_{x\in\nondeg{n}{X}} \bord{}{} U(x)$};
	\node[scale=1.25] (1) at (3.5,2) {$\displaystyle\coprod_{x\in\nondeg{n}{X}} U(x)$};
	\node[scale=1.25] (2) at (0,0) {$\skel{n-1}{X}$};
	\node[scale=1.25] (3) at (3.5,0) {$\skel{n}{X}$};
	\draw[1c] (0) to node[auto,swap] {$(\bord{}{}x)_{x\in\nondeg{n}{X}}$} (2);
	\draw[1c] (1) to node[auto] {$(x)_{x\in\nondeg{n}{X}}$} (3);
	\draw[1cinc] (0) to (1);
	\draw[1cinc] (2) to (3);
	\draw[edge] (2.5,0.2) to (2.5,0.8) to (3.3,0.8);
\end{tikzpicture}
\end{equation}
is a pushout in $\dgmset$.
\end{enumerate}
\end{prop}
\begin{proof}
Since every $n$\nbd cell of $X$ is of the form $p^*x$ for some irreducible cell $x$ of dimension $k \leq n$ and surjective map $p$, there is always a surjective morphism 
\begin{equation*}
	\skel{n-1}{X} \cup \coprod_{x\in\nondeg{n}{X}} U(x) \to \skel{n}{X}.
\end{equation*}
It is easy to see that this morphism is also injective if and only if $X$ has the Eilenberg-Zilber property. 
\end{proof}

\begin{remark}
Since the top side of (\ref{eq:ez_pushout}) is a monomorphism and $\dgmset$ is a topos, the bottom side is also a monomorphism. It follows that if $X$ satisfies the Eilenberg-Zilber property, the $n$\nbd skeleton $\skel{n}{X} \to X$ is a monomorphism.
\end{remark}

\section{Representable diagrammatic sets}

\subsection{Equivalences and representability} \label{sec:equivalences}

The following are the fundamental definitions for this section.

\begin{dfn}
Let $W$ be an $(n+1)$\nbd atom, $n \geq 0$. A \emph{horn} of $W$ is an inclusion $\Lambda \incl W$ whose image is $W$ minus the greatest element and a single $n$\nbd dimensional element. 

In particular, $\Lambda \incl W$ is a \emph{composition horn} if the image of $\Lambda$ is $\bord{}{\alpha}W$ for some $\alpha \in \{+,-\}$, that is, $\bord{}{-\alpha}W$ is an atom, and its greatest element is the single $n$\nbd dimensional element not in the image of $\Lambda$. In this case, $\Lambda$ is an $n$\nbd molecule with spherical boundary.

Let $X$ be a diagrammatic set. A \emph{horn of $W$ in $X$} is a pair of a horn $\Lambda \incl W$ and a morphism $\lambda: \Lambda \to X$. A \emph{filler} for the horn is an $(n+1)$\nbd cell $h: W \to X$ such that
\begin{equation*}
\begin{tikzpicture}[baseline={([yshift=-.5ex]current bounding box.center)}]
	\node[scale=1.25] (0) at (0,1.5) {$\Lambda$};
	\node[scale=1.25] (1) at (2,1.5) {$X$};
	\node[scale=1.25] (2) at (0,0) {$W$};
	\draw[1c] (0) to node[auto] {$\lambda$} (1);
	\draw[1cinc] (0) to (2);
	\draw[1c] (2) to node[auto,swap] {$h$} (1);
\end{tikzpicture}
\end{equation*}
commutes. 
\end{dfn}

\begin{dfn}
For $n > 0$, a \emph{ternary} $(n+1)$\nbd atom $W$ is an $(n+1)$\nbd atom with three $n$\nbd dimensional elements. Necessarily, one boundary of $W$ splits into two $n$\nbd atoms $W_+$ and $W_-$, with $W_+ \cap W_- \subseteq \bord{}{+}W_+ \cap \bord{}{-}W_-$, and the other boundary is a single $n$\nbd atom $W_0$. Let
\begin{equation*}
	\Lambda^W_- := W_+ \cup W_0, \quad \quad \Lambda^W_+ := W_- \cup W_0.
\end{equation*}

Let $X$ be a diagrammatic set, $x: U \to X$ an $n$\nbd cell of $X$ and $V \subseteq \bord{}{\alpha}U$. A \emph{division horn for $x$ at $V$} is a pair of an inclusion $\imath: U \incl W$ such that $\imath(U) = W_\alpha$ and $\imath(V) = W_+ \cap W_-$, and a horn $\lambda: \Lambda^W_{-\alpha} \to X$ such that
\begin{equation*}
\begin{tikzpicture}[baseline={([yshift=-.5ex]current bounding box.center)}]
	\node[scale=1.25] (0) at (0,1.5) {$U$};
	\node[scale=1.25] (1) at (2,0) {$X$};
	\node[scale=1.25] (2) at (0,0) {$\Lambda^W_{-\alpha}$};
	\draw[1c] (0) to node[auto] {$x$} (1);
	\draw[1cinc] (0) to (2);
	\draw[1c] (2) to node[auto,swap] {$\lambda$} (1);
\end{tikzpicture}
\end{equation*}
commutes. If $\alpha = +$, we call this a \emph{left division horn}, and if $\alpha = -$ a \emph{right division horn}. We write $\horn{x: U \to X, V \subseteq \bord{}{\alpha}U}$ for the set of division horns for $x$ at $V$. 
\end{dfn}

\begin{remark} \label{rmk:dual_horn}
If $W$ is an $(n+1)$\nbd atom, every horn of $W$ is isomorphic to a horn of $\oppn{n+1}{W}$. In particular, for each division horn for $x: U \to X$ at $V$, given by $\Lambda \incl W$ and $\lambda: \Lambda \to X$, there is a \emph{dual} division horn, where $\Lambda$ is seen as a horn of $\oppn{n+1}{W}$.
\end{remark}

\begin{dfn}
Let $X$ be a diagrammatic set and $e: U \to X$ an $n$\nbd cell of $X$, $n > 0$. Coinductively, we say that $e$ is an \emph{$n$\nbd equivalence} if all division horns $\Lambda \incl W, \lambda: \Lambda \to X$ for $e$ at $\bord{}{+}U$ or at $\bord{}{-}U$ have a filler $h: W \to X$ which is an $(n+1)$\nbd equivalence.

We write $\equi{n}{X}$ for the set of $n$\nbd equivalences of $X$.
\end{dfn}

\begin{remark} \label{rmk:explain_equi}
The definition requires some explanation. The idea is that an $(n+1)$\nbd equivalence $h$ such that $\bord{}{\alpha}h$ is an $n$\nbd cell exhibits $\bord{}{\alpha}h$ as a \emph{weak composite} of the spherical $n$\nbd diagram $\bord{}{-\alpha}h$. Thus, filling a division horn for $e$ at $\bord{}{+}U$ is finding a solution $x$ to a well-formed equation 
\begin{equation*}
	e \cup x = y
\end{equation*}
where equality is ``up to higher equivalence'', and $e \cup x$ is a union along the entire output boundary $\bord{}{+}e$. Similarly, filling a division horn for $e$ at $\bord{}{-}U$ is finding a solution to an equation $x \cup e = y$ where the union is along $\bord{}{-}e$. 

Thus, if $e$ is an equivalence, any cell that can factor through $e$ along a submolecule of its boundary does so. This is similar to the definition of universal cells in the opetopic or multitopic approach \cite{baez1998higher}, yet dispenses with any requirement of weak uniqueness, in favour of non-unique but \emph{two-sided} factorisation. 

This requires the input-output symmetry of regular atoms in a fundamental way, and cannot be adapted to asymmetric shapes such as opetopes or oriented simplices. An important advantage is that having a purely existential requirement allows us to obtain a proper coinductive definition, and tackle infinite-dimensional higher categories.
\end{remark}

\begin{remark} \label{rmk:coinduction}
Because coinductive definitions and proofs are not very common outside of computer science, we take some time to justify more in detail the definition and the corresponding proof method. 

Let $X$ be a diagrammatic set, and let $S$ be the set of all its cells. For all subsets $A \subseteq S$, define
\begin{align*}
	\mathcal{F}(A) := \; & \{x: U \to X \,|\, \text{ for all } \alpha \in \{+,-\} \text{ and } (\Lambda \incl W, \lambda: \Lambda \to X) \in \horn{x, \bord{}{\alpha}U}, \\
	& \text{there exists } (h: W \to X) \in A \text{ such that } \restr{h}{\Lambda} = \lambda \};
\end{align*}
that is, $\mathcal{F}(A)$ contains the cells $x: U \to X$ with the property that all division horns for $x$ at $\bord{}{\alpha}U$ have fillers in $A$.

Clearly if $A \subseteq B$, then $\mathcal{F}(A) \subseteq \mathcal{F}(B)$, so $\mathcal{F}$ defines an endofunctor on the powerset $\mathcal{P}(S)$ seen as a poset. Any such endofunctor has a greatest fixed point (terminal coalgebra) which can be constructed as the sequential limit of
\begin{equation*}
	\ldots \subseteq \mathcal{F}^k(S) \subseteq \ldots \subseteq \mathcal{F}(S) \subseteq S.
\end{equation*}
By definition, this coincides with $\equi{}{X} := \bigcup_{n\in\mathbb{N}}\equi{n}{X}$.

This provides the following proof method: given a subset $A \subseteq S$, if $A \subseteq \mathcal{F}(A)$ (so $A$ is a coalgebra for $\mathcal{F}$), then $A \subseteq \equi{}{X}$. We will use this to prove closure and characterisation properties of equivalences. 

Because of the grading of $S$ and its subsets given by the dimension of diagrams, such proofs may look like ``inductive proofs without the base case''. Indeed, if $A = \bigcup_{n\in\mathbb{N}} A_n$, to prove $A \subseteq \mathcal{F}(A)$ we need to show that for each $n$, given $x \in A_n$, there are enough $(n+1)$\nbd dimensional fillers in $A_{n+1}$. 

Informally, this proof may be phrased as follows. Let $P(n)$ be the statement that $A_n \subseteq \equi{n}{X}$. Suppose for all $n$, we find horn fillers for $x \in A_n$ in $A_{n+1}$. ``Assuming'' $P(n+1)$, those horn fillers are actually in $\equi{n+1}{X}$, which implies that $x \in \equi{n}{X}$, hence that $P(n)$ holds. 

From the coinductive proof principle, we have concluded that $P(n)$ holds for all $n$, so it may seem as if we proved $\forall n \, P(n)$ from the fact that, for all $n$, $P(n+1)$ implies $P(n)$. This, of course, is \emph{not} a valid proof principle, since it would allow us to derive a contradiction from a uniformly false $P(n)$: the point is not that the fillers ``are equivalences by assumption'', but rather that they are in $A_{n+1}$. 
\end{remark}

We have the following easy properties.

\begin{prop} \label{prop:equiv_coprod}
Let $\{X_i\}_{i \in I}$ be a family of diagrammatic sets, and $\coprod_{i \in I} X_i$ its coproduct. Then $\equi{}{\left(\coprod_{i \in I} X_i\right)} = \coprod_{i \in I} \equi{}{X_i}$.
\end{prop}
\begin{proof}
Immediate from the fact that cell and every horn in $\coprod_{i \in I} X_i$ factor through an inclusion $X_j \incl \coprod_{i \in I} X_i$.
\end{proof}

\begin{prop} \label{prop:equiv_prod}
Let $\{X_i\}_{i \in I}$ be a family of diagrammatic sets, and $\prod_{i \in I} X_i$ its product. If $\{x_i: U \to X_i\}_{i \in I}$ is a family of equivalences with the same shape $U$, then the cell $(x_i)_{i \in I}: U \to \prod_{i \in I} X_i$ is an equivalence.
\end{prop}
\begin{proof}
Let $A$ be the set of cells of $X := \prod_{i \in I} X_i$ whose projections are all equivalences, and let $x \in A$ be a cell of shape $U$. 

A division horn $\Lambda \incl W, \lambda: \Lambda \to X$ for $x$ at $\bord{}{\alpha}U$ produces a division horn $\lambda; p_i: \Lambda \to X_i$ for the projection $p_i(x)$ at $\bord{}{\alpha}U$, for each $i \in I$. By assumption, this produces an equivalence filler $h_i: W \to X_i$ for each $i \in I$. It follows that $(h_i)_{i \in I} \in A$. Thus $A \subseteq \mathcal{F}(A)$, and by coinduction $A \subseteq \equi{}{X}$. 
\end{proof}

After the explanation in Remark \ref{rmk:explain_equi}, the following definition should be clear: a diagrammatic set is representable if it has weak composites of all spherical diagrams. 

\begin{dfn}
Let $X$ be a diagrammatic set. We say that $X$ is \emph{representable} if all composition horns $\Lambda \incl W, x: \Lambda \to X$ have a filler $c(x): W \to X$ which is an $(n+1)$\nbd equivalence.

We call $c(x)$ a \emph{compositor} for $x$. If the image of $\Lambda$ is $\bord{}{\alpha}W$, we call $\compos{x} := \restr{c(x)}{\bord{}{-\alpha}W}$ a \emph{weak composite} of $x$.
\end{dfn}

We will abbreviate ``representable diagrammatic set'' as RDS.

\begin{remark}
The use of \emph{representable} here is modelled on Hermida's representable multicategories \cite{hermida2000representable}, and should not be confused with the notion of representable presheaf.
\end{remark}

\begin{remark}
The definition of representable diagrammatic sets is similar to the definition of fibrant objects in a cofibrantly generated model structure, as objects that have the right lifting property with respect to a set of generating trivial cofibrations, except for the condition that $c(x)$ be an equivalence. 

This is analogous to the condition that certain horns have \emph{thin} fillers in the complicial model of higher categories \cite{verity2008weak}, with the difference that the choice of thin cells is structure on a simplicial set, whereas being an equivalence is a property of a cell in a diagrammatic set. However \emph{saturated} complicial sets, as described in \cite{riehl2018complicial}, may be a more accurate parallel. See also Remark \ref{rmk:stratified}.
\end{remark}

We have a first, simple closure result for representables. 

\begin{prop} \label{prop:rep_product_closure}
Let $\{X_i\}_{i \in I}$ be a family of RDSs. Then its coproduct $\coprod_{i \in I} X_i$ and its product $\prod_{i \in I} X_i$ in $\dgmset$ are both representable.
\end{prop}
\begin{proof}
The fact that fillers for composition horns exist follows from general properties of classes of objects defined by a right lifting property. The fact that these fillers are equivalences follows from Proposition \ref{prop:equiv_coprod} and \ref{prop:equiv_prod}.
\end{proof}

\begin{remark}
On the other hand, the lax Gray product $X \tensor Y$ of two representables is almost never representable: already if $X$ and $Y$ have at least one non-degenerate 1-cell each, say $x: x_1 \celto x_2$ in $X$ and $y: y_1 \celto y_2$ in $Y$, then the spherical 1-diagram
\begin{equation*}
	(x \tensor y_1)\cp{0}(x_2 \tensor y): x_1 \tensor y_1 \celto x_2 \tensor y_2
\end{equation*}
has no weak composite. 
\end{remark}

\subsection{Closure properties of equivalences} \label{sec:closure}

At the moment, it is not clear that equivalences satisfy any of their expected properties, in particular that equivalence cells satisfy a 2-out-of-3 property (relative to weak composition witnessed by higher equivalences) and that all degenerate cells are equivalences. 

To prove this, we will proceed as follows: we will add all degenerate cells to $\equi{}{X}$ and ``saturate'' by closing under the desired 2-out-of-3 property, obtaining a set 
\begin{equation*}
	\satur{\infty}{\equi{}{X} \cup \degg{X}} \supseteq \equi{}{X};
\end{equation*}
then we will prove that, if $X$ is representable,
\begin{equation*}
	\satur{\infty}{\equi{}{X} \cup \degg{X}} \subseteq \mathcal{F}(\satur{\infty}{\equi{}{X} \cup \degg{X}})
\end{equation*}
where $\mathcal{F}$ is the functor of Remark \ref{rmk:coinduction}, so by coinduction $\satur{\infty}{\equi{}{X} \cup \degg{X}} = \equi{}{X}$. 

\begin{cons}
Let $A$ be a set of cells of a diagrammatic set $X$. We define $\satur{}{A}$ to be the set of cells $x: U \to X$ such that either $x \in A$, or, letting $n = \dmn{U}$, there exists an $(n+1)$\nbd atom $W$, an inclusion $U \incl W$, and a cell $h: W \to X$ such that $h \in A$ and, for some $\alpha \in \{+,-\}$, either
\begin{itemize}
	\item $\restr{h}{V} \in A$ for all $n$\nbd atoms $V \submol \bord{}{\alpha}W$ and $x = \restr{h}{\bord{}{-\alpha}W}$, or
	\item $W$ is ternary, $\restr{h}{W_\alpha}, \restr{h}{W_0} \in A$, $W_+ \cap W_- = \bord{}{\alpha}W_\alpha$, and $x = \restr{h}{W_{-\alpha}}$.
\end{itemize}
We then let $\satur{\infty}{A} := \bigcup_{n\in \mathbb{N}} \satur{n}{A}$. 
\end{cons}

The idea is that $\satur{\infty}{A}$ is the closure of $A$ under composition of cells in $A$ and division of cells in $A$ by cells in $A$ (along the entire boundary of the latter), both witnessed by higher-dimensional cells in $A$.

\begin{lem} \label{lem:2satur_reverse}
Let $A$ be a set of cells in a diagrammatic set $X$, and $e: x \celto y$ be a cell in $A \cap \mathcal{F}(A)$. Then there is a cell $\tilde{e}: y \celto x$ in $\satur{2}{A}$.
\end{lem}
\begin{proof}
Suppose $e$ is an $n$\nbd cell of shape $U$, and let $W := U \celto (U \cp{n-1} (\bord{}{+}U \celto \bord{}{+}U))$, a ternary $(n+1)$\nbd atom. Then the morphism $\lambda: \Lambda^W_- \to X$ equal to $e$ both on $W_+$ and on $W_0$ is a left division horn for $e$ at $\bord{}{+}U$; filling it, that is, ``dividing $e$ by itself'', we obtain a cell $h: W \to X$ in $A$, restricting on $W_-$ to $e': y \celto y$, such that $e' \in \satur{1}{A}$.

Now, let $W' := (\bord{}{+}U \celto \bord{}{+}U) \celto (\oppn{n}{U} \cp{n-1} U)$, another ternary atom. The morphism $\lambda': \Lambda^{W'}_+ \to X$ equal to $e'$ on $W'_0$ and to $e$ on $W'_-$ is a right division horn for $e$ at $\bord{}{-}U$. Filling it, we obtain a cell $h': W' \to X$ in $A$, restricting on $W'_+$ to $\tilde{e}: y \celto x$, which belongs to $\satur{2}{A}$.
\end{proof}

\begin{prop} \label{prop:equiv_reverse}
If $e: x \celto y$ is an equivalence in a diagrammatic set $X$, then there is a cell $\tilde{e}: y \celto x$ in $\satur{2}{\equi{}{X}}$.
\end{prop}
\begin{proof}
Follows from the lemma with $A := \equi{}{X}$.
\end{proof}

\begin{lem} \label{lem:nary_composite}
Let $X$ be a representable diagrammatic set, $x$ a spherical $n$\nbd diagram, and $\compos{x}$ a weak composite. Then there are $(n+1)$\nbd cells 
\begin{equation*}
	h: \compos{x} \celto x, \quad \quad h': x \celto \compos{x}, 
\end{equation*}
such that $h, h' \in \satur{\infty}{\equi{}{X} \cup \degg{X}}$. If all $n$\nbd cells in $x$ belong to $\satur{\infty}{\equi{}{X} \cup \degg{X}}$, then so does $\compos{x}$.
\end{lem}
\begin{proof}
Let $x: U \to X$ be a spherical $n$\nbd diagram. We can suppose $\compos{x}$ is given by the filler $h$ of the horn $U \incl (\compos{U} \celto U)$, $x: U \to X$. Then $h \in \equi{}{X}$, and by Proposition \ref{prop:equiv_reverse} there is also a cell $h': x \celto \compos{x}$ in $\satur{2}{\equi{}{X}}$. \emph{A fortiori}, $h, h' \in \satur{\infty}{\equi{}{X} \cup \degg{X}}$. If all $n$\nbd cells in $x$ belong to $\satur{\infty}{\equi{}{X} \cup \degg{X}}$, then by its closure property so does $\compos{x}$.
\end{proof}

\begin{dfn}
Let $A$ be a set of cells in a diagrammatic set $X$, and $x: U \to X$ an $n$\nbd cell in $\mathcal{F}(A)$. We say that $A$ is \emph{unbiased for $x$} if, for all $\alpha \in \{+,-\}$ and division horns for $x$ at $\bord{}{\alpha}U$, there are fillers of dual horns $h: W \to X$, $h': \oppn{n+1}{W} \to X$ in $A$ such that $\restr{h}{W_{-\alpha}} = \restr{h'}{{\oppn{n+1}{W}}_{-\alpha}}$. 

If $B \subseteq \mathcal{F}(A)$, we say that $A$ is \emph{unbiased for $B$} if $A$ is unbiased for each $x \in B$.
\end{dfn}

\begin{remark}
Under the interpretation of ``filling a division horn for $x$ at $\bord{}{\alpha}U$'' as ``dividing a cell by $x$'', the restriction of a filler $h: W \to X$ to $W_{-\alpha}$ is the result of the division. If $x \in \mathcal{F}(A)$, given a division horn for $x$ at $\bord{}{\alpha}U$, both it and its dual in the sense of Remark \ref{rmk:dual_horn} have fillers in $A$, but filling the original horn and its dual may give different ``results of the division''. If $A$ is unbiased for $x$, then the fillers of dual horns can be chosen in such a way that they produce the same result.
\end{remark}

\begin{lem} \label{lem:eq_deg_unbiased}
Let $X$ be a representable diagrammatic set. Then:
\begin{enumerate}
	\item $\equi{}{X} \cup \degg{X} \subseteq \mathcal{F}(\satur{\infty}{\equi{}{X} \cup \degg{X}})$;
	\item $\satur{\infty}{\equi{}{X} \cup \degg{X}}$ is unbiased for $\equi{}{X} \cup \degg{X}$. 
\end{enumerate}
\end{lem}
\begin{proof}
We have $\equi{}{X} = \mathcal{F}(\equi{}{X}) \subseteq \mathcal{F}(\satur{\infty}{\equi{}{X} \cup \degg{X}})$. Moreover, if $e: U \to X$ is an equivalence, then every division horn for $e$ at $\bord{}{\alpha}U$ has a filler $h: x \celto y$ in $\equi{}{X}$ which is itself an equivalence. Proposition \ref{prop:equiv_reverse} then gives a cell $h': y \celto x$ in $\satur{2}{\equi{}{X}}$, which is a filler of the dual horn. This proves that $\satur{2}{\equi{}{X}}$ is unbiased for $\equi{}{X}$, and \emph{a fortiori} so is $\satur{\infty}{\equi{}{X} \cup \degg{X}}$.

Next, let $x: U \to X$ be an $n$\nbd cell in $\degg{X}$, that is, $x = f^*y$ for some surjective map $f: U \surj V$ and cell $y: V \to X$, where $V$ has lower dimension than $U$. Let $\Lambda \incl W, \lambda: \Lambda \to X$ be a division horn for $x$ at $\bord{}{+}U$: then $\restr{\lambda}{W_0}$ is an $n$\nbd cell $z: W_0 \to X$ and by Proposition \ref{prop:boundary_submol} $\bord{}{-}x \subsph \bord{}{-}z$. Suppose $W_0 = \bord{}{-}W$.

Now, consider the $(n+1)$\nbd atom $L^{\bord{}{-}U}_{W_0} := W_0 \celto (O(\bord{}{-}U) \cup W_0)$ as in Construction \ref{cons:unitor_molec}, together with the retraction $q := \ell_{W_0}^{\bord{}{-}U}: L^{\bord{}{-}U}_{W_0} \surj W_0$. Then, $q^*z$ is a degenerate $(n+1)$\nbd cell $z \celto (p_{\bord{}{-}U}^*(\bord{}{-}x) \cup z)$. 

Let $f': \oppn{n}{U} \surj V$ be the ``reverse'' of $f$ defined in Lemma \ref{lem:reverse_surj}. Then, let $\tilde{U}$ be the $(n+1)$\nbd atom $O(\bord{}{-}U) \celto (U \cp{n-1} \oppn{n}{U})$. There is a surjective map $g: \tilde{U} \surj V$, 
\begin{itemize}
	\item sending the greatest element of $\tilde{U}$ to the greatest element of $V$,
	\item equal to the composite of $p_{\bord{}{-}U}: O(\bord{}{-}U) \surj \bord{}{-}U$ and $\restr{f}{\bord{}{-}U}: \bord{}{-}U \surj V$ on $\bord{}{-}\tilde{U}$, and 
	\item equal to $f$ on $U \submol \bord{}{+}\tilde{U}$ and to $f'$ on $\oppn{n}{U} \submol \bord{}{+}\tilde{U}$.
\end{itemize}
Then $g^*y$ is a degenerate $(n+1)$\nbd cell $p_{\bord{}{-}U}^*(\bord{}{-}x) \celto (x \cup f'^*y)$.

Finally, $f'^*y \cup z$, where the union is along $\bord{}{+}(f'^*y) = \bord{}{-}x \subsph \bord{}{-}z$, is a spherical $n$\nbd diagram, which by representability has a compositor $h: (f'^*y \cup z) \celto \compos{f'^*y \cup z}$.

The union 
\begin{equation*}
	q^*z \cup g^*y \cup h: z \celto (x \cup \compos{f'^*y \cup z})
\end{equation*}
is a spherical $(n+1)$-diagram in $X$, whose $(n+1)$\nbd cells are either degenerate or equivalences. It follows from Lemma \ref{lem:nary_composite} that it has a weak composite $k \in \satur{\infty}{\equi{}{X} \cup \degg{X}}$, which fills $\lambda$.

Now, observe that both $q: L^{\bord{}{-}U}_{W_0} \celto W_0$ and $g: \tilde{U} \surj V$ can be reversed as by Lemma \ref{lem:reverse_surj}, giving surjective maps $q': \tilde{L}^{\bord{}{-}U}_{W_0} \celto W_0$ and $g': \oppn{n+1}{\tilde{U}} \surj V$ and degenerate $(n+1)$\nbd cells
\begin{equation*}
	q'^*z: (p_{\bord{}{-}U}^*(\bord{}{-}x) \cup z) \celto z, \quad \quad g'^*y: (x \cup f'^*y) \celto p_{\bord{}{-}U}^*(\bord{}{-}x).
\end{equation*}
Moreover, by Proposition \ref{prop:equiv_reverse}, we can ``reverse'' the compositor $h$ to obtain an $(n+1)$\nbd cell $h': \compos{f'^*y \cup z} \celto (f'^*y \cup z)$ in $\satur{2}{\equi{}{X}}$. Then 
\begin{equation*}
	h' \cup g'^*y \cup q'^*z: (x \cup \compos{f'^*y \cup z}) \celto z
\end{equation*}
also has a weak composite $k' \in \satur{\infty}{\equi{}{X} \cup \degg{X}}$, which fills the dual horn of $\lambda$, and such that $\restr{k}{W_-} = \restr{k'}{\oppn{n+1}{W}_-} = \compos{f'^*y \cup z}$. 

The case of horns for $x$ at $\bord{}{-}U$ is dual. Hence $\degg{X} \subseteq \mathcal{F}(\satur{\infty}{\equi{}{X} \cup \degg{X}})$, and $\satur{\infty}{\equi{}{X} \cup \degg{X}}$ is unbiased for $\degg{X}$, which completes the proof.
\end{proof}

\begin{dfn}
Let $x: U \to X$ be a spherical $n$\nbd diagram in a representable diagrammatic set. A \emph{unit} $\thin{x}$ on $x$ is a weak composite of the spherical $(n+1)$\nbd diagram $p_U^*x: O(U) \to X$. 
\end{dfn}

By Lemma \ref{lem:nary_composite}, since all the $(n+1)$\nbd cells in $p_U^*x$ are degenerate, $\thin{x}$ belongs to $\satur{\infty}{\equi{}{X} \cup \degg{X}}$. Moreover, we can assume that we have cells $h: \thin{x} \celto p_U^*x$ and $h': p_U^*x \celto \thin{x}$ both in $\satur{\infty}{\equi{}{X} \cup \degg{X}}$.

\begin{cons} \label{cons:unitors}
Let $x: U \to X$ be an $n$\nbd cell in a representable diagrammatic set, and let $V \subsph \bord{}{-}U$, $y := \restr{x}{V} \subsph \bord{}{-}x$. If $\ell^V_U: L^V_U \surj U$ is the map of Construction \ref{cons:unitor_molec}, we have that $(\ell^V_U)^*x: L_U^V \to X$ is an $(n+1)$\nbd cell $x \celto (p_U^*y \cup x)$. Taking a compositor $h': p_U^*y \celto \thin{y}$, we can consider the spherical $(n+1)$\nbd diagram 
\begin{equation*}
	(\ell^V_U)^*x \cup h': x \celto (\thin{y} \cup x).
\end{equation*}
This has a weak composite $\ell_x^V: x \celto (\thin{y} \cup x)$, which belongs to $\satur{\infty}{\equi{}{X} \cup \degg{X}}$.

Similarly, using the dual atoms and maps of Construction \ref{cons:unitor_molec}, and the ``reverse compositors'', given $V' \subsph \bord{}{+}U$ and $z := \restr{x}{V'}$, we can build cells 
\begin{equation*}
	\tilde{\ell}_x^V: (\thin{y} \cup x) \celto x, \quad \quad r_x^{V'}: x \celto (x \cup \thin{z}), \quad \quad \tilde{r}_x^{V'}: (x \cup \thin{z}) \celto x,
\end{equation*}
all belonging to $\satur{\infty}{\equi{}{X} \cup \degg{X}}$. We call the $\ell_x^V, \tilde{\ell}_x^V$ \emph{left unitors}, and the $r_x^{V'}, \tilde{r}_x^{V'}$ \emph{right unitors} for $x$.
\end{cons}

\begin{lem} \label{lem:satur_inductive_step}
Let $X$ be a representable diagrammatic set, and let $B$ be a subset of $\satur{\infty}{\equi{}{X} \cup \degg{X}} \cap \mathcal{F}(\satur{\infty}{\equi{}{X} \cup \degg{X}})$.

If $\satur{\infty}{\equi{}{X} \cup \degg{X}}$ is unbiased for $B$, then $\satur{}{B} \subseteq \mathcal{F}(\satur{\infty}{\equi{}{X} \cup \degg{X}})$, and $\satur{\infty}{\equi{}{X} \cup \degg{X}})$ is unbiased for $\satur{}{B}$.
\end{lem}
\begin{proof}
Let $e: U \to X$ be an $n$\nbd cell in $\satur{}{B}$. If $e \in B$, there is nothing to prove. Otherwise, let $\Lambda \incl W, \lambda: \Lambda \to X$ be a division horn for $e$ at $\bord{}{+}U$, that is, $\restr{\lambda}{W_0}$ is an $n$\nbd cell $x$ and, by Proposition \ref{prop:boundary_submol}, $\bord{}{-}e \subsph \bord{}{-}x$. Suppose $W_0 = \bord{}{-}W$; because we will construct fillers both of $\lambda$ and its dual horn, this is inconsequential. We proceed by case distinction.

\begin{enumerate}
	\item Suppose there is an $(n+1)$\nbd cell $h: e \celto y$, where $h$ and every $n$\nbd cell in $y$ belong to $B$. By Lemma \ref{lem:2satur_reverse} applied to $A := \satur{\infty}{\equi{}{X} \cup \degg{X}}$, there is also an $(n+1)$\nbd cell $h': y \celto e$ in $\satur{\infty}{\equi{}{X} \cup \degg{X}}$.
	
	By Lemma \ref{lem:composition_form}, we can write $y = y_1 \cp{n} \ldots \cp{n} y_m$, where each $y_i$ contains a single $n$\nbd cell $e_i$; because $y$ is spherical, its shape is pure and $n$\nbd dimensional, so $y = e_1 \cup \ldots \cup e_m$. 
	
	By Proposition \ref{prop:boundary_submol},
	\begin{equation*}
		\bord{}{-}e_1 \subsph \bord{}{-}y = \bord{}{-}e \subsph \bord{}{-}x,
	\end{equation*}
	and $e_1 \in \mathcal{F}(\satur{\infty}{\equi{}{X} \cup \degg{X}})$, so we can ``divide $x$ by $e_1$'' instead of $e$, and obtain a cell $k_1: x \celto (e_1 \cup x_1)$ in $\satur{\infty}{\equi{}{X} \cup \degg{X}}$. Now, letting $y' := y_2 \cp{n} \ldots \cp{n} y_m$, we have
	\begin{equation*}
		\bord{}{-}x_1 = \bord{}{-}x[\bord{}{+}e_1/\bord{}{-}e_1] \supsph \bord{}{-}e[\bord{}{+}e_1/\bord{}{-}e_1] = \bord{}{-}y' \supsph \bord{}{-}e_2,
	\end{equation*}
	using Proposition \ref{prop:substitution}; because $e_2 \in \mathcal{F}(\satur{\infty}{\equi{}{X} \cup \degg{X}})$, we can ``divide $x_1$ by $e_2$'', and obtain a cell $k_2: x_1 \celto (e_2 \cup x_2)$, also in $\satur{\infty}{\equi{}{X} \cup \degg{X}}$. Now
	\begin{equation*}
		k_1 \cup k_2: x \celto (e_1 \cup e_2 \cup x_2)
	\end{equation*}
	is a spherical $(n+1)$\nbd diagram. We can iterate this, each time dividing $x_j$ by $e_{j+1}$, for $j < m$, to obtain cells 
	\begin{equation*}
		k_{j+1}: x_j \celto (e_{j+1} \cup x_{j+1}),
	\end{equation*}
	until we obtain a spherical $(n+1)$\nbd diagram
	\begin{equation*}
		k_1 \cup \ldots \cup k_m: x \celto (y \cup x_m).
	\end{equation*}
	Then $k_1 \cup \ldots \cup k_m \cup h': x \celto (e \cup x_m)$ is also spherical, therefore it has a weak composite, still in $\satur{\infty}{\equi{}{X} \cup \degg{X}}$, which works as a filler of $\lambda$. 
	
	Since $\satur{\infty}{\equi{}{X} \cup \degg{X}}$ is unbiased for $e_1, \ldots, e_m$, we can also obtain cells
	\begin{equation*}
		k'_{j+1}: (e_{j+1} \cup x_{j+1}) \celto x_j
	\end{equation*}
	for each $j < m$ (where we let $x_0 := x$), and a weak composite of the spherical diagram $h \cup k'_m \cup \ldots \cup k'_1: (e \cup x_m) \celto x$ belongs to $\satur{\infty}{\equi{}{X} \cup \degg{X}}$ and fills the dual horn.
	
	The case where we start with $h: y \celto e$, where $h$ and all the cells in $y$ are in $B$, is dual.
	
	\item Suppose there is an $(n+1)$\nbd cell $h: e_1 \celto (e_2 \cup e)$, where $h, e_1, e_2 \in B$, and $\bord{}{+}e_2 \subsph \bord{}{-}e$. As in the previous point, we can also find a cell $h': (e_2 \cup e) \celto e_1$ in $\satur{\infty}{\equi{}{X} \cup \degg{X}}$.
	
	Suppose $e_2: y \celto z$, and take a unit $\thin{z}: z \celto z$ in $\satur{\infty}{\equi{}{X} \cup \degg{X}}$. Since $e_2 \in \mathcal{F}(\satur{\infty}{\equi{}{X} \cup \degg{X}})$, we can ``divide $\thin{z}$ by $e_2$'' to obtain an $n$\nbd cell $e_2^*: z \celto y$, and, by unbiasedness, $(n+1)$\nbd cells
	\begin{equation*}
		k: (e_2^* \cup e_2) \celto \thin{z}, \quad \quad k': \thin{z} \celto (e_2^* \cup e_2),
	\end{equation*}
	both in $\satur{\infty}{\equi{}{X} \cup \degg{X}}$. 
	
	Since $z \subsph \bord{}{-}e \subsph \bord{}{-}x$, there is a left unitor $\ell^V_x: x \celto (\thin{z} \cup x)$ for the suitable $V \subsph \bord{}{-}W_0$, and 
	\begin{equation*}
		\ell^V_x \cup k': x \celto (e_2^* \cup e_2 \cup x)
	\end{equation*}
	is a spherical $(n+1)$\nbd diagram of cells in $\satur{\infty}{\equi{}{X} \cup \degg{X}}$, with a weak composite $\compos{\ell^V_x \cup k'}$ in $\satur{\infty}{\equi{}{X} \cup \degg{X}}$. 
	
	The $n$\nbd diagram $e_2 \cup x$ is spherical and binary, so by representability we can find a weak composite $x'$ and an equivalence $c: x' \celto (e_2 \cup x)$. Since 
	\begin{equation*}
	e_2 \cup x = \bord{}{+}c \subsph \bord{}{+}\compos{\ell^V_x \cup k'},
	\end{equation*}
	we can divide $\compos{\ell^V_x \cup k'}$ by $c$ to obtain an $(n+1)$\nbd cell 
	\begin{equation*}
		d: x \celto (e_2^* \cup x'),
	\end{equation*}
	which still belongs to $\satur{\infty}{\equi{}{X} \cup \degg{X}}$. Then
	\begin{equation*}
		\bord{}{-}x' = \bord{}{-}x[\bord{}{-}e_2/\bord{}{+}e_2] \sqsupseteq \bord{}{-}e[\bord{}{-}e_2/\bord{}{+}e_2] = \bord{}{-}e_1;
	\end{equation*}
	because $e_1 \in \mathcal{F}(\satur{\infty}{\equi{}{X} \cup \degg{X}})$, we can divide $x'$ by $e_1$, and get an $(n+1)$\nbd cell $m: x' \celto (e_1 \cup x'')$ in $\satur{\infty}{\equi{}{X} \cup \degg{X}}$. 
	
	Then, $d \cup m: x \celto (e_2^* \cup e_1 \cup x'')$ is a spherical $(n+1)$\nbd diagram, and so are, successively,
	\begin{align*}
		(d \cup m) \cup h: & \; x \celto (e_2^* \cup e_2 \cup e \cup x''), \\
		((d \cup m) \cup h) \cup k: & \; x \celto (1_z \cup e \cup x''), \\
		(((d \cup m) \cup h) \cup k) \cup \tilde{\ell}_{e}^V: & \; x \celto (e \cup x''),
	\end{align*}
	where $\tilde{\ell}_{e}^V: (1_z \cup e) \celto e$ is a left unitor. All these $(n+1)$\nbd cells belong to $\satur{\infty}{\equi{}{X} \cup \degg{X}}$, so they have a weak composite in $\satur{\infty}{\equi{}{X} \cup \degg{X}}$, which works as a filler of $\lambda$. 
	
	As for the dual horn, the only cell that we constructed that cannot obviously be reversed using either Lemma \ref{lem:2satur_reverse} or unbiasedness is $d: x \celto (e_2^* \cup x')$. Indeed, while we have 
	\begin{equation*}
		k \cup \tilde{l}_x^V: (e_2^* \cup e_2 \cup x) \celto x
	\end{equation*}
	and a weak composite $\compos{k \cup \tilde{l}_x^V}$, we are not guaranteed that there is an equivalence $(e_2 \cup x) \celto x'$ by which we can divide $\compos{k \cup \tilde{l}_x^V}$. Nevertheless, there is a compositor $\tilde{c}: (e_2 \cup x) \celto \tilde{x}$ for some $\tilde{x}$, so we can obtain a cell
	\begin{equation*}
		\tilde{d}: (e_2^* \cup \tilde{x}) \celto x
	\end{equation*}
	in $\satur{\infty}{\equi{}{X} \cup \degg{X}}$. Then $\tilde{c} \cup \tilde{d}: (e_2^* \cup e_2 \cup x) \celto x$ is a spherical $(n+1)$\nbd diagram, and so is
	\begin{equation*}
		c \cup (\tilde{c} \cup \tilde{d}): (e_2^* \cup x') \celto x.
	\end{equation*}
	This has a weak composite $d': (e_2^* \cup x') \celto x$, still in $\satur{\infty}{\equi{}{X} \cup \degg{X}}$, which we use to reverse $d$. This enables us to fill the dual horn of $\lambda$ with a cell $(e \cup x'') \celto x$.

	The case where we start with $h: (e_2 \cup e) \celto e_1$ with $e_1, e_2, h \in B$ is dual.
	
	\item Suppose there is an $(n+1)$\nbd cell $h: e_1 \celto (e \cup e_2)$, where $h, e_1, e_2 \in B$, and $\bord{}{-}e_2 \subsph \bord{}{+}e$. As in the previous two points, there is also a cell $h': (e \cup e_2) \celto e_1$ in $\satur{\infty}{\equi{}{X} \cup \degg{X}}$. 
	
	In this case, $\bord{}{-}e = \bord{}{-}e_1$, and we can divide $x$ by $e_1$, producing an $(n+1)$\nbd cell $k: x \celto (e_1 \cup x')$ in $\satur{\infty}{\equi{}{X} \cup \degg{X}}$. Finally, by representability, there exist an $n$\nbd cell $x'$ and an equivalence $c: (e_2 \cup x') \celto x''$. The $(n+1)$\nbd diagram
	\begin{equation*}
		(k \cup h) \cup c: x \celto (e \cup x'')
	\end{equation*}
	is spherical and has a weak composite in $\satur{\infty}{\equi{}{X} \cup \degg{X}}$, which fills $\lambda$.
	
	Both $k$ and $c$ can be ``reversed'', the first because $\satur{\infty}{\equi{}{X} \cup \degg{X}}$ is unbiased for $e_1 \in B$, the second by Proposition \ref{prop:equiv_reverse}, so we have cells
	\begin{equation*}
		k': (e_1 \cup x') \celto x, \quad \quad c': x'' \celto (e_2 \cup x'),
	\end{equation*}
	in $\satur{\infty}{\equi{}{X} \cup \degg{X}}$, and a spherical diagram 
	\begin{equation*}
		c' \cup (h' \cup k'): (e \cup x'') \celto x
	\end{equation*}
	with a weak composite in $\satur{\infty}{\equi{}{X} \cup \degg{X}}$, which fills the dual horn of $\lambda$.
	
	The case where we start with $h: (e \cup e_2) \celto e_1$ with $e_1, e_2, h \in B$ is dual.
\end{enumerate}
This proves the existence in $\satur{\infty}{\equi{}{X} \cup \degg{X}}$ of fillers of $\lambda$ and its dual horn which restrict to the same cell on $W_-$. 

By duality, we can also find fillers of division horns for $e$ at $\bord{}{-}U$ which restrict to the same cell on $W_+$. Therefore $e \in \mathcal{F}(\satur{\infty}{\equi{}{X} \cup \degg{X}})$, and $\satur{\infty}{\equi{}{X} \cup \degg{X}}$ is unbiased for $e$. 
\end{proof}

\begin{thm} \label{thm:equiv_closure}
If $X$ is a representable diagrammatic set, $\satur{\infty}{\equi{}{X} \cup \degg{X}} = \equi{}{X}$.
\end{thm}
\begin{proof}
The inclusion $\equi{}{X} \subseteq \satur{\infty}{\equi{}{X} \cup \degg{X}}$ is obvious. For the converse, we prove by induction that, for all $n \geq 0$, $\satur{n}{\equi{}{X} \cup \degg{X}} \subseteq \mathcal{F}(\satur{\infty}{\equi{}{X} \cup \degg{X}})$ and $\satur{\infty}{\equi{}{X} \cup \degg{X}})$ is unbiased for $\satur{n}{\equi{}{X} \cup \degg{X}}$: the base case is given by Lemma \ref{lem:eq_deg_unbiased}, and the inductive step by Lemma \ref{lem:satur_inductive_step}.

It follows that $\satur{\infty}{\equi{}{X} \cup \degg{X}} \subseteq \mathcal{F}(\satur{\infty}{\equi{}{X} \cup \degg{X}})$. By coinduction, then, $\satur{\infty}{\equi{}{X} \cup \degg{X}} \subseteq \equi{}{X}$.
\end{proof}

\begin{remark}
This implies, in retrospect, that in the proofs towards this result, whenever we constructed cells in $\satur{\infty}{\equi{}{X} \cup \degg{X}}$, they were actually equivalences. In particular, in a representable diagrammatic set
\begin{enumerate}
	\item all degenerate cells are equivalences;
	\item there are ``unbiased'' compositors for spherical diagrams $c: \compos{x} \celto x$, $c': x \celto \compos{x}$, which are equivalences,
	\item units on spherical diagrams are equivalences.
\end{enumerate}
\end{remark}

\begin{remark}
We obtain an equivalent definition of RDS by requiring composition horns $\Lambda \incl W, x: \Lambda \to X$, where $\Lambda$ is $n$\nbd dimensional, to have fillers only when $\Lambda$ contains two $n$\nbd cells (``there are binary weak composites''), or when it is non-atomic and unsplittable. Indeed, 
\begin{itemize}
	\item when $\Lambda$ is an atom, we can pick the degenerate cell $p_\Lambda^*x: x \celto x$ as a filler;
	\item when $\Lambda$ has a binary split $\Lambda_1 \cup \Lambda_2$, for $x_i := \restr{x}{\Lambda_i}$, we can inductively construct compositors $h_i: \compos{x_i} \celto x_i$ and then compose them with a binary compositor $h: \compos{x} \celto \compos{x_1} \cup \compos{x_2}$. 
\end{itemize}
\emph{A priori}, we obtain ``compositors'' which are only in $\satur{\infty}{\equi{}{X} \cup \degg{X}})$ rather than $\equi{}{X}$, but it can be checked that the proofs all go through anyway.
\end{remark}

Theorem \ref{thm:equiv_closure} implies all the expected properties of equivalences: we have a sequence of easily proved consequences.

\begin{dfn} \label{dfn:equiv_relation}
Let $x, y$ be two parallel spherical $n$\nbd diagrams in a diagrammatic set $X$. We write $x \simeq y$, and say that $x$ is \emph{equivalent} to $y$, if there exists an $(n+1)$\nbd equivalence $h: x \celto y$.
\end{dfn}

\begin{prop}
Let $X$ be a representable diagrammatic set. Then $\simeq$ is an equivalence relation on spherical diagrams of $X$.
\end{prop}
\begin{proof}
By representability, there is a unit $\thin{x}: x \celto x$ on every spherical diagram $x$ in $X$, and Theorem \ref{thm:equiv_closure} implies that units are equivalences. This proves reflexivity.

If $e: x \celto y$ is an $n$\nbd equivalence, we can divide it by a unit $\thin{x}: x \celto x$, and obtain a cell $e^*: y \celto x$ together with an equivalence $h: \thin{x} \celto e\cp{n-1}e^*$. By Theorem \ref{thm:equiv_closure}, $e^*$ is an equivalence. This proves symmetry.

Finally, if $e: x \celto y$ and $e': y \celto z$ are both $n$\nbd equivalences, then $e \cp{n-1} e': x \celto z$ is a spherical $n$\nbd diagram. This has a weak composite, which by Theorem \ref{thm:equiv_closure} is an equivalence. This proves transitivity.
\end{proof}

\begin{cor}
Let $x$ be a spherical diagram in a representable diagrammatic set. Then weak composites of $x$ and units on $x$ are unique up to equivalence.
\end{cor}
\begin{proof}
A cell $\compos{x}$ is a weak composite of the spherical diagram $x$ if and only if $\compos{x} \simeq x$. Thus if $\compos{x}'$ is another weak composite of $x$, we have $\compos{x} \simeq \compos{x}'$. In particular, units on $x$, defined as weak composites of $p_U^*x$, are unique up to equivalence.
\end{proof}

Transitivity holds in a more general sense, compatibly with the spherical subdiagram relation.

\begin{prop} \label{prop:equi-subst}
Let $x, y, y', z$ be spherical diagrams in a representable diagrammatic set. Suppose $x \simeq y$, $y \subsph y'$, and $y' \simeq z$. Then $y'[x/y] \simeq z$.
\end{prop}
\begin{proof}
There are equivalences $e: x \celto y$ and $e': y' \celto z$. The union $e \cup e'$ along $\bord{}{+}e \subsph \bord{}{-}e'$ is a spherical diagram with a weak composite $\compos{e \cup e'}: y'[x/y] \celto z$, which by Theorem \ref{thm:equiv_closure} is an equivalence.
\end{proof}

\begin{prop}
Let $x, y$ be parallel $n$\nbd cells in a representable diagrammatic set. If $x \simeq y$ and $x$ is an equivalence, then $y$ is an equivalence.
\end{prop}
\begin{proof}
By symmetry, there are equivalences $e: x \celto y$ and $e': y \celto x$. Since $x$ and $y$ are parallel, any division horn for $x$ is also a division horn for $y$. Given an equivalence filler $h$ for a division horn for $x$, such that $x \subsph \bord{}{\alpha}h$, we can weakly compose $h$ with $e$ (if $\alpha = +$) or $e'$ (if $\alpha = -$) to obtain an equivalence filler for the corresponding division horn for $y$. 
\end{proof}

In a representable diagrammatic set, we can also recognise units by the fact that they act like units for weak composition. 

\begin{prop}
Let $1_x: x \celto x$ be an $n$\nbd cell in a representable diagrammatic set. The following are equivalent:
\begin{enumerate}[label=(\alph*)]
	\item $1_x$ is a unit on $x$;
	\item for all $n$\nbd cells $y, z$ with $x \subsph \bord{}{-}y$ and $x \subsph \bord{}{+}z$, we have
	\begin{equation*}
		(1_x \cup y) \simeq y, \quad \quad (z \cup 1_x) \simeq z.
	\end{equation*}
\end{enumerate}
\end{prop}
\begin{proof}
The implication from \emph{(a)} to \emph{(b)} follows from the fact that unitors are equivalences. For the converse, we know that there is a unit $\thin{x}'$ which is a weak composite of $p_U^*x$. Then
\begin{equation*}
	1_x \simeq (1_x \cp{n-1} \thin{x}') \simeq \thin{x}' \simeq p_U^*x,
\end{equation*}
so $1_x$ is also a weak composite of $p_U^*x$. 
\end{proof}

\subsection{Morphisms of representables} \label{sec:morphrep}

Diagrammatic sets have degeneracies as structure which is strictly preserved by all morphisms. However, for a morphism of RDSs to qualify as a functor of weak $\omega$\nbd categories, we would expect it to also preserve composites, at least weakly. We will show that no additional condition is required: every morphism of representables in $\dgmset$ sends equivalences to equivalences, hence weak composites to weak composites. For this we need an alternative characterisation of equivalences.

\begin{dfn}
Let $e: x \celto y$ be an $n$\nbd cell in a diagrammatic set $X$. Let $x$ be of shape $U$ and $y$ of shape $V$. We say that $e$ is \emph{weakly invertible} if there exists an $n$\nbd cell $e^*: y \celto x$ and weakly invertible $(n+1)$\nbd cells $h: e \cp{n-1} e^* \celto p_U^*x$ and $h': e^* \cp{n-1} e \celto p_V^*y$. In this case, $e^*$ is called a \emph{weak inverse} of $e$.

We write $\invrt{X}$ for the set of invertible cells of $X$.
\end{dfn}

\begin{remark}
This is another coinductive definition, with the following coinductive proof principle: for all subsets $A \subseteq S$ of the set $S$ of cells of $X$, let
\begin{align*}
\mathcal{I}(A) := \; & \{e: x \celto y \,|\, \text{ there exist } (e^*: y \celto x) \in S \\
	& \text{and } (h: e \cp{n-1} e^* \celto p_U^*x), (h': e^* \cp{n-1} e \celto p_V^*y) \in A \}.
\end{align*}
If $A \subseteq \mathcal{I}(A)$, then $A \subseteq \invrt{X}$. 
\end{remark}

\begin{remark}
Unlike equivalences, weakly invertible cells obviously determine a symmetric relation on spherical diagrams: if $e: x \celto y$ is weakly invertible, then a weak inverse $e^*: y \celto x$ is also weakly invertible.
\end{remark}

\begin{lem} \label{lem:equi_in_invrt}
Let $X$ be a representable diagrammatic set. Then $\equi{}{X} \subseteq \invrt{X}$.
\end{lem}
\begin{proof}
Let $\thin{x}$ and $\thin{y}$ be units on spherical diagrams $x$ and $y$ of shapes $U$ and $V$, respectively. Suppose $e: x \celto y$ is an $n$\nbd equivalence. Dividing $\thin{x}$ and $\thin{y}$ by $e$, we get $n$\nbd cells $e^*: y \celto x$ and $e': y \celto x$, together with equivalences $h: e \cp{n-1} e^* \celto 1_x$ and $h': e' \cp{n-1} e \celto 1_y$. 

Weakly precomposing the first with an equivalence $1_x \celto p_U^*x$ gives an equivalence $e \cp{n-1} e^* \celto p_U^*x$. Moreover, by the results of the previous section we have
\begin{equation*}
	e' \simeq (e' \cp{n-1} 1_x) \simeq (e' \cp{n-1} e \cp{n-1} e^*) \simeq (1_y \cp{n-1} e^*) \simeq e^*,
\end{equation*}
so weakly postcomposing $h'$ with an equivalence $1_y \celto p_V^*y$ and precomposing with an equivalence $e^* \celto e'$ we obtain an equivalence $e^* \cp{n-1} e \celto p_V^*y$. This proves that $\equi{}{X} \subseteq \mathcal{I}(\equi{}{X})$, hence that $\equi{}{X} \subseteq \invrt{X}$.
\end{proof}

The converse requires slightly more effort.

\begin{lem} \label{lem:invrt_closure}
Let $X$ be a representable diagrammatic set. Then $\invrt{X}$ is closed under weak composition.
\end{lem}
\begin{proof}
Let $e$ be a spherical $n$\nbd diagram in $X$, and let $\compos{e}: x \celto y$ be a weak composite. Write $e = f_1 \cp{n-1} \ldots \cp{n-1} f_m$, where each $f_i$ contains a single $n$\nbd cell $e_i: x_i \celto y_i$ of shape $U_i$, and suppose the $e_i$ all have weak inverses $e^*_i$. Then $e^* := e_m^* \cup \ldots \cup e_1^*$ is also a spherical diagram, which has a weak composite $\compos{e^*}: y \celto x$. We also let
\begin{equation*}
	\tilde{e}_j := e_1 \cup \ldots \cup e_j, \quad \quad \tilde{e}^*_j := e_j^* \cup \ldots \cup e_1^*.
\end{equation*}
Fix a unit $\thin{x}$ on $x$; we have
\begin{equation*}
	\compos{e} \cp{n-1} \compos{e^*} \simeq e \cp{n-1} e^* \simeq \thin{x} \cp{n-1} e \cp{n-1} e^*.
\end{equation*} 
Now, $e_m \cp{n-1} e_m^*$ is a spherical subdiagram of the latter diagram, so we can compose with a weakly invertible cell $h: (e_m \cp{n-1} e_m^*) \celto p_{U_m}^*x_m$ to obtain a cell
\begin{equation*}
	\thin{x} \cp{n-1} e \cp{n-1} e^* \celto \thin{x} \cp{n-1} (\tilde{e}_{m-1} \cup p_{U_m}^*x_m \cup \tilde{e}_{m-1}^*),
\end{equation*}
Since $1_x \cup \tilde{e}_{m-1}$ is a spherical subdiagram, from $(1_x \cup \tilde{e}_{m-1}) \cup p_{U_m}^*x_m \simeq (1_x \cup \tilde{e}_{m-1})$ we obtain
\begin{equation*}
	1_x \cp{n-1} (e' \cup p_{U_m}^*x_m \cup e'^*) \simeq 1_x \cup (\tilde{e}_{m-1} \cp{n-1} \tilde{e}_{m-1}^*).
\end{equation*}
Next, $e_{m-1} \cp{n-1} e_{m-1}^*$ is a spherical subdiagram. Proceeding like before, we can obtain cells
\begin{equation*}
	\thin{x} \cup (\tilde{e}_{j} \cp{n-1} \tilde{e}_{j}^*) \celto \thin{x} \cup (\tilde{e}_{j-1} \cp{n-1} \tilde{e}_{j-1}^*)
\end{equation*}
for $j = 1,\ldots,m$. Finally, from $e_1 \cp{n-1} e_1^* \celto p_{U_1}^*x_1$, and $1_x \cup p_{U_1}^*x_1 \simeq 1_x \simeq p_{U}^*x$, we can construct a cell
\begin{equation*}
	\compos{e} \cp{n-1} \compos{e^*} \celto p_U^*x
\end{equation*}
as a weak composite of weakly invertible cells and of equivalences, which by Lemma \ref{lem:equi_in_invrt} are also weakly invertible. The construction of a cell $\compos{e^*} \cp{n-1} \compos{e} \celto p_V^*y$ is dual, and we conclude by coinduction.
\end{proof}

\begin{prop} \label{prop:equi_equals_invrt}
Let $X$ be a representable diagrammatic set. Then $\invrt{X} = \equi{}{X}$.
\end{prop}
\begin{proof}
We have already proved one inclusion. For the other, let $e: x \celto y$ be a weakly invertible $n$\nbd cell of shape $U$, and let $\Lambda \incl W, \lambda: \Lambda \to X$ be a division horn for $e$ at $\bord{}{+}U$, that is, $\restr{\lambda}{W_0}$ is an $n$\nbd cell $z$ with $x \subsph \bord{}{-}z$. 

Suppose $W_0 = \bord{}{-}W$ and let $V := \bord{}{-}U$. There is a unitor $\ell_z^{V}: z \celto (p_V^*x \cup z)$, and by assumption there are a weak inverse $e^*: y \celto x$ and a weakly invertible cell $h: p_V^*x \celto (e \cp{n-1} e^*)$. Now, $\ell_z^{V} \cup h: z \celto (e \cp{n-1} e^*) \cup z$ is a spherical diagram, and $e^* \cup z$ is a spherical subdiagram of its output boundary. Picking a weak composite $z'$ and compositor $c: (e^* \cup z) \celto z'$, we finally obtain a spherical $(n+1)$\nbd diagram
\begin{equation*}
	\ell_z^{V} \cup h \cup c: z \celto e \cup z'
\end{equation*}
whose $(n+1)$\nbd cells are all weakly invertible. By Lemma \ref{lem:invrt_closure}, its weak composite is weakly invertible, and works as a filler of $\lambda$; any weak inverse fills the dual horn.

The case of division horns for $e$ at $\bord{}{-}U$ is dual. This proves that $\invrt{X} \subseteq \mathcal{F}(\invrt{X})$, so by coinduction $\invrt{X} \subseteq \equi{}{X}$.
\end{proof}

\begin{thm} \label{thm:equiv_preserve}
Let $f: X \to Y$ be a morphism of RDSs. Then $f$ maps equivalences to equivalences.
\end{thm}
\begin{proof}
By Proposition \ref{prop:equi_equals_invrt}, it suffices to show that $f$ preserves weakly invertible cells. Suppose $e: x \celto y$ is weakly invertible, let $e^*: y \celto x$ be a weak inverse, and $h: e \cp{n-1} e^* \celto p_U^*x, h': e^* \cp{n-1} e \celto p_V^*y$ be weakly invertible. 

Because $f(p_U^*x) = p_U^*f(x)$ and $f(p_V^*y) = p_V^*f(y)$, we have cells
\begin{equation*}
	f(h): f(e) \cp{n-1} f(e^*) \celto p_U^*f(x), \quad \quad f(h'): f(e^*) \cp{n-1} f(e) \celto p_V^*f(y)
\end{equation*}
in $Y$. This proves that $f(\invrt{X}) \subseteq \mathcal{I}(f(\invrt{X}))$, and by coinduction we conclude that $f(\invrt{X}) \subseteq \invrt{Y}$. 
\end{proof}

\begin{cor}
Let $f: X \to Y$ be a morphism of RDSs, and let $x$ be a spherical diagram in $X$. Then:
\begin{enumerate}[label=(\alph*)]
	\item for all weak composites $\compos{x}$ of $x$, $f(\compos{x})$ is a weak composite of $f(x)$;
	\item for all units $\thin{x}$ on $x$, $f(\thin{x})$ is a unit on $f(x)$.
\end{enumerate}	
\end{cor}
\begin{proof}
By Theorem \ref{thm:equiv_preserve}, if $c: \compos{x} \celto x$ is an equivalence exhibiting $\compos{x}$ as a weak composite of $x$, then $f(c): f(\compos{x}) \celto f(x)$ exhibits $f(\compos{x})$ as a weak composite of $x$. If $e: p_U^*x \celto 1_x$ is an equivalence exhibiting $1_x$ as a unit on $x$, then $f(e): p_U^*f(x) \celto f(1_x)$ exhibits $f(1_x)$ as a unit on $f(x)$.
\end{proof}

We have the following consequence (which may hold more generally).

\begin{prop}
Let $\{X_i\}_{i \in I}$ be a family of RDSs, and let $x = (x_i)_{i \in I}$ be a cell of $\prod_{i \in I} X_i$. Then $x$ is an equivalence if and only if each projection $x_i$ is an equivalence.
\end{prop}
\begin{proof}
One implication is Proposition \ref{prop:equiv_prod}. For the other, suppose $x$ is an equivalence. By Proposition \ref{prop:rep_product_closure} $\prod_{i \in I} X_i$ is representable, therefore each projection $p_j: \prod_{i \in I} X_i \to X_j$ maps $x$ to an equivalence.
\end{proof}

Thus we have the following candidate for a category of weak $\omega$\nbd categories and weak functors.
\begin{dfn}
We write $\rdgmset$ for the full subcategory of $\dgmset$ on the representable diagrammatic sets.
\end{dfn}

We leave it to future work to investigate any further structure on $\rdgmset$, including higher morphisms, model structures, or monoidal and closed structures. 

We briefly note that, given the properties of the relation $\simeq$, we have an obvious candidate for a notion of weak equivalence in $\rdgmset$, based on the $\omega$\nbd weak equivalences of \cite{lafont2010folk}.

\begin{dfn}
Let $f: X \to Y$ be a morphism of RDSs. We say that $f$ is a \emph{weak equivalence} if 
\begin{enumerate}
	\item for all 0-cells $y$ in $Y$, there exists a 0-cell $x$ in $X$ such that $f(x) \simeq y$, and
	\item for all parallel spherical $n$\nbd diagrams $x_1, x_2$ in $X$ and $(n+1)$\nbd cells $y: f(x_1) \celto f(x_2)$ in $Y$,  there exists an $(n+1)$\nbd cell $x: x_1 \celto x_2$ in $X$ such that $f(x) \simeq y$.
\end{enumerate}
\end{dfn}

Some of the proofs of \cite{lafont2010folk} translate immediately.
\begin{prop}
Let $f: X \to Y$ be a weak equivalence of RDSs. If $x, x'$ are parallel spherical $n$\nbd diagrams in $X$ such that $f(x) \simeq f(x')$, then $x \simeq x'$.
\end{prop}
\begin{proof}
Essentially the same as \cite[Lemma 6]{lafont2010folk}.
\end{proof}

\begin{prop}
Let $f: X \to Y$ and $g: Y \to Z$ be weak equivalences of RDSs. Then:
\begin{enumerate}[label=(\alph*)]
	\item if $f$ and $g$ are weak equivalences, then so is $f;g$, and
	\item if $f;g$ and $g$ are weak equivalences, then so is $f$.
\end{enumerate}
\end{prop}
\begin{proof}
Essentially the same as \cite[Lemma 7 and 8]{lafont2010folk}.
\end{proof}

Adapting the proof of the remaining 2-out-of-3 property seems to require more effort, and we leave it to future work.

\begin{remark}
It is not hard to see that, in a non-representable diagrammatic set, equivalences as we defined them do not satisfy the same nice properties as in a representable: there is no reason, for example, that $\simeq$ should be transitive, since a spherical diagram of equivalences may not have a weak composite.

It seems reasonable, however, that there could exist a more general class of equivalences --- perhaps based on equivalence \emph{diagrams} rather than cells --- which satisfies them in every diagrammatic set, and coincides with our class in the representable case. For example, we could have required as part of the definition that composition horns of ternary atoms have fillers when their restriction to $W_+$ or $W_-$ is an equivalence. Moreover, we could have had both ``problems'' and ``solution'' of filling problems be spherical diagrams, and not just cells. 

Ultimately, we decided to focus on what seemed the shortest and simplest way to representables, but it is worth investigating its generalisations. In particular, it would be desirable to extend the definition of weak equivalence to morphisms between diagrammatic sets. Indeed, we conjecture that there is a model structure on $\dgmset$ whose
\begin{enumerate}
	\item fibrant objects are the representables, 
	\item cofibrant objects are those satisfying the Eilenberg-Zilber property, and
	\item weak equivalences extend weak equivalences of representables.
\end{enumerate}
\end{remark}

\subsection{Nerves of strict $\omega$-categories} \label{sec:nerves}

In \cite{hadzihasanovic2018combinatorial}, we observed that the construction $P \mapsto (\molec{}{P})^*$ is functorial on inclusions of directed complexes. In fact, the functoriality extends to maps.

\begin{prop}
The assignment $P \mapsto (\molec{}{P})^*$ for regular directed complexes extends to a faithful functor $(\molec{}{-})^*: \rdcpx \to \omegacat$.
\end{prop}
\begin{proof}
Same as \cite[Proposition 5.28]{hadzihasanovic2018combinatorial}.
\end{proof}

\begin{remark}
The $\omega$\nbd categories $(\molec{}{P})^*$ are composition-generated by the atoms of $P$; the functors $(\molec{}{P})^* \to (\molec{}{Q})^*$ in the subcategory are precisely those that map generators to generators, possibly of lower dimension.
\end{remark}

We can assume that $\atom$ is a small category, since it has countably many objects up to unique isomorphism, and finitely many morphisms between any two of them. Moreover, $\dgmset$ is locally small, and $\omegacat$ is cocomplete, so we are in the conditions of \cite[Theorem X.4.1-2]{maclane1971cats}: the left Kan extension of the embedding $\atom \incl \omegacat$ along the Yoneda embedding $\atom \incl \dgmset$ exists and is computed by the coend
\begin{equation*}
	X \mapsto \hatto{X} := \int^{U \in \atom} (\molec{}{U})^* \times X(U).
\end{equation*}
This has a right adjoint $\nerve{}$ defined by
\begin{equation*}
	\nerve{X}(-) := \homset{\omegacat}((\molec{}{-})^*,X),
\end{equation*}
that is, the cells of shape $U$ in $\nerve{X}$ are the functors $x: (\molec{}{U})^* \to X$. We call this the diagrammatic \emph{nerve} of the strict $\omega$\nbd category $X$.

Taking the left Kan extension of $\hatto{-}$ along $\atom \incl \rdcpx$, we also obtain a functor $\hatto{-}: \rdcpx \to \omegacat$.

\begin{prop}
The functor $\nerve{}: \omegacat \to \dgmset$ is faithful.
\end{prop}
\begin{proof}
The underlying $\omega$\nbd graph functor $\mathcal{U}: \omegacat \to \globset$ factors as the functor $\nerve{}: \omegacat \to \dgmset$ followed by the presheaf restriction $-_\mathbf{O}: \dgmset \to \globset$. Because $\mathcal{U}$ is faithful, so is $\nerve{}$.
\end{proof}

We recall the definition of polygraphs.

\begin{dfn}
Let $X$ be a partial $\omega$\nbd category and $n \in \mathbb{N}$. The \emph{$n$\nbd skeleton} $\skel{n}{X}$ of $X$ is the partial $\omega$\nbd category whose underlying $\omega$\nbd graph is
\begin{equation*}
\begin{tikzpicture}
	\node[scale=1.25] (0) at (0,0) {$X_0$};
	\node[scale=1.25] (1) at (2,0) {$\ldots$};
	\node[scale=1.25] (2) at (4,0) {$X_n$};
	\node[scale=1.25] (3) at (6,0) {$\idd{}(X_n)$};
	\node[scale=1.25] (4) at (8,0) {$\ldots$};
	\node[scale=1.25] (5) at (10,0) {$\idd{m}(X_n)$};
	\node[scale=1.25] (6) at (12,0) {$\ldots$};
	\draw[1c] (1.west |- 0,.15) to node[auto,swap] {$\bord{}{+}$} (0.east |- 0,.15);
	\draw[1c] (1.west |- 0,-.15) to node[auto] {$\bord{}{-}$} (0.east |- 0,-.15);
	\draw[1c] (2.west |- 0,.15) to node[auto,swap] {$\bord{}{+}$} (1.east |- 0,.15);
	\draw[1c] (2.west |- 0,-.15) to node[auto] {$\bord{}{-}$} (1.east |- 0,-.15);
	\draw[1c] (3.west |- 0,.15) to node[auto,swap] {$\bord{}{+}$} (2.east |- 0,.15);
	\draw[1c] (3.west |- 0,-.15) to node[auto] {$\bord{}{-}$} (2.east |- 0,-.15);
	\draw[1c] (4.west |- 0,.15) to node[auto,swap] {$\bord{}{+}$} (3.east |- 0,.15);
	\draw[1c] (4.west |- 0,-.15) to node[auto] {$\bord{}{-}$} (3.east |- 0,-.15);
	\draw[1c] (5.west |- 0,.15) to node[auto,swap] {$\bord{}{+}$} (4.east |- 0,.15);
	\draw[1c] (5.west |- 0,-.15) to node[auto] {$\bord{}{-}$} (4.east |- 0,-.15);
	\draw[1c] (6.west |- 0,.15) to node[auto,swap] {$\bord{}{+}$} (5.east |- 0,.15);
	\draw[1c] (6.west |- 0,-.15) to node[auto] {$\bord{}{-}$} (5.east |- 0,-.15);
	\node[scale=1.25] at (12.5,-.25) {$,$};
\end{tikzpicture}
\end{equation*}
that is, the restriction of $X$ to cells $x$ with $\dmn{x} \leq n$; unit and composition operations are restricted as appropriate. 

There is an obvious inclusion $\skel{n}X \hookrightarrow X$, which factors through $\skel{m}{X} \hookrightarrow X$ for all $m > n$. Any partial $\omega$\nbd category is the sequential colimit of its skeleta.
\end{dfn}

For each $n \in \mathbb{N}$, let $\hatto{O^n}$ be the $n$\nbd globe as an $\omega$\nbd category: $\hatto{O^n}$ has two $k$\nbd dimensional cells $\underline{k}^+, \underline{k}^-$ for each $k < n$, and a single $n$\nbd dimensional cell $\underline{n}$, such that $\bord{k}{\alpha}\underline{n} = \underline{k}^\alpha$ for all $k < n$. Let $\bord{}{}\hatto{O^n}$ be the $(n-1)$\nbd skeleton of $\hatto{O^n}$. There is a bijection between $n$\nbd cells $x$ of a partial $\omega$\nbd category $X$ and functors $x: \hatto{O^n} \to X$, and we will identify the two. We write $\bord{}{}x$ for the precomposition of $x$ with the inclusion $\bord{}{}\hatto{O^n} \hookrightarrow \hatto{O^n}$.

For a family $\{X_i\}_{i \in I}$ of $\omega$\nbd categories, let $\coprod_{i \in I} X_i$ be its coproduct in $\omegacat$; the sets of $n$\nbd cells and the boundary, unit, and composition operations of the coproduct are all induced pointwise by coproducts of sets and functions. Given a family of functors $\{f_i: X_i \to Y\}_{i \in I}$, let $(f_i)_{i \in I}: \coprod_{i \in I}X_i \to Y$ be the functor produced by the universal property of coproducts.

\begin{dfn}
A \emph{polygraph} is an $\omega$\nbd category $X$ together with families $\{\nondeg{n}{X}\}_{n \in \mathbb{N}}$ of $n$\nbd dimensional cells of $X$, such that, for all $n$,
\begin{equation*}
\begin{tikzpicture}[baseline={([yshift=-.5ex]current bounding box.center)}]
	\node[scale=1.25] (0) at (0,2) {$\displaystyle \coprod_{x\in\nondeg{n}{X}} \bord{}{} \hatto{O^n}$};
	\node[scale=1.25] (1) at (3.5,2) {$\displaystyle\coprod_{x\in\nondeg{n}{X}} \hatto{O^n}$};
	\node[scale=1.25] (2) at (0,0) {$\skel{n-1}{X}$};
	\node[scale=1.25] (3) at (3.5,0) {$\skel{n}{X}$};
	\draw[1c] (0) to node[auto,swap] {$(\bord{}{}x)_{x\in\nondeg{n}{X}}$} (2);
	\draw[1c] (1) to node[auto] {$(x)_{x\in\nondeg{n}{X}}$} (3);
	\draw[1cinc] (0) to (1);
	\draw[1cinc] (2) to (3);
	\draw[edge] (2.5,0.2) to (2.5,0.8) to (3.3,0.8);
\end{tikzpicture}
\end{equation*}
is a pushout diagram in $\omegacat$. The cells in $\nondeg{n}{X}$ are called $n$\nbd dimensional \emph{generators} of $X$.

A \emph{map} of polygraphs is a functor $f: X \to Y$ of $\omega$\nbd categories that sends $n$\nbd dimensional generators of $X$ to $n$\nbd dimensional generators of $Y$; that is, $f$ restricts to, and is essentially determined by a sequence of functions $\{f_n: \nondeg{n}{X} \to \nondeg{n}{Y}\}_{n \in \mathbb{N}}$. Polygraphs and their maps form a category $\pol$.
\end{dfn}

\begin{dfn} \label{dfn:freegen}
Let $P$ be a directed complex. We say that $P$ is \emph{freely generating} if $(\molec{}{P})^*$ admits the structure of a polygraph, whose $n$\nbd dimensional generators are the $n$\nbd dimensional atoms of $P$.
\end{dfn}

\begin{prop}
Let $P$ be a freely generating regular directed complex. Then $\hatto{P}$ and $(\molec{}{P})^*$ are isomorphic.
\end{prop}
\begin{proof}
Same as \cite[Proposition 6.8]{hadzihasanovic2018combinatorial}.
\end{proof}

The following conjecture is analogous to, and implies \cite[Conjecture 6.6]{hadzihasanovic2018combinatorial}.
\begin{conj} \label{conj:freegen}
Every regular directed complex is freely generating.
\end{conj}

\begin{remark} \label{rmk:ifconjfalse}
Some important aspects of the relation between representable diagrammatic sets and strict $\omega$\nbd categories are conditional to the validity of this conjecture, which we believe to be true. If, on the other hand, it turned out to be false, we see two ways of salvaging the results in this section.
\begin{enumerate}
	\item It may be possible to restrict to a freely generating subclass of regular directed complexes which is closed under the constructions used in the definitions and proofs relative to RDSs: mainly, substitutions and the $O(-)$ construction. Then one could restrict to this class to develop an equivalent, better-behaved theory. We note that loop-free directed complexes, which are known to be freely generating, are \emph{not} closed under these (in \cite[Appendix A.2]{henry2019non}, S.\ Henry suggested that the failure to recognise this is one of the critical flaws of \cite{kapranov1991infty}).  
	\item The conjecture is likely to be equivalent to the statement that any two \emph{regular polyplexes} in the sense of \cite{henry2018regular} have isomorphic underlying polygraphs, and these are uniquely determined by their oriented graded poset of generators. A counterexample may suggest what information is missing from the oriented graded poset in order to determine a unique regular polyplex, and one could try working with a richer combinatorial structure containing this additional information.
\end{enumerate}
On the other hand, failures of presentations of $\omega$\nbd categories, such as directed complexes, to be freely generating have so far been linked to some kinds of topological degeneracies in cell shapes, which do not apply to regular directed complexes. A counterexample to the conjecture may even be interpreted as a failure of the combinatorics of $\omega$\nbd categorical (globular) composition to identify what should be identical ``directed spaces''. 
\end{remark}

If Conjecture \ref{conj:freegen} holds and $U$ is an $n$\nbd molecule, then $\hatto{U}$ can be identified with $(\molec{}{U})^*$, and $U$ with an $n$\nbd cell of $\hatto{U}$, which a functor $x: \hatto{U} \to X$ maps to an $n$\nbd cell $x(U)$ of $X$.

\begin{lem} \label{lem:omega_deg_equiv}
Let $U$ be an $n$\nbd atom, $X$ an $\omega$\nbd category, and $x: \hatto{U} \to X$ a functor. Assume Conjecture \ref{conj:freegen}. If $x(U)$ is a degenerate $n$\nbd cell of $X$, then the transpose $x: U \to \nerve{X}$ is an $n$\nbd equivalence of $\nerve{X}$.
\end{lem}
\begin{proof}
If $x(U)$ is degenerate, it means that $x(\bord{}{+}U) = x(\bord{}{-}U)$, and $x(U) = \idd{}(x(\bord{}{\alpha}U))$. Let $\Lambda \incl W, \lambda: \Lambda \to \nerve{X}$ be a division horn for $x$ at $\bord{}{+}U$, that is, $\restr{\lambda}{W_0}$ is an $n$\nbd cell $z$ with $\bord{}{-}x \subsph \bord{}{-}z$. Then $z$ has a transpose $z: \hatto{(W_0)} \to X$, such that $z(W_0)$ is an $n$\nbd cell of $X$. 

We define a functor $h: \hatto{W} \to X$ which sends $W$ to $\idd{}(z(W_0))$, is equal to $z$ on $W_0$ and to $x$ on $W_+$, and sends $W_-$ to $z(W_0)$ again. This defines $h$ on all generators of $\hatto{W}$, and is well-defined: $z(\bord{}{-}W_0)$ has an expression as a composite involving $x(\bord{}{-}U)$, and $h(\bord{}{-}W_-)$ has the same expression with $x(\bord{}{+}U)$ replacing $x(\bord{}{-}U)$, which is equal in $X$. Finally, $h(W_+ \cup W_-)$ is a composite of $z(W_0)$ with a degenerate cell, hence it is equal to $z(W_0)$, so it makes sense to ask that $h(W)$ be a unit on $z(W_0)$. 

The transpose $h: W \to \nerve{X}$ fills $\lambda$, and the case of a division horn for $x$ at $\bord{}{-}U$ is dual. It follows that, letting $A := \{x: U \to \nerve{X} \,|\, x(U) \text{ is degenerate}\}$, we have $A \subseteq \mathcal{F}(A)$, and by coinduction $A \subseteq \equi{}{X}$.
\end{proof}

\begin{prop} \label{prop:omega_representable}
Let $X$ be an $\omega$\nbd category, and assume Conjecture \ref{conj:freegen}. Then $\nerve{X}$ is representable.
\end{prop}
\begin{proof}
Let $\Lambda \incl W, x: \Lambda \to \nerve{X}$ be a composition horn. The transpose $x: \hatto{\Lambda} \to X$ is a pasting diagram in $X$, so $x(\Lambda)$ is a cell of $X$. We define a functor $c(x): \hatto{W} \to X$ which sends $W$ to $\idd{}x(\Lambda)$, is equal to $x$ on $\Lambda$, and sends $W_0$ to $x(\Lambda)$. Its transpose $c(x): W \to \nerve{X}$ fills the horn, and by Lemma \ref{lem:omega_deg_equiv} it is an equivalence.
\end{proof}

Thus, conditional to Conjecture \ref{conj:freegen}, we can see $\nerve{}$ as a faithful functor from $\omegacat$ to $\rdgmset$, including strict $\omega$\nbd categories as a subcategory of weak $\omega$\nbd categories. This functor is far from being full, since morphisms of diagrammatic sets do not know about composition in an $\omega$\nbd category: for example, two isomorphic cells of which only one is a composite may be exchanged.

\begin{exm}
Let $X$ be a 2-category generated by 0-cells $x, y, z$, 1-cells $a: x \celto y$, $b: y \celto z$, and $c,d: x \celto z$, and 2-cells $h: c \celto d$, $h^*: d \celto c$, subject to equations
\begin{equation*}
	a \cp{0} b = c, \quad \quad h \cp{1} h^* = \idd{}c, \quad \quad h^* \cp{1} h = \idd{}d.
\end{equation*}
The generators are the only non-degenerate cells of $X$. For each atom $U$, we can define an involution $s$ on $\nerve{X}(U)$ such that a functor $x: \hatto{U} \to X$ maps a 1\nbd cell to $c$ if and only if $s(x): \hatto{U} \to X$ maps it to $d$. Indeed, this determines $s(x)$ completely on the 1\nbd dimensional generators; then, for each 2\nbd dimensional generator, there is always only one possible choice given the image of its boundaries.

This determines a morphism $s: \nerve{X} \to \nerve{X}$ of representable diagrammatic sets, which is not the image of an endofunctor of $X$, because 
\begin{equation*}
	s(a \cp{0} b) = d \neq c = s(a) \cp{0} s(b).
\end{equation*}
\end{exm}

However, we defined representability as a \emph{property} of diagrammatic sets, the existence of a pair of compositors for every spherical diagram of two cells. We can also treat compositors as a fixed \emph{structure} on a representable diagrammatic set.

\begin{dfn}
An \emph{algebraic} RDS is a representable diagrammatic set $X$ equipped with a weak composite $\compos{x}$ and a pair of compositors $c(x): \compos{x} \celto x$, $c'(x): x \celto \compos{x}$ for every spherical diagram $x$ in $X$ whose shape is either
\begin{enumerate}
	\item a molecule with two maximal elements, or
	\item an unsplittable, non-atomic molecule.
\end{enumerate}

We write $\rdgmsetalg$ for the category of algebraic RDSs, together with \emph{strict} morphisms which respect the choice of compositors.
\end{dfn}

\begin{remark}
From a constructive point of view, it may be preferable to consider all RDSs as having an algebraic choice of compositors; in that case, the difference between $\rdgmset$ and $\rdgmsetalg$ is purely at the level of morphisms. The two perspectives are equivalent modulo choice.
\end{remark}

\begin{remark}
The compositors $c(x)$ and $c'(x)$ can always be chosen to be weak inverses of each other.
\end{remark}

As seen in the proof of Proposition \ref{prop:omega_representable}, $\nerve{X}$ has a canonical choice of compositors, where the weak composite of a spherical diagram is, in fact, its composite in $X$, and the compositor maps to a unit in $X$. Any functor of $\omega$\nbd categories clearly respects this choice, so assuming Conjecture \ref{conj:freegen}, $\nerve{}$ factors through a functor
\begin{equation*}
	\nervealg{}: \omegacat \to \rdgmsetalg.
\end{equation*}
We will prove that $\nervealg{}$ is full and faithful. 
 
\begin{cons}
Let $X$ be an algebraic RDS, $x: U \to X$ a spherical $n$\nbd diagram. For each merger tree $\mathfrak{T}$ of $U$ (Construction \ref{cons:mergertree}), we define an $n$\nbd cell $\compos{x}_\mathfrak{T}$ and $(n+1)$\nbd equivalences 
\begin{equation*}
	c_\mathfrak{T}(x): \compos{x}_\mathfrak{T} \celto x, \quad \quad c'_\mathfrak{T}(x): x \celto \compos{x}_\mathfrak{T},
\end{equation*}
by induction on the structure of $\mathfrak{T}$.
\begin{itemize}
	\item If $\mathfrak{T}$ is equal to its root, either $U$ is an atom, and we let 
		\begin{equation*}
			\compos{x}_\mathfrak{T} := x, \quad \quad c_\mathfrak{T}(x), c'_\mathfrak{T}(x) := p_U^*x: x \celto x,
		\end{equation*}
	or $U$ is non-atomic and unsplittable, and we let
		\begin{equation*}
			\compos{x}_\mathfrak{T} := \compos{x}, \quad \quad c_\mathfrak{T}(x) := c(x), c'_\mathfrak{T}(x) := c'(x).
		\end{equation*}
	\item Suppose $\mathfrak{T}$ branches into $\mathfrak{T}_1$ and $\mathfrak{T}_2$, labelled by molecules $U_1$ and $U_2$ with spherical boundary. Letting $x_1 := \restr{x}{U_1}, x_2 := \restr{x}{U_2}$, we have $(n+1)$\nbd equivalences 
	\begin{align*}
		c_\mathfrak{T_1}(x_1): \compos{x_1}_\mathfrak{T_1} \celto x_1, \quad \quad & c'_\mathfrak{T_1}(x_1): x_1 \celto \compos{x_1}_\mathfrak{T_1}, \\
		c_\mathfrak{T_2}(x_2): \compos{x_2}_\mathfrak{T_2} \celto x_2, \quad \quad & c'_\mathfrak{T_2}(x_2): x_2 \celto \compos{x_2}_\mathfrak{T_2}.
	\end{align*} 
	We define $\compos{x}_\mathfrak{T} := \compos{\compos{x_1}_\mathfrak{T_1} \cup \compos{x_2}_\mathfrak{T_2}}$, and 
	\begin{itemize}
		\item if $U_1, U_2$ are both atoms, then $c_\mathfrak{T}(x) := c(x)$ and $c'_\mathfrak{T}(x) := c'(x)$;
		\item if $U_1$ is an atom and $U_2$ is not, then 
		\begin{equation*}
			c_\mathfrak{T}(x) := \compos{c(x_1 \cup \compos{x_2}_\mathfrak{T_2}) \cup c_\mathfrak{T_2}(x_2)}, \quad c'_\mathfrak{T}(x) :=  \compos{c'_\mathfrak{T_2}(x_2) \cup c'(x_1 \cup \compos{x_2}_\mathfrak{T_2})};
		\end{equation*}
		\item if $U_2$ is an atom and $U_1$ is not, then
		\begin{equation*}
			c_\mathfrak{T}(x) := \compos{c(\compos{x_1}_\mathfrak{T_1} \cup x_2) \cup c_\mathfrak{T_1}(x_1)}, \quad c'_\mathfrak{T}(x) :=  \compos{c'_\mathfrak{T_1}(x_1) \cup c'(\compos{x_1}_\mathfrak{T_1} \cup x_2)};
		\end{equation*}
		\item if neither $U_1$ nor $U_2$ are atoms,
	\begin{align*}
		c_\mathfrak{T}(x) & := \compos{\compos{c(\compos{x_1}_\mathfrak{T_1} \cup \compos{x_2}_\mathfrak{T_2}) \cup c_\mathfrak{T_1}(x_1)} \cup c_\mathfrak{T_2}(x_2)}, \\
		c'_\mathfrak{T}(x) & := \compos{c'_\mathfrak{T_2}(x_2) \cup \compos{c'_\mathfrak{T_1}(x_1) \cup c'(\compos{x_1}_\mathfrak{T_1} \cup \compos{x_2}_\mathfrak{T_2})}}.
	\end{align*}
	\end{itemize}
\end{itemize}
\end{cons}

Intuitively, the construction produces a weak composite and a compositor for each spherical diagram of shape $U$, depending on a choice of sequential binary splits of $U$. These coincide with the structural compositors of $X$ for binary diagrams or $n$\nbd ary unsplittable diagrams.

\begin{prop} \label{prop:asso_preserve}
Let $f: X \to Y$ be a strict morphism of algebraic RDSs. For all spherical diagrams $x: U \to X$ and all merger trees $\mathfrak{T}$ for $U$, 
\begin{equation*}
	f(c_\mathfrak{T}(x)) = c_\mathfrak{T}(f(x)), \quad \quad f(c'_\mathfrak{T}(x)) = c'_\mathfrak{T}(f(x)).
\end{equation*}
\end{prop}
\begin{proof}
If $x$ is a cell, that is, $U$ is atomic, the statement is true because $f(p_U^*x) = p_U^*f(x)$. Otherwise, it is proved by a simple induction, using the fact that $f$ preserves the structural compositors.
\end{proof}

\begin{dfn}
Let $X$ be an algebraic RDS. We say that $X$ is \emph{associative} if, for each spherical diagram $x: U \to X$ in $X$, and each pair $\mathfrak{T}_1$, $\mathfrak{T}_2$ of merger trees for $U$, we have
\begin{equation*}
	\compos{x}_\mathfrak{T_1} = \compos{x}_\mathfrak{T_2}, \quad \quad c_\mathfrak{T_1}(x) = c_\mathfrak{T_2}(x), \quad \quad c'_\mathfrak{T_1}(x) = c'_\mathfrak{T_2}(x).
\end{equation*}
In an associative algebraic RDS, we write $\compos{x}, c(x), c'(x)$ for $\compos{x}_\mathfrak{T}, c_\mathfrak{T}(x), c'_\mathfrak{T}(x)$ where $\mathfrak{T}$ is an arbitrary merger tree.
\end{dfn}

\begin{prop} \label{prop:omega_associative}
Let $X$ be an $\omega$\nbd category, and assume Conjecture \ref{conj:freegen}. Then $\nervealg{X}$ is associative.
\end{prop}
\begin{proof}
Let $x: U \to \nervealg{X}$ be a spherical diagram with transpose $x: \hatto{U} \to X$. For any pair $\mathfrak{T}_1, \mathfrak{T}_2$ of merger trees for $U$, the functors $\compos{x}_\mathfrak{T_1}, \compos{x}_\mathfrak{T_2}: \hatto{\compos{U}} \to X$ are equal on the boundary, and both send $\compos{U}$ to $x(U)$, therefore $\compos{x}_\mathfrak{T_1} = \compos{x}_\mathfrak{T_2}$. 

Then $c_\mathfrak{T_1}(x), c_\mathfrak{T_2}(x): \hatto{(\compos{U} \celto U)} \to X$ are defined and equal on their boundaries, and send $\compos{U} \celto U$ to a composite of units, hence a unit. Thus $c_\mathfrak{T_1}(x) = c_\mathfrak{T_2}(x)$, and similarly $c'_\mathfrak{T_1}(x) = c'_\mathfrak{T_2}(x)$.
\end{proof}

\begin{cons} \label{cons:shell}
Let $U$ be an $n$\nbd molecule with spherical boundary. We construct an $n$\nbd molecule $\shell{U} \supsph U$ with spherical boundary, which we call a \emph{shell} of $U$, with the property that $\bord{}{\alpha}\shell{U}$ is isomorphic to $O^{n-1}$ for both $\alpha \in \{+,-\}$. Moreover, if $U$ and $V$ are $n$\nbd molecules with isomorphic, spherical boundaries, then
\begin{equation*}
	\shell{U} \incliso \shell{V}[U/V],
\end{equation*}
that is, $\shell{U}$ and $\shell{V}$ are isomorphic outside of $U, V$.

If $n = 0, 1$, then $\shell{U} := U$. If $n > 0$, we let $\shell{U}$ be the colimit
\begin{equation*} 
\begin{tikzpicture}[baseline={([yshift=-.5ex]current bounding box.center)}]
	\node[scale=1.25] (0) at (-2,-1.5) {$O^{n-1} \celto \shell{(\bord{}{-}U)}$};
	\node[scale=1.25] (1) at (0,0) {$\bord{}{-}U$};
	\node[scale=1.25] (2) at (2,-1.5) {$U$};
	\node[scale=1.25] (3) at (4,0) {$\bord{}{+}U$};
	\node[scale=1.25] (4) at (6,-1.5) {$\shell{(\bord{}{+}U)} \celto O^{n-1}$,};
	\draw[1cincl] (1) to node[auto,swap] {$j_1$} (0);
	\draw[1cinc] (1) to node[auto] {$\imath^-$} (2);
	\draw[1cincl] (3) to node[auto,swap] {$\imath^+$} (2);
	\draw[1cinc] (3) to node[auto] {$j_2$} (4);
\end{tikzpicture}
\end{equation*}
where $j_1$ and $j_2$ are the inclusions of $\bord{}{\alpha}U$ as a submolecule of $\shell{(\bord{}{\alpha}U)}$. 

The $n$\nbd molecules $O^{n-1} \celto \shell{(\bord{}{-}U)}$ and $\shell{(\bord{}{+}U)} \celto O^{n-1}$ are well-defined and have spherical boundary by the inductive hypothesis on the boundaries of the shell of an $(n-1)$\nbd molecule, and if the $\shell{(\bord{}{\alpha}U)}$ are $(n-1)$\nbd molecules containing $\bord{}{\alpha}U$ as a spherical submolecule, then $\shell{U}$ is an $n$\nbd molecule with $\bord{}{\alpha}\shell{U}$ isomorphic to $O^{n-1}$, and containing $U$ as a spherical subdiagram. The fact that $\shell{U}$ and $\shell{V}[U/V]$ are isomorphic is also evident.
\end{cons}

\begin{lem} \label{lem:minimal_suffice}
Let $X, Y$ be $\omega$\nbd categories and $f: \nervealg{X} \to \nervealg{Y}$ a strict morphism. Assuming Conjecture \ref{conj:freegen}, $f$ is uniquely determined by the functions
\begin{equation*}
	f: \nervealg{X}(O^n) \to \nervealg{Y}(O^n), \quad n \geq 0.
\end{equation*}
\end{lem}
\begin{proof}
By Proposition \ref{prop:omega_associative}, there are uniquely determined compositors $c(x), c'(x)$ for all spherical diagrams in $\nervealg{X}$ and $\nervealg{Y}$. For each spherical diagram $x: U \to \nervealg{X}$, we define an extension
\begin{equation*}
	\shell{x}: \shell{U} \to \nervealg{X}, \quad \quad x \subsph \shell{x},
\end{equation*}
by induction on the dimension of $U$; we do the same for spherical diagrams in $\nervealg{Y}$. This will have the property that if $x$ and $y$ are parallel, then $\shell{x} = \shell{y}[x/y]$. 

If $n = 0,1$, then $\shell{x} := x$. Otherwise, we define $\shell{x}$ to be
\begin{itemize}
	\item equal to $c(\shell{(\restr{x}{\bord{}{-}U})}): \compos{\shell{(\restr{x}{\bord{}{-}U}})} \celto \shell{(\restr{x}{\bord{}{-}U})}$ on $O^{n-1} \celto \shell{(\bord{}{-}U)}$,
	\item equal to $x$ on $U$, and
	\item equal to $c'(\shell{(\restr{x}{\bord{}{+}U})}): \shell{(\restr{x}{\bord{}{+}U})} \celto \compos{\shell{(\restr{x}{\bord{}{+}U})}}$ on $\shell{(\bord{}{+}U)} \celto O^{n-1}$.
\end{itemize}
This is well-defined by the inductive hypothesis, and clearly has the stated property.

Because $\shell{x}$ is equal to a compositor on every atom of $\shell{U}$ which is not in $U$, and by Proposition \ref{prop:asso_preserve} $f$ preserves compositors, it follows that for all $x$,
\begin{equation} \label{eq:shell_preserve}
	f(\shell{x}) = \shell{f(x)}.
\end{equation}
Now, we will prove by induction on $n$ that $f: \nervealg{X}(O^n) \to \nervealg{Y}(O^n)$ determines uniquely $f: \nervealg{X}(U) \to \nervealg{Y}(U)$ for all $n$\nbd atoms $U$. For $n = 0, 1$, there is nothing to prove. 

For $n > 1$, let $x: U \to \nervealg{X}$ be an $n$\nbd cell with transpose $x: \hatto{U} \to X$. By the inductive hypothesis, $f(\restr{x}{\bord{}{\alpha}U})$ is already determined for $\alpha \in \{+,-\}$, so it suffices to determine $f(x)(U)$ as a cell of $Y$. But by the definition of compositors in $\nervealg{X}$,
\begin{equation*}
	x(U) = \shell{x}(\shell{U}) = \compos{\shell{x}}(\compos{\shell{U}}) = \compos{\shell{x}}(O^n),
\end{equation*}
where the last equality holds because $\compos{\shell{U}} = O^n$. The same equations hold for cells in $\nervealg{Y}$, so
\begin{equation*}
	f(x)(U) = \shell{f(x)}(\shell{U}) = \compos{\shell{f(x)}}(O^n) = f(\compos{\shell{x}})(O^n),
\end{equation*}
by (\ref{eq:shell_preserve}) and the fact that $f$ preserves structural weak composites. Thus $f(x)$ is determined by $f(\compos{\shell{x}})$, which completes the proof.
\end{proof}

For the main result, we need a slight extension of the shell construction.

\begin{cons} \label{cons:shell_kcomp}
For $n > 0$, $k < n$, consider the $n$\nbd molecules $O^n \cp{k} O^n$. These have spherical boundary if and only if $k = n-1$, and have the property that 
\begin{equation*}
	\bord{}{\alpha}(O^n \cp{k} O^n) \incliso O^{n-1} \cp{k} O^{n-1}
\end{equation*}
for $k < n-1$. 

Even though the $O^n \cp{k} O^n$ do not have spherical boundary, we will show that we can define a shell $\shell{(O^n \cp{k} O^n)}$ as in Construction \ref{cons:shell} and that it has spherical boundary. We fix $k$ and proceed by induction on $n$. If $n = k+1$, this is the shell as previously defined. 

Let $n > k+1$. We let $\shell{(O^n \cp{k} O^n)}$ be the colimit 
\begin{equation*} 
\begin{tikzpicture}[baseline={([yshift=-.5ex]current bounding box.center)}]
	\node[scale=1.25] (0) at (-2,-1.5) {$O^{n-1} \celto \shell{(O^{n-1} \cp{k} O^{n-1})}$};
	\node[scale=1.25] (1) at (0,0) {$O^{n-1} \cp{k} O^{n-1}$};
	\node[scale=1.25] (2) at (2,-1.5) {$O^n \cp{k} O^n$};
	\node[scale=1.25] (3) at (4,0) {$O^{n-1} \cp{k} O^{n-1}$};
	\node[scale=1.25] (4) at (6,-1.5) {$\shell{(O^{n-1} \cp{k} O^{n-1})} \celto O^{n-1}$};
	\draw[1cincl] (1) to node[auto,swap] {$j_1$} (0);
	\draw[1cinc] (1) to node[auto] {$\imath^-$} (2);
	\draw[1cincl] (3) to node[auto,swap] {$\imath^+$} (2);
	\draw[1cinc] (3) to node[auto] {$j_2$} (4);
\end{tikzpicture}
\end{equation*}
which is well-defined. If we label as $O^n_1, O^n_2$ the leftmost and rightmost copy of $O^n$ in $O^n \cp{k} O^n$, since
\begin{align*}
	\bord{}{-}O^n_1 & \subsph \shell{(O^{n-1} \cp{k} O^{n-1})} \incliso \bord{}{+}(O^{n-1} \celto \shell{(O^{n-1} \cp{k} O^{n-1})}),  \\
	\bord{}{+}O^n_2 & \subsph \shell{(O^{n-1} \cp{k} O^{n-1})} \incliso \bord{}{-}(\shell{(O^{n-1} \cp{k} O^{n-1})} \celto O^{n-1}),
\end{align*}
we have that
\begin{align*}
	U_1 & := (O^{n-1} \celto \shell{(O^{n-1} \cp{k} O^{n-1})}) \cup O^n_1, \\
	U_2 & := O^n_2 \cup (\shell{(O^{n-1} \cp{k} O^{n-1})} \celto O^{n-1}),
\end{align*}
are both  $n$\nbd molecules with spherical boundary, and $O^n_1 \subsph U_1, O^n_2 \subsph U_2$. Therefore $\shell{(O^n \cp{k} O^n)} = U_1 \cp{n-1} U_2$ also has spherical boundary, and $\bord{}{\alpha}\shell{(O^n \cp{k} O^n)} \incliso O^{n-1}$ for both $\alpha \in \{+,-\}$.
\end{cons}

\begin{thm} \label{thm:nerve_ff}
Assume Conjecture \ref{conj:freegen}. The functor $\nervealg{}: \omegacat \to \rdgmsetalg$ is full and faithful.
\end{thm}
\begin{proof}
We have already proved faithfulness. Let $X, Y$ be $\omega$\nbd categories, and consider a strict morphism $f: \nervealg{X} \to \nervealg{Y}$. 

Applying the restriction functor $\rdgmset \to \globsetrefl$, we see that $f$ determines a morphism $\tilde{f}$ of the underlying reflexive $\omega$\nbd graphs of $X$ and $Y$. Moreover, by Lemma \ref{lem:minimal_suffice}, $f$ is uniquely determined by $\tilde{f}$. Thus it suffices to prove that $\tilde{f}$ is compatible with the composition operations of $X$ and $Y$.

Suppose $x_1, x_2$ are $n$\nbd cells of $X$ such that $x_1 \cp{k} x_2$ is defined. This composition is classified by a pasting diagram
\begin{equation*}
	x: \hatto{(O^n \cp{k} O^n)} \to X
\end{equation*}
such that $x(O^n \cp{k} O^n) = x_1 \cp{k} x_2$, whose transpose is a diagram $x: O^n \cp{k} O^n \to \nervealg{X}$. 

We can proceed exactly as in the proof of Lemma \ref{lem:minimal_suffice}, and define a spherical $n$\nbd diagram $\shell{x}: \shell{(O^n \cp{k} O^n)} \to \nervealg{X}$ with $x_1, x_2 \subsph \shell{x}$, and such that $\shell{f(x)} = f(\shell{x})$. We then have
\begin{equation*}
	x_1 \cp{k} x_2 = x(O^n \cp{k} O^n) = \shell{x}(\shell{(O^n\cp{k} O^n)}) = \compos{\shell{x}}(\compos{\shell{(O^n\cp{k} O^n)}}) = \compos{\shell{x}}(O^n),
\end{equation*}
so the composite $x_1 \cp{k} x_2$ is classified by the transpose of $\compos{\shell{x}}: O^n \to \nervealg{X}$, and $\tilde{f}$ maps it to the transpose of $f(\compos{\shell{x}}): O^n \to \nervealg{Y}$. Then
\begin{align*}
	\tilde{f}(x_1 \cp{k} x_2) & = f(\compos{\shell{x}})(\compos{\shell{(O^n\cp{k} O^n)}}) = \compos{\shell{f(x)}}(\compos{\shell{(O^n\cp{k} O^n)}}) = \\
		& = f(x)(O^n \cp{k} O^n) = \tilde{f}(x_1) \cp{k} \tilde{f}(x_2).
\end{align*}
So $\tilde{f}: X \to Y$ is a functor of $\omega$\nbd categories, and $f = \nervealg{\tilde{f}}$. 
\end{proof}

Thus, conditional to Conjecture \ref{conj:freegen}, we can see $\omegacat$ as a full subcategory of $\rdgmsetalg$. We do not have, at the moment, an ``internal'' characterisation of strict $\omega$\nbd categories inside $\rdgmsetalg$: being associative is certainly a necessary condition, but not having required any compatibility between compositors and degeneracies, we cannot expect it to be sufficient. 

In fact, let us leave the following open question, a kind of ``C.\ Simpson's conjecture'' for algebraic RDSs.

\begin{conj}
Every RDS is weakly equivalent to an associative algebraic RDS.
\end{conj}

\subsection{Low-dimensional examples} \label{sec:lowdim}

The previous section proved that, conditional to a conjecture, the theory of strict $\omega$\nbd categories embeds into the theory of representable diagrammatic sets, and the last part of the article will prove that, in a specific sense, so does the classical theory of homotopy types. 

The goal of this section is to give a concrete idea of the theory of \emph{weak $n$\nbd categories} exhibited by the theory of RDSs, especially for low dimensions. This is, in fact, where we believe that RDSs have an advantage over other formalisms, because of their flexible combinatorics of pasting diagrams. 

First, we need a definition of what a weak $n$\nbd category is in the form of a RDS.

\begin{dfn}
A representable diagrammatic set $X$ is \emph{$n$\nbd truncated} if
\begin{enumerate}
	\item $X$ is $(n+1)$\nbd coskeletal,
	\item any two parallel $(n+1)$\nbd cells of $X$ are equal, and
	\item if $x, y$ are parallel $n$\nbd cells of $X$ and $e: x \celto y$ is an $(n+1)$\nbd cell, then $x = y$.
\end{enumerate}
We write $n \rdgmset$ for the full subcategory of $\rdgmset$ on the $n$\nbd truncated RDSs, and $n \rdgmsetalg$ for the full subcategory of $\rdgmsetalg$ on the $n$\nbd truncated algebraic RDSs.
\end{dfn}

\begin{remark}
An $n$\nbd truncated RDS, being $(n+1)$\nbd coskeletal, is fully determined by its $(n+1)$\nbd skeleton. Thus we may see $n \rdgmset$ as a subcategory of $(n+1)\dgmset$.
\end{remark}

Let us prove a couple of properties of coskeletal diagrammatic sets, before focussing on $n$\nbd truncated RDSs. 

\begin{lem} \label{lem:coskel_fillers}
Let $X$ be an $n$\nbd coskeletal diagrammatic set. Then for all $k > n$
\begin{enumerate}[label=(\alph*)]
	\item all horns of $(k+1)$\nbd atoms have fillers in $X$, and
	\item all $k$\nbd cells of $X$ are equivalences.
\end{enumerate}
\end{lem}
\begin{proof}
Let $W$ be a $(k+1)$\nbd atom, and $\Lambda \incl W$ a horn. If $k > n$, then $\skel{n}{\Lambda} = \skel{n}{W}$: it follows that horns of $W$ in $X$ are in bijection with their fillers, so they all have fillers. The fact that all $k$\nbd cells are equivalences then follows by a simple coinduction.
\end{proof}

\begin{prop} \label{prop:unique_composite}
Let $X$ be an $n$\nbd truncated RDS. Then for all $k > n$
\begin{enumerate}[label=(\alph*)]
	\item all spherical $(k-1)$\nbd diagrams have a unique weak composite, and
	\item all $k$\nbd cells of $X$ are equivalences.
\end{enumerate}
\end{prop}
\begin{proof}
Let $x$ be a spherical $(k-1)$\nbd diagram, and let $\compos{x}$ and $\compos{x}'$ be two weak composites of $x$. Then $\compos{x}$ and $\compos{x}'$ are parallel cells, which by definition implies they are equal if $k > n+1$, and $\compos{x} \simeq \compos{x}'$, which implies they are equal if $k = n+1$. 

Since $X$ is $(n+1)$\nbd coskeletal, all $k$\nbd cells are equivalences for $k > n+1$. Let $e: x \celto y$ be an $(n+1)$\nbd cell, and pick compositors $c: x \celto \compos{x}$ and $c': \compos{y} \celto y$. Then 
\begin{equation*}
	e \simeq c \cp{n} e' \cp{n} c'
\end{equation*}
for some $(n+1)$\nbd cell $e': \compos{x} \celto \compos{y}$. Since $X$ is $n$\nbd truncated, $\compos{x} = \compos{y} =: z$, and $e' = p_U^*z$, where $U$ is the shape of $z$. It follows that $e$ is a weak composite of two equivalences and a degenerate cell, hence an equivalence.
\end{proof}

\begin{remark}
It follows that, in an $n$\nbd truncated RDS, two parallel spherical $n$\nbd diagrams $x, y$ have a cell between them if and only if $x \simeq y$, if and only if the unique weak composites $\compos{x}, \compos{y}$ are equal.
\end{remark}

\begin{cor} \label{cor:1alg_iso_1}
The categories $1 \rdgmset$ and $1 \rdgmsetalg$ are isomorphic.
\end{cor}
\begin{proof}
In a 1-truncated RDS, every spherical 1-diagram has a unique weak composite. Two compositors for the same diagram are parallel 2-cells, hence they are equal. It follows that any morphism between 1-truncated RDSs is strict.
\end{proof}

\begin{prop} \label{prop:unique_inverse}
Let $X$ be an $n$\nbd truncated RDS, and let $e: x \celto y$ be an $n$\nbd equivalence. Then $e$ has a unique weak inverse.
\end{prop}
\begin{proof}
Let $e', e''$ be two weak inverses. The usual argument for uniqueness of inverses up to equivalence gives $e' \simeq e''$, and since $X$ is $n$\nbd truncated, $e' = e''$. 
\end{proof}

An \emph{$n$\nbd category} is an $\omega$\nbd category whose $k$\nbd cells are all degenerate for $k > n$. The following is evident.

\begin{prop} \label{prop:nerve_truncated}
Let $X$ be an $n$\nbd category, and assume Conjecture \ref{conj:freegen}. Then $\nerve{X}$ is an $n$\nbd truncated RDS.
\end{prop}

Thus if $n\cat{Cat}$ is the full subcategory of $\omegacat$ on the $n$\nbd categories, we have for all $n \geq 0$ a faithful functor $\nerve{}: n\cat{Cat} \to n \rdgmset$ and a full and faithful functor $\nervealg{}: n\cat{Cat} \to n \rdgmset$.

\begin{remark}
For $n \leq 2$, we know for certain that regular directed complexes $P$ of dimension $n$ are freely generating. Any non-free identifications which may occur in $k$\nbd skeleta for $k > n$ in $(\molec{}{}P)^*$ are inconsequential if we are looking at functors into an $n$\nbd category $X$, since all parallel $k$\nbd cells in $X$ are equal. This is sufficient to remove the conditional on Proposition \ref{prop:nerve_truncated} for $n \leq 2$.   
\end{remark}

Now, $1\cat{Cat}$ is isomorphic to the usual category $\cat{Cat}$ of small categories.

\begin{prop}
$\nerve{}: \cat{Cat} \to 1\rdgmset$ is an equivalence of categories.
\end{prop}
\begin{proof}
By Corollary \ref{cor:1alg_iso_1}, we know that $\nerve{}$ is full and faithful. Let $X$ be a 1-truncated RDS. We define a category $\tilde{X}$ with $\tilde{X}_0 := X(1)$, $\tilde{X}_1 := X(O^1)$, with the obvious boundaries; $\idd{}v :=\ !^*v$ for all $v \in X(1)$, where $!: O^1 \to 1$ is the unique map to the terminal object; and $x\cp{0}y := \compos{x \cp{0} y}$ for all 0-composable 1-cells, well defined because $x \cp{0} y: O^1 \cp{0} O^1 \to X$ is a spherical 1-diagram. 

It is easy to see that $\tilde{X}$ is well-defined, and that $X$ is isomorphic to $\nerve{\tilde{X}}$.
\end{proof}

Thus 1-truncated RDSs and their morphisms are equivalent to small categories and functors, as expected. The first non-trivial case is that of 2-truncated RDSs and bicategories. We recall the definition of bicategory, in order to fix the notation.

\begin{dfn}
A \emph{bicategory} $B$ is a reflexive 2-graph together with the following data:
\begin{enumerate}
	\item a family of 2-cells, the \emph{1-composites} $\{p\cpm{1}q: a \celto c\}$, indexed by composable pairs of 2-cells $p: a \celto b$ and $q: b \celto c$;
	\item a family of 1-cells $\{a \cpm{0} b: x \celto z\}$, indexed by 0-composable pairs of 1-cells $a: x \celto y$, $b: y \celto z$, and a family of 2-cells $\{p \cpm{0} q: a \cpm{0} c \celto b \cpm{0} d\}$, indexed by 0-composable pairs of 2-cells $p: a \celto b$, $q: c \celto d$, both called \emph{0-composites};
	\item a family of 2-cells, the \emph{associators} $\{\alpha_{a,b,c}: (a \cpm{0} b) \cpm{0} c \celto a \cpm{0} (b \cpm{0} c)\}$, indexed by 0-composable triples of 1-cells $a: x \celto y$, $b: y \celto z$, $c: z \celto w$;
	\item two families of 2-cells, the \emph{left unitors} $\{\lambda_a: \idd{}{x} \cpm{0} a \celto a\}$ and the \emph{right unitors} $\{\rho_a: a \cpm{0} \idd{}{y} \celto a\}$, indexed by 1-cells $a: x \celto y$.
\end{enumerate}
These are subject to the following conditions:
\begin{enumerate}
	\item 1-composition is associative and unital with 2-units, that is, $(p\cpm{1} q)\cpm{1} r = p\cpm{1}(q\cpm{1}r)$, and $p\cpm{1}\idd{}{b} = p = \idd{}{a}\cpm{1}q$, whenever both sides make sense;
	\item 0-composition is natural with respect to 1-composition and 2-units, that is, 
	\begin{align*} 
		(p_1\cpm{1}p_2) \cpm{0} (q_1\cpm{1}q_2) & = (p_1 \cpm{0} q_1)\cpm{1}(p_2 \cpm{0} q_2), \\
		\idd{}{a} \cpm{0} \idd{}{b} & = \idd{}(a \cpm{0} b), 
	\end{align*}
	whenever the left-hand side makes sense;
	\item the associators and the unitors are natural in their parameters, that is, for all $p: a \celto a'$, $q: b \celto b'$, and $r: c \celto c'$, the following diagrams commute:
	\begin{equation*}
\begin{tikzpicture}[baseline={([yshift=-.5ex]current bounding box.center)}]
\begin{scope}
	\node[scale=1.25] (0) at (-1.5,.75) {$(a \cpm{0} b) \cpm{0} c$};
	\node[scale=1.25] (1) at (2,.75) {$a \cpm{0} (b \cpm{0} c)$};
	\node[scale=1.25] (2) at (-1.5,-.75) {$(a' \cpm{0} b') \cpm{0} c'$};
	\node[scale=1.25] (3) at (2,-.75) {$a' \cpm{0} (b' \cpm{0} c')$};
	\draw[1c] (0.east) to node[auto] {$\alpha_{a,b,c}$} (1.west);
	\draw[1c] (2.east) to node[auto,swap] {$\alpha_{a',b',c'}$} (3.west);
	\draw[1c] (0.south) to node[auto,swap] {$(p \cpm{0} q) \cpm{0} r$} (2.north);
	\draw[1c] (1.south) to node[auto] {$p \cpm{0} (q \cpm{0} r)$} (3.north);
	\node[scale=1.25] at (3.5,-.9) {,};
\end{scope}
\begin{scope}[shift={(6,0)}]
	\node[scale=1.25] (0) at (-1,.75) {$\idd{}{x} \cpm{0} a$};
	\node[scale=1.25] (1) at (1,.75) {$a$};
	\node[scale=1.25] (0') at (3,.75) {$a \cpm{0} \idd{}{y}$};
	\node[scale=1.25] (2) at (-1,-.75) {$\idd{}{x} \cpm{0} a'$};
	\node[scale=1.25] (3) at (1,-.75) {$a'$};
	\node[scale=1.25] (2') at (3,-.75) {$a' \cpm{0} \idd{}{y}$};
	\draw[1c] (0.east) to node[auto] {$\lambda_a$} (1.west);
	\draw[1c] (2.east) to node[auto,swap] {$\lambda_{a'}$} (3.west);
	\draw[1c] (0'.west) to node[auto, swap] {$\rho_a$} (1.east);
	\draw[1c] (2'.west) to node[auto] {$\rho_{a'}$} (3.east);
	\draw[1c] (0.south) to node[auto,swap] {$\idd{}{x} \cpm{0} p$} (2.north);
	\draw[1c] (0'.south) to node[auto] {$p \cpm{0} \idd{}{y}$} (2'.north);
	\draw[1c] (1.south) to node[auto] {$p$} (3.north);
\end{scope}
\end{tikzpicture}
\end{equation*}
	commute;
	\item the associators and unitors satisfy the pentagon and triangle equations, that is, for all $a: x \celto y$, $b: y \celto z$, $c: z \celto w$, $d: w \celto v$, the following diagrams commute:
	\begin{equation*}
\begin{tikzpicture}[baseline={([yshift=-.5ex]current bounding box.center)}]
	\node[scale=1.25] (0) at (-5,.75) {$((a \cpm{0} b) \cpm{0} c) \cpm{0} d$};
	\node[scale=1.25] (1) at (5,.75) {$(a \cpm{0} b) \cpm{0} (c \cpm{0} d)$};
	\node[scale=1.25] (2) at (-5,-.75) {$(a \cpm{0} (b \cpm{0} c)) \cpm{0} d$};
	\node[scale=1.25] (2b) at (0,-.75) {$a \cpm{0} ((b \cpm{0} c) \cpm{0} d)$};
	\node[scale=1.25] (3) at (5,-.75) {$a \cpm{0} (b \cpm{0} (c \cpm{0} d))$};
	\draw[1c] (0.east) to node[auto] {$\alpha_{a\cpm{0} b,c,d}$} (1.west);
	\draw[1c] (2.east) to node[auto,swap] {$\alpha_{a,b\cpm{0} c, d}$} (2b.west);
	\draw[1c] (2b.east) to node[auto,swap] {$\idd{}{a} \cpm{0} \alpha_{b,c,d}$} (3.west);
	\draw[1c] (0.south) to node[auto,swap] {$\alpha_{a,b,c} \cpm{0} \idd{}{d}$} (2.north);
	\draw[1c] (1.south) to node[auto] {$\alpha_{a,b,c \cpm{0} d}$} (3.north);
	\node[scale=1.25] at (6.5,-.9) {,};
\end{tikzpicture}
\end{equation*}
	\begin{equation*}
\begin{tikzpicture}[baseline={([yshift=-.5ex]current bounding box.center)}]
	\node[scale=1.25] (0) at (-2.5,.75) {$(a \cpm{0} \idd{}{y}) \cpm{0} b$};
	\node[scale=1.25] (1) at (2.5,.75) {$a \cpm{0} (\idd{}{y} \cpm{0} b)$};
	\node[scale=1.25] (3) at (2.5,-.75) {$a \cpm{0} b$};
	\draw[1c] (0.east) to node[auto] {$\alpha_{a,\idd{}{y},b}$} (1.west);
	\draw[1c] (0) to node[auto,swap] {$\rho_a \cpm{0} \idd{}{b}$} (3);
	\draw[1c] (1.south) to node[auto] {$\idd{}{a} \cpm{0} \lambda_b$} (3.north);
	\node[scale=1.25] at (3.5,-.8) {;};
\end{tikzpicture}
\end{equation*}
	\item the associators and unitors are invertible with respect to 1-composition.
\end{enumerate}
Given two bicategories $B$, $C$, a \emph{functor} $f: B \to C$ is a morphism of the underlying 2-graphs that respects 1-composition and 2-units, together with
\begin{enumerate}
	\item a family of invertible 2-cells in $C$ $\{f_{a,b}: f(a) \cpm{0} f(b) \celto f(a\cpm{0} b)\}$, indexed by 0-composable pairs of 1-cells $a: x \celto y$, $b: y \celto z$ of $B$, and
	\item a family of invertible 2-cells in $C$ $\{f_x: \idd{}{f(x)} \celto f(\idd{}{x})\}$, indexed by 0-cells $x$ of $B$,
\end{enumerate}
where the first one is natural in its parameters, that is, for all $p: a \celto a'$, $q: b \celto b'$, the diagram
\begin{equation*}
\begin{tikzpicture}[baseline={([yshift=-.5ex]current bounding box.center)}]
	\node[scale=1.25] (0) at (-1.5,.75) {$f(a) \cpm{0} f(b)$};
	\node[scale=1.25] (1) at (2,.75) {$f(a \cpm{0} b)$};
	\node[scale=1.25] (2) at (-1.5,-.75) {$f(a') \cpm{0} f(b')$};
	\node[scale=1.25] (3) at (2,-.75) {$f(a' \cpm{0} b')$};
	\draw[1c] (0.east) to node[auto] {$f_{a,b}$} (1.west);
	\draw[1c] (2.east) to node[auto,swap] {$f_{a',b'}$} (3.west);
	\draw[1c] (0.south) to node[auto,swap] {$f(p) \cpm{0} f(q)$} (2.north);
	\draw[1c] (1.south) to node[auto] {$f(p \cpm{0} q)$} (3.north);
\end{tikzpicture}
\end{equation*}
commutes, and both families are compatible with the associators and unitors, in the sense that, for all $a: x \celto y$, $b: y \celto z$, $c: z \celto w$, the following diagrams commute:
		\begin{equation*}
\begin{tikzpicture}[baseline={([yshift=-.5ex]current bounding box.center)}]
	\node[scale=1.25] (0) at (-5,.75) {$(f(a) \cpm{0} f(b)) \cpm{0} f(c)$};
	\node[scale=1.25] (0b) at (0,.75) {$f(a) \cpm{0} (f(b) \cpm{0} f(c))$};
	\node[scale=1.25] (1) at (5,.75) {$f(a) \cpm{0} f(b \cpm{0} c)$};
	\node[scale=1.25] (2) at (-5,-.75) {$f(a \cpm{0} b) \cpm{0} f(c)$};
	\node[scale=1.25] (2b) at (0,-.75) {$f((a \cpm{0} b) \cpm{0} c)$};
	\node[scale=1.25] (3) at (5,-.75) {$f(a \cpm{0} (b \cpm{0} c))$};
	\draw[1c] (0.east) to node[auto] {$\alpha_{f(a),f(b),f(c)}$} (0b.west);
	\draw[1c] (0b.east) to node[auto] {$\idd{}{f(a)} \cpm{0} f_{b,c}$} (1.west);
	\draw[1c] (2.east) to node[auto,swap] {$f_{a \cpm{0} b,c}$} (2b.west);
	\draw[1c] (2b.east) to node[auto,swap] {$f(\alpha_{a,b,c})$} (3.west);
	\draw[1c] (0.south) to node[auto,swap] {$f_{a,b} \cpm{0} \idd{}{f(c)}$} (2.north);
	\draw[1c] (1.south) to node[auto] {$f_{a,b\cpm{0} c}$} (3.north);
	\node[scale=1.25] at (6.5,-.9) {,};
\end{tikzpicture}
\end{equation*}
\begin{equation*}
\begin{tikzpicture}[baseline={([yshift=-.5ex]current bounding box.center)}]
\begin{scope}
	\node[scale=1.25] (0) at (-4,.75) {$\idd{}{f(x)} \cpm{0} f(a)$};
	\node[scale=1.25] (1b) at (0,.75) {$f(\idd{}{x}) \cpm{0} f(a)$};
	\node[scale=1.25] (1) at (4,.75) {$f(\idd{}{x} \cpm{0} a)$};
	\node[scale=1.25] (3) at (4,-.75) {$f(a)$};
	\draw[1c] (0.east) to node[auto] {$f_x \cpm{0} \idd{}{f(a)}$} (1b.west);
	\draw[1c] (1b.east) to node[auto] {$f_{\idd{}{x},a}$} (1.west);
	\draw[1c] (0) to node[auto,swap] {$\lambda_{f(a)}$} (3);
	\draw[1c] (1.south) to node[auto] {$f(\lambda_a)$} (3.north);
	\node[scale=1.25] at (5,-.9) {,};
\end{scope}
\end{tikzpicture}
\end{equation*}
\begin{equation*}
\begin{tikzpicture}[baseline={([yshift=-.5ex]current bounding box.center)}]
\begin{scope}
	\node[scale=1.25] (0) at (-4,.75) {$f(a) \cpm{0} \idd{}{f(y)}$};
	\node[scale=1.25] (1b) at (0,.75) {$f(a) \cpm{0} f(\idd{}{y})$};
	\node[scale=1.25] (1) at (4,.75) {$f(a \cpm{0} \idd{}{y})$};
	\node[scale=1.25] (3) at (4,-.75) {$f(a)$};
	\draw[1c] (0.east) to node[auto] {$\idd{}{f(a)} \cpm{0} f_y$} (1b.west);
	\draw[1c] (1b.east) to node[auto] {$f_{a,\idd{}{y}}$} (1.west);
	\draw[1c] (0) to node[auto,swap] {$\rho_{f(a)}$} (3);
	\draw[1c] (1.south) to node[auto] {$f(\rho_a)$} (3.north);
	\node[scale=1.25] at (5,-.9) {.};
\end{scope}
\end{tikzpicture}
\end{equation*}
A functor is \emph{strictly unital} if the $f_x$ are all 2-units, and \emph{strict} if it is strictly unital and the $f_{a,b}$ are all 2-units. 
\end{dfn}

Let $\bicatun$ be the category of bicategories and strictly unital functors, and $\bicatst$ the category of bicategories and strict functors. We claim that $\bicatun$ is equivalent to $2\rdgmset$, and $\bicatst$ to $2\rdgmsetalg$. This follows from our results in \cite[Section 5]{hadzihasanovic2018weak} after observing that a 2\nbd truncated RDS is equivalent to a \emph{representable merge-bicategory} which has, in addition, a choice of coherent unitors, preserved by all morphisms. 

Since it is straightforward to connect the two formalisms, we will be brief in our treatment, and only define the two sides of the equivalence; details can be filled in using \cite[Section 5]{hadzihasanovic2018weak}.

\begin{cons} \label{cons:rds_to_bicat}
Let $X$ be a 2\nbd truncated RDS with a choice of compositors. We can always assume that dual pairs of compositors are each other's weak inverse, and by Proposition \ref{prop:unique_inverse} a ``one-sided'' choice of 2-dimensional compositors uniquely determines their duals.

The 2-skeleton of the reflexive $\omega$\nbd graph $X_{\tilde{\cat{O}}}$ is a reflexive 2-graph, which we denote by $\mathcal{B}X$. We endow $\mathcal{B}X$ with the structure of a bicategory, as follows.
\begin{enumerate}
	\item For each 1-composable pair of 2-cells $p: a \celto b$ and $q: b \celto c$, by Proposition \ref{prop:unique_composite} the spherical 2-diagram $p\cp{1}q$ has a unique weak composite $\compos{p\cp{1}q}$, and we define $p \cpm{1} q := \compos{p\cp{1}q}$. 
	
	\item For each 0-composable pair of 1-cells $a: x \celto y$, $b: y \to z$, let $a \cpm{0} b := \compos{a \cp{0} b}$ be the chosen weak composite of the spherical 1-diagram $a \cp{0} b$. For each 0-composable pair of 2-cells $p: a \celto c$, $q: b \celto d$, let $h: a \cpm{0} b \celto a \cp{0} b$ and $h': c \cp{0} d \celto c \cpm{0} d$ be the chosen compositors. Then we define $p \cpm{0} q$ to be the unique weak composite of the spherical 2-diagram
	\begin{equation*}
		h\cp{1}(p \cp{0} q)\cp{1}h': a \cpm{0} b \celto c \cpm{0} d
	\end{equation*}
	of shape $\shell{(O^2 \cp{0} O^2)}$, as in Construction \ref{cons:shell_kcomp}.
	\item For each composable triple of 1-cells $a, b, c$, we have binary compositors 
	\begin{equation*}\begin{gathered}
		h: a\cpm{0}b \celto a\cp{0}b, \quad \quad k': b\cp{0}c \celto b \cpm{0} c, \\
		\ell: (a\cpm{0}b)\cpm{0}c \celto (a\cpm{0}b)\cp{0}c, \quad \quad m': a\cp{0}(b\cpm{0}c) \celto a\cpm{0}(b\cpm{0}c).
	\end{gathered} \end{equation*}
	We define $\alpha_{a,b,c}$ to be the unique weak composite of the spherical 2-diagram
	\begin{equation*}
		\ell \cp{1} (h \cp{0} c) \cp{1} (a \cp{0} k') \cp{1} m': (a\cpm{0}b)\cpm{0}c \celto a\cpm{0}(b\cpm{0}c).
	\end{equation*}
	All the cells in the diagram are (weakly) invertible, so $\alpha_{a,b,c}$ is invertible.
	\item From Construction \ref{cons:unitor_molec}, we have 2-atoms $\tilde{L}^1_{O^1} = \tilde{R}^1_{O^1} := (O^1 \cp{0} O^1) \celto O^1$ and retractions $\tilde{\ell}^1_{O^1}: \tilde{L}^1_{O^1} \surj O^1$ and $\tilde{r}^{1}_{O^1}: \tilde{R}^{1}_{O^1} \surj O^1$. For each 1-cell $a: x \celto y$, let $h: \idd{}x \cpm{0} a \celto \idd{}x \cp{0} a$ and $k: a \cpm{0} \idd{}y \celto a \cp{0} \idd{}x$ be compositors. Then we define 
	\begin{equation*}
		\lambda_a := h \cpm{1} (\tilde{\ell}^1_{O^1})^*a, \quad \quad \rho_a := k \cpm{1} (\tilde{r}^1_{O^1})^*a.
	\end{equation*}
	These are weak composites of an equivalence and a degenerate cell, hence equivalences, hence (weakly) invertible.
\end{enumerate}
Now, let $Y$ be another 2\nbd truncated RDS, and $f: X \to Y$ a morphism. We endow the restriction $\mathcal{B}f: \mathcal{B}X \to \mathcal{B}Y$ of $f$ with the structure of a strictly unital functor of bicategories, as follows.
\begin{enumerate}
	\item For each 0-composable pair of 1-cells $a: x \celto y$, $b: y \celto z$ of $X$, a compositor $h': a \cp{0} b \celto a \cpm{0} b$ is mapped to $f(h'): f(a) \cp{0} f(b) \celto f(a \cpm{0} b)$, and in $Y$ there is a compositor $k: f(a) \cpm{0} f(b) \celto f(a) \cp{0} f(b)$. We define $f_{a,b}$ to be the unique weak composite of the spherical 2-diagram 
	\begin{equation*}
		k \cp{1} f(h'): f(a) \cpm{0} f(b) \celto f(a \cpm{0} b).
	\end{equation*}
	\item For each 0-cell $x$ of $X$, we define $f_x$ to be the 2-unit on $\idd{}f(x) = f(\idd{}x)$.
\end{enumerate}
If $f: X \to Y$ is a strict morphism, then $f(h')$ is the unique weak inverse of $k$ in $Y$, so the $f_{a,b}$ are units, and $\mathcal{B}f$ is a strict functor.

This defines functors $\mathcal{B}: 2\rdgmset \to \bicatun$ and $\mathcal{B}: 2\rdgmsetalg \to \bicatst$.
\end{cons}

\begin{cons}
Let $B$ be a bicategory; we define a diagrammatic set $\nerve{B}$ as follows. The 0-cells and 1-cells of $\nerve{B}$ are the same as those of $B$. The 2-cells 
\begin{equation*}
	p: a_1 \cp{0} \ldots \cp{0} a_n \celto b_1 \cp{0} \ldots \cp{0} b_m
\end{equation*}
of $\nerve{B}$ correspond to 2-cells
\begin{equation} \label{eq:bracketing}
	p: (\ldots(a_1\cpm{0} a_2) \ldots \cpm{0} a_{n-1}) \cpm{0} a_n \celto (\ldots(b_1\cpm{0} b_2) \ldots \cpm{0} b_{m-1}) \cpm{0} b_m
\end{equation}
of $B$. Given any spherical 2-diagram in $\nerve{B}$, we can always insert 2-units, associators and their inverses to ``rebracket'' so to obtain a composable diagram in $B$, whose composite corresponds to a 2-cell in $\nerve{B}$; the coherence theorem for bicategories \cite[Theorem 3.1]{maclane1963natural} implies that the result is independent of any choices made. For any two parallel spherical 2-diagrams $x, y$ in $\nerve{B}$, there is a unique 3-cell $e: x \celto y$ if and only if the composite of $x$ is equal to the composite of $y$ in $B$. Finally, we impose that $\nerve{B}$ be 3-coskeletal. 

The face maps of $\nerve{B}$ are the obvious ones. The degenerate 1-cells of $\nerve{B}$ are determined by the underlying reflexive 2-graph of $B$, and similarly for the action of $p_{O^1}: O^2 \surj O^1$. Any other surjective map $U \surj O^1$ factors through $p_{O^1}$, so it suffices to describe the action of surjective maps $f: U \surj V$ between different 2-atoms. If $p$ is a 2-cell of shape $V$, then the input and output boundary $f^*p$ are equal to the input and output boundary of $p$ with some 1-units inserted. We can always precompose and postcompose $p$ with unitors and their inverses to obtain a 2-cell of $B$ of the type of $f^*p$, and by coherence the result is independent of the particular way that this is done. 

Finally, coherence also implies that, given a surjective map $f$ from a 3-atom, the action of its restriction to the 2-boundaries takes diagrams with equal composites to diagrams with equal composites, so there is a unique way of defining the action of $f$. 

This suffices to define $\nerve{B}$ as a 3-coskeletal diagrammatic set. Moreover, $\nerve{B}$ is representable, and has a canonical choice of compositors for diagrams $a \cp{0} b$ of two 1-cells: we let $\compos{a \cp{0} b} := a \cpm{0} b$, and the compositor $c: a \cpm{0} b \celto a \cp{0} b$ corresponds to the 2-unit $\idd{}(a \cpm{0} b)$ in $B$. By construction, $\nerve{B}$ is 2-truncated.

Given a strictly unital functor $f: B \to C$ of bicategories, we define a morphism $\nerve{f}: \nerve{B} \to \nerve{C}$ as follows: $\nerve{f}$ is equal to $f$ on 0-cells and 1-cells. Given a 2-cell $p$ of $\nerve{B}$, corresponding to a 2-cell of $B$ of type (\ref{eq:bracketing}), by \cite[Corollary 1.8]{joyal1993braided} there are unique 2-cells of $C$ of type
\begin{align*}
	f((\ldots(a_1\cpm{0} a_2) \ldots )\cpm{0} a_n) & \celto (\ldots(f(a_1)\cpm{0} f(a_2))\ldots )\cpm{0} f(a_n), \\
	(\ldots(f(b_1)\cpm{0} f(b_2))\ldots )\cpm{0} f(b_m) & \celto f((\ldots(b_1\cpm{0} b_2) \ldots )\cpm{0} b_m),
\end{align*}
built from 2-units, associators, and their inverses in $C$, images through $f$ of associators and ther inverses in $B$, and the structural 2-cells $f_{a,b}$. Then we define $\nerve{f}(p)$ to be the unique 2-cell of the correct type obtained by precomposing and postcomposing $f(p)$ with these 2-cells. There is a unique way of extending $\nerve{f}$ to higher-dimensional cells. 

Notice that, if $f$ if a strict functor, then $f(a \cpm{0} b) = f(a) \cpm{0} f(b)$, and $\nerve{f}$ takes compositors in $\nerve{B}$ to compositors in $\nerve{C}$, that is, $\nerve{f}$ is a strict morphism.

This defines functors $\nerve{}: \bicatun \to 2\rdgmset$ and $\bicatst \to 2\rdgmsetalg$.
\end{cons}

\begin{remark}
If the bicategory $B$ is in fact a strict 2-category, then $\nerve{B}$ is isomorphic to its diagrammatic nerve as previously defined.
\end{remark}

\begin{prop}
The functors $\mathcal{B}$ and $\nerve{}$ are two sides of an equivalence of categories between $2\rdgmset$ and $\bicatun$, restricting to an equivalence between $2\rdgmsetalg$ and $\bicatst$.
\end{prop}
\begin{proof}
Essentially the same proof as the equivalence between bicategories and merge-bicategories in \cite[Section 5]{hadzihasanovic2018weak}.
\end{proof}

\begin{remark}
Unlike in \cite{hadzihasanovic2018weak}, we do not obtain general functors of bicategories; this is due to the presence of structural degeneracies which are strictly preserved by morphisms of diagrammatic sets. We could recover the general case here, and plausibly in higher-dimensional cases, by using morphisms of the underlying regular polygraphs; in \cite[Appendix B]{hadzihasanovic2018combinatorial} we sketched a definition of representability in the absence of structural degeneracies. However, in addition to the handling of equivalences being more complicated, morphisms of regular polygraphs would not automatically preserve equivalences. 
\end{remark}

\begin{remark}
Construction \ref{cons:rds_to_bicat} is a blueprint for transitioning from an RDS $X$ to an algebraic weak higher category structure on the reflexive $\omega$\nbd graph $X_{\tilde{\cat{O}}}$:
\begin{itemize}
	\item given $k$\nbd composable $n$\nbd cells $x, y$, one can extend the diagram 
	\begin{equation*}
		x \cp{k} y: O^n \cp{k} O^n \to X
	\end{equation*}
	to a spherical diagram
	\begin{equation*}
		\shell{(x \cp{k} y)}: \shell{(O^n \cp{k} O^n)} \to X
	\end{equation*}
	adding compositors, similarly to the proof of Theorem \ref{thm:nerve_ff}, and use a chosen weak composite of $\shell{(x \cp{k} y)}$ as the algebraic $k$\nbd composite of $x$ and $y$;
	\item degenerate cells produced by the retractions of Construction \ref{cons:unitor_molec} may be used as higher-dimensional unitors;
	\item combinations of degenerate cells and compositors should provide other structural cells, such as associators and interchangers.
\end{itemize}
Of course, beyond the boundary case of three dimensions, treating the algebraic structure as explicitly as with bicategories is prohibitive. We hope that something more could be said about the relation between RDSs and more abstract approaches to algebraic weak higher categories, such as Leinster's \cite{leinster2004higher}; this should be investigated in future work.
\end{remark}

We conclude this section by briefly considering the case of three dimensions, where strict and weak higher categories diverge. A benchmark for a model of weak 3-categories is the ability to construct non-trivial braidings between pairs of 2-cells with degenerate boundaries, which, as displayed for example in \cite{joyal2007weak}, is necessary for the modelling of homotopy 3-types with non-trivial Whitehead brackets. 

The ability of RDSs to model all homotopy 3-types (in the sense of the homotopy hypothesis) is implied by the results of the following section; however, we take the opportunity here to show that one can construct spherical \emph{braiding diagrams} in an arbitrary diagrammatic set, only using the combinatorics of degeneracies, independently of representability. 

This is an important feature of our model, since it implies that a higher-dimensional theory presented as a diagrammatic set --- where one wants to work explicitly with diagrams, as ``free composites'', rather than take weak composites --- will automatically come with braiding diagrams. For example, a ``2-PRO'' modelled as a diagrammatic set with a single 0-cell and a single irreducible 2-cell $a$ will have braiding diagrams 
\begin{equation*}
	b, b^*: a \cp{1} a \celto a \cp{1} a, 
\end{equation*}
allowing one to encode the data of a PROB, the braided analogue of a PROP \cite{lack2004composing}. By contrast, ``2-PROs'' modelled as strict 3-categories are highly degenerate. Thus diagrammatic sets seem to be a solution to the problem we set at the end of \cite[Section 2.3]{hadzihasanovic2017algebra}. 

\begin{exm} \label{exm:braiding}
Consider a diagrammatic set $X$ with a 0-cell $v$ and two 2-cells 
\begin{equation*}
	a, b:\, !^*v \celto\, !^*v
\end{equation*}
of shape $O^2$, where $!$ is the unique map from $O^1$ to $1$. Let $U_1, U_2$ be the 2-atoms
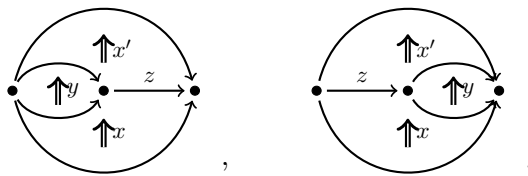
\begin{equation*}
\begin{tikzpicture}[baseline={([yshift=-.5ex]current bounding box.center)},scale=.8]
\begin{scope}
	\node[0c] (0) at (-1.5,0) {};
	\node[0c] (m) at (0,0) {};
	\node[0c] (1) at (1.5,0) {};
	\draw[1c, out=75, in=105, looseness=1.5] (0) to (1);
	\draw[1c, out=-75, in=-105, looseness=1.5] (0) to (1);
	\draw[1c, out=60,in=120] (0) to (m);
	\draw[1c, out=-60,in=-120] (0) to (m);
	\draw[1c] (m) to node[auto] {$z$} (1);
	\draw[2c] (0,-1.2) to node[auto,swap] {$x$} (0,-.2);
	\draw[2c] (0,.2) to node[auto,swap] {$x'$} (0,1.2);
	\draw[2c] (-.75,-.5) to node[auto,swap] {$y$} (-.75,.5);
	\node[scale=1.25] at (2,-1.25) {,};
\end{scope}
\begin{scope}[shift={(5,0)}]
	\node[0c] (0) at (-1.5,0) {};
	\node[0c] (m) at (0,0) {};
	\node[0c] (1) at (1.5,0) {};
	\draw[1c, out=75, in=105, looseness=1.5] (0) to (1);
	\draw[1c, out=-75, in=-105, looseness=1.5] (0) to (1);
	\draw[1c, out=60,in=120] (m) to (1);
	\draw[1c] (0) to node[auto] {$z$} (m);
	\draw[1c, out=-60,in=-120] (m) to (1);
	\draw[2c] (0,-1.2) to node[auto,swap] {$x$} (0,-.2);
	\draw[2c] (0,.2) to node[auto,swap] {$x'$} (0,1.2);
	\draw[2c] (.75,-.5) to node[auto,swap] {$y$} (.75,.5);
	\node[scale=1.25] at (2,-1.25) {.};
\end{scope}
\end{tikzpicture}
\end{equation*}
There are surjective maps $p_1: (O^2 \celto U_1) \surj O^2$, $p_2: (O^2 \celto U_2) \surj O^2$, which
\begin{enumerate}
	\item send the greatest element to $\fnct{2}$,
	\item are equal to the identity on $O^2$,
	\item send $y$ to $\fnct{2}$ and collapse $z, x, x'$ onto their boundaries.
\end{enumerate}
Then in $X$ there is a spherical 3-diagram
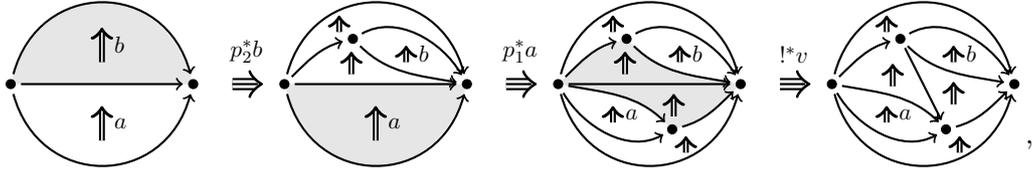
\begin{equation*}
\begin{tikzpicture}[baseline={([yshift=-.5ex]current bounding box.center)},scale=.8]
\begin{scope}
	\path[fill, color=gray!20] (-1.5,0) to [out=75,in=105,looseness=1.6] (1.5,0) -- cycle;
	\node[0c] (0) at (-1.5,0) {};
	\node[0c] (1) at (1.5,0) {};
	\draw[1c, out=75, in=105, looseness=1.5] (0) to (1);
	\draw[1c, out=-75, in=-105, looseness=1.5] (0) to (1);
	\draw[1c] (0) to (1);
	\draw[2c] (0,-1.2) to node[auto,swap] {$\,a$} (0,-.1);
	\draw[2c] (0,.1) to node[auto,swap] {$\,b$} (0,1.2);
	\draw[3c1] (1.75,0) to node[above=6pt] {$p_2^*b$} (3,0);
	\draw[3c2] (1.75,0) to (3,0);
	\draw[3c3] (1.75,0) to (3,0);
\end{scope}
\begin{scope}[shift={(4.5,0)}]
	\path[fill, color=gray!20] (-1.5,0) to [out=-75,in=-105,looseness=1.6] (1.5,0) -- cycle;
	\node[0c] (0) at (-1.5,0) {};
	\node[0c] (1) at (1.5,0) {};
	\node[0c] (m1) at (-.375,.75) {};
	\draw[1c, out=75, in=105, looseness=1.5] (0) to (1);
	\draw[1c, out=-75, in=-105, looseness=1.5] (0) to (1);
	\draw[1c] (0) to (1);
	\draw[1c, out=45,in=-165] (0) to (m1);
	\draw[1c, out=30,in=120] (m1) to (1);
	\draw[1c, out=-45,in=170] (m1) to (1);
	\draw[2c] (0,-1.2) to node[auto,swap] {$\,a$} (0,-.1);
	\draw[2c] (.5,.1) to node[auto,swap] {$\,b$} (.5,.9);
	\draw[2c] (-.4,-.1) to (-.4,.8);
	\draw[2c] (-.6,.6) to (-.6,1.4);
	\draw[3c1] (1.75,0) to node[above=6pt] {$p_1^*a$} (3,0);
	\draw[3c2] (1.75,0) to (3,0);
	\draw[3c3] (1.75,0) to (3,0);
\end{scope}
\begin{scope}[shift={(9,0)}]
	\path[fill, color=gray!20] (-1.5,0) to [out=-10,in=135] (.375,-.75) to [out=15,in=-135] (1.5,0) to [out=170,in=-45] (-.375,.75) to [out=-165,in=45] (-1.5,0);
	\node[0c] (0) at (-1.5,0) {};
	\node[0c] (1) at (1.5,0) {};
	\node[0c] (m1) at (-.375,.75) {};
	\node[0c] (m2) at (.375,-.75) {};
	\draw[1c, out=75, in=105, looseness=1.5] (0) to (1);
	\draw[1c, out=-75, in=-105, looseness=1.5] (0) to (1);
	\draw[1c] (0) to (1);
	\draw[1c, out=45,in=-165] (0) to (m1);
	\draw[1c, out=30,in=120] (m1) to (1);
	\draw[1c, out=-45,in=170] (m1) to (1);
	\draw[1c, out=15,in=-135] (m2) to (1);
	\draw[1c, out=-60,in=-150] (0) to (m2);
	\draw[1c, out=-10,in=135] (0) to (m2);
	\draw[2c] (-.6,-.9) to node[auto,swap] {$\,a$} (-.6,-.1);
	\draw[2c] (.5,.1) to node[auto,swap] {$\,b$} (.5,.9);
	\draw[2c] (-.4,-.1) to (-.4,.8);
	\draw[2c] (-.6,.6) to (-.6,1.4);
	\draw[2c] (.4,-.8) to (.4,.1);
	\draw[2c] (.6,-1.4) to (.6,-.6);
	\draw[3c1] (1.75,0) to node[above=6pt] {$!^*v$} (3,0);
	\draw[3c2] (1.75,0) to (3,0);
	\draw[3c3] (1.75,0) to (3,0);
\end{scope}
\begin{scope}[shift={(13.5,0)}]
	\node[0c] (0) at (-1.5,0) {};
	\node[0c] (1) at (1.5,0) {};
	\node[0c] (m1) at (-.375,.75) {};
	\node[0c] (m2) at (.375,-.75) {};
	\draw[1c, out=75, in=105, looseness=1.5] (0) to (1);
	\draw[1c, out=-75, in=-105, looseness=1.5] (0) to (1);
	\draw[1c] (m1) to (m2);
	\draw[1c, out=45,in=-165] (0) to (m1);
	\draw[1c, out=30,in=120] (m1) to (1);
	\draw[1c, out=-45,in=170] (m1) to (1);
	\draw[1c, out=15,in=-135] (m2) to (1);
	\draw[1c, out=-60,in=-150] (0) to (m2);
	\draw[1c, out=-10,in=135] (0) to (m2);
	\draw[2c] (-.6,-.9) to node[auto,swap] {$\,a$} (-.6,-.1);
	\draw[2c] (.5,.1) to node[auto,swap] {$\,b$} (.5,.9);
	\draw[2c] (-.5,-.3) to (-.5,.6);
	\draw[2c] (-.6,.6) to (-.6,1.4);
	\draw[2c] (.5,-.6) to (.5,.3);
	\draw[2c] (.6,-1.4) to (.6,-.6);
	\node[scale=1.25] at (1.75,-1) {,};
\end{scope}
\end{tikzpicture}
\end{equation*}
where the shaded area indicates the input boundary of the following cell. Now, let $V_1$ and $V_2$ be the 3-atoms
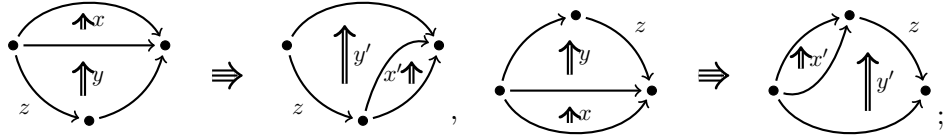
\begin{equation*}
\begin{tikzpicture}[baseline={([yshift=-.5ex]current bounding box.center)},scale=.8]
\begin{scope}
	\node[0c] (a) at (-1.25,.5) {};
	\node[0c] (b) at (0,-.75) {};
	\node[0c] (c) at (1.25,.5) {};
	\draw[1c, out=60, in=120] (a) to (c);
	\draw[1c, out=-75, in=165] (a) to node[auto,swap] {$z$} (b);
	\draw[1c, out=15, in=-105] (b) to (c);
	\draw[1c] (a) to (c);
	\draw[2c] (-.1,.5) to node[auto,swap] {$x$} (-.1,1.3);
	\draw[2c] (-.1,-.6) to node[auto,swap] {$y$} (-.1,.5);
	\draw[3c1] (1.625,0) to (2.875,0);
	\draw[3c2] (1.625,0) to (2.875,0);
	\draw[3c3] (1.625,0) to (2.875,0);
\end{scope}
\begin{scope}[shift={(4.5,0)}]
	\node[0c] (a) at (-1.25,.5) {};
	\node[0c] (b) at (0,-.75) {};
	\node[0c] (c) at (1.25,.5) {};
	\draw[1c, out=60, in=120] (a) to (c);
	\draw[1c, out=-75, in=165] (a) to node[auto,swap] {$z$} (b);
	\draw[1c, out=15, in=-105] (b) to (c);
	\draw[1c, out=75, in=165] (b) to (c);
	\draw[2c] (.8,-.4) to node[auto] {$x'$} (.8,.5);
	\draw[2c] (-.3,-.4) to node[auto,swap] {$y'$} (-.3,1.1);
	\node[scale=1.25] at (1.5,-.75) {,};
\end{scope}
\begin{scope}[shift={(8,.25)}]
	\node[0c] (a) at (-1.25,-.5) {};
	\node[0c] (b) at (0,.75) {};
	\node[0c] (c) at (1.25,-.5) {};
	\draw[1c, out=-60, in=-120] (a) to (c);
	\draw[1c, out=75, in=-165] (a) to (b);
	\draw[1c, out=-15, in=105] (b) to node[auto] {$z$} (c);
	\draw[1c] (a) to (c);
	\draw[2c] (-.1,-1.3) to node[auto,swap] {$x$} (-.1,-.5);
	\draw[2c] (-.1,-.5) to node[auto,swap] {$y$} (-.1,.6);
	\draw[3c1] (1.625,-.25) to (2.875,-.25);
	\draw[3c2] (1.625,-.25) to (2.875,-.25);
	\draw[3c3] (1.625,-.25) to (2.875,-.25);
\end{scope}
\begin{scope}[shift={(12.5,.25)}]
	\node[0c] (a) at (-1.25,-.5) {};
	\node[0c] (b) at (0,.75) {};
	\node[0c] (c) at (1.25,-.5) {};
	\draw[1c, out=-60, in=-120] (a) to (c);
	\draw[1c, out=75, in=-165] (a) to (b);
	\draw[1c, out=-15, in=105] (b) to node[auto] {$z$} (c);
	\draw[1c, out=-15, in=-105] (a) to (b);
	\draw[2c] (-.8,-.4) to node[auto,swap] {$x'$} (-.8,.5);
	\draw[2c] (.3,-1.1) to node[auto,swap] {$y'$} (.3,.4);
	\node[scale=1.25] at (1.5,-1) {;};
\end{scope}
\end{tikzpicture}
\end{equation*}
there are surjective maps $q_1: V_1 \surj O^2$, $q_2: V_2 \surj O^2$, which
\begin{enumerate}
	\item send the greatest element to $\fnct{2}$,
	\item are equal to the identity on $\clos\{x\}$ and $\clos\{x'\}$, and
	\item collapse $z, y, y'$ onto their boundaries.
\end{enumerate}
We then have a 3-diagram in $X$
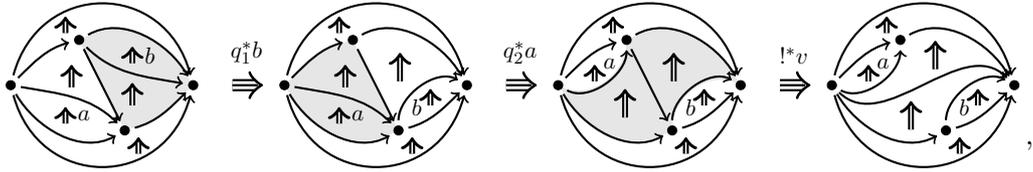
\begin{equation*}
\begin{tikzpicture}[baseline={([yshift=-.5ex]current bounding box.center)},scale=.8]
\begin{scope}
	\path[fill, color=gray!20] (-.375,.75) to (.375,-.75) to [out=15,in=-135] (1.5,0) to [out=120,in=30, looseness=1.1] (-.375,.75);
	\node[0c] (0) at (-1.5,0) {};
	\node[0c] (1) at (1.5,0) {};
	\node[0c] (m1) at (-.375,.75) {};
	\node[0c] (m2) at (.375,-.75) {};
	\draw[1c, out=75, in=105, looseness=1.5] (0) to (1);
	\draw[1c, out=-75, in=-105, looseness=1.5] (0) to (1);
	\draw[1c] (m1) to (m2);
	\draw[1c, out=45,in=-165] (0) to (m1);
	\draw[1c, out=30,in=120] (m1) to (1);
	\draw[1c, out=-45,in=170] (m1) to (1);
	\draw[1c, out=15,in=-135] (m2) to (1);
	\draw[1c, out=-60,in=-150] (0) to (m2);
	\draw[1c, out=-10,in=135] (0) to (m2);
	\draw[2c] (-.6,-.9) to node[auto,swap] {$\,a$} (-.6,-.1);
	\draw[2c] (.5,.1) to node[auto,swap] {$\,b$} (.5,.9);
	\draw[2c] (-.5,-.3) to (-.5,.6);
	\draw[2c] (-.6,.6) to (-.6,1.4);
	\draw[2c] (.5,-.6) to (.5,.3);
	\draw[2c] (.6,-1.4) to (.6,-.6);
	\draw[3c1] (1.75,0) to node[above=6pt] {$q_1^*b$} (3,0);
	\draw[3c2] (1.75,0) to (3,0);
	\draw[3c3] (1.75,0) to (3,0);
\end{scope}
\begin{scope}[shift={(4.5,0)}]
	\path[fill, color=gray!20] (-1.5,0) to [out=-60,in=-150, looseness=1.1] (.375,-.75) to (-.375,.75) to [out=-165, in=45] (-1.5,0);
	\node[0c] (0) at (-1.5,0) {};
	\node[0c] (1) at (1.5,0) {};
	\node[0c] (m1) at (-.375,.75) {};
	\node[0c] (m2) at (.375,-.75) {};
	\draw[1c, out=75, in=105, looseness=1.5] (0) to (1);
	\draw[1c, out=-75, in=-105, looseness=1.5] (0) to (1);
	\draw[1c] (m1) to (m2);
	\draw[1c, out=45,in=-165] (0) to (m1);
	\draw[1c, out=30,in=120] (m1) to (1);
	\draw[1c, out=15,in=-135] (m2) to (1);
	\draw[1c, out=90,in=150] (m2) to (1);
	\draw[1c, out=-60,in=-150] (0) to (m2);
	\draw[1c, out=-10,in=135] (0) to (m2);
	\draw[2c] (-.6,-.9) to node[auto,swap] {$\,a$} (-.6,-.1);
	\draw[2c] (.9,-.6) to node[auto,pos=.3] {$b$} (.9,.2);
	\draw[2c] (-.5,-.3) to (-.5,.6);
	\draw[2c] (.4,-.2) to (.4,.8);
	\draw[2c] (-.6,.6) to (-.6,1.4);
	\draw[2c] (.6,-1.4) to (.6,-.6);
	\draw[3c1] (1.75,0) to node[above=6pt] {$q_2^*a$} (3,0);
	\draw[3c2] (1.75,0) to (3,0);
	\draw[3c3] (1.75,0) to (3,0);
\end{scope}
\begin{scope}[shift={(9,0)}]
	\path[fill, color=gray!20] (-1.5,0) to [out=-60,in=-150, looseness=1.1] (.375,-.75) to [out=90,in=150, looseness=1.1] (1.5,0) to [out=120,in=30, looseness=1.1] (-.375,.75) to [out=-90, in=-30, looseness=1.1] (-1.5,0);
	\node[0c] (0) at (-1.5,0) {};
	\node[0c] (1) at (1.5,0) {};
	\node[0c] (m1) at (-.375,.75) {};
	\node[0c] (m2) at (.375,-.75) {};
	\draw[1c, out=75, in=105, looseness=1.5] (0) to (1);
	\draw[1c, out=-75, in=-105, looseness=1.5] (0) to (1);
	\draw[1c] (m1) to (m2);
	\draw[1c, out=45,in=-165] (0) to (m1);
	\draw[1c, out=-30,in=-90] (0) to (m1);
	\draw[1c, out=30,in=120] (m1) to (1);
	\draw[1c, out=15,in=-135] (m2) to (1);
	\draw[1c, out=90,in=150] (m2) to (1);
	\draw[1c, out=-60,in=-150] (0) to (m2);
	\draw[2c] (-.9,-.2) to node[auto,swap,pos=.7] {$a$} (-.9,.6);
	\draw[2c] (.9,-.6) to node[auto,pos=.3] {$b$} (.9,.2);
	\draw[2c] (-.4,-.8) to (-.4,.2);
	\draw[2c] (.4,-.2) to (.4,.8);
	\draw[2c] (-.6,.6) to (-.6,1.4);
	\draw[2c] (.6,-1.4) to (.6,-.6);
	\draw[3c1] (1.75,0) to node[above=6pt] {$!^*v$} (3,0);
	\draw[3c2] (1.75,0) to (3,0);
	\draw[3c3] (1.75,0) to (3,0);
\end{scope}
\begin{scope}[shift={(13.5,0)}]
	\node[0c] (0) at (-1.5,0) {};
	\node[0c] (1) at (1.5,0) {};
	\node[0c] (m1) at (-.375,.75) {};
	\node[0c] (m2) at (.375,-.75) {};
	\draw[1c, out=75, in=105, looseness=1.5] (0) to (1);
	\draw[1c, out=-75, in=-105, looseness=1.5] (0) to (1);
	\draw[1c, out=-45, in=135] (0) to (1);
	\draw[1c, out=45,in=-165] (0) to (m1);
	\draw[1c, out=-30,in=-90] (0) to (m1);
	\draw[1c, out=30,in=120] (m1) to (1);
	\draw[1c, out=15,in=-135] (m2) to (1);
	\draw[1c, out=90,in=150] (m2) to (1);
	\draw[1c, out=-60,in=-150] (0) to (m2);
	\draw[2c] (-.9,-.2) to node[auto,swap,pos=.7] {$a$} (-.9,.6);
	\draw[2c] (.9,-.6) to node[auto,pos=.3] {$b$} (.9,.2);
	\draw[2c] (-.2,-1) to (-.2,0);
	\draw[2c] (.2,-0) to (.2,1);
	\draw[2c] (-.6,.6) to (-.6,1.4);
	\draw[2c] (.6,-1.4) to (.6,-.6);
	\node[scale=1.25] at (1.75,-1) {,};
\end{scope}
\end{tikzpicture}
\end{equation*}
and finally, letting $p_1'$ and $p_2'$ be the ``reverse'' surjective maps of $p_1$ and $p_2$, a 3-diagram
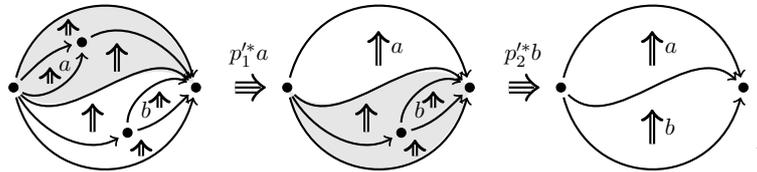
\begin{equation*}
\begin{tikzpicture}[baseline={([yshift=-.5ex]current bounding box.center)},scale=.8]
\begin{scope}
	\path[fill, color=gray!20] (-1.5,0) to [out=75,in=105,looseness=1.6] (1.5,0) to [out=135,in=-45, looseness=1.1] (-1.5,0);
	\node[0c] (0) at (-1.5,0) {};
	\node[0c] (1) at (1.5,0) {};
	\node[0c] (m1) at (-.375,.75) {};
	\node[0c] (m2) at (.375,-.75) {};
	\draw[1c, out=75, in=105, looseness=1.5] (0) to (1);
	\draw[1c, out=-75, in=-105, looseness=1.5] (0) to (1);
	\draw[1c, out=-45, in=135] (0) to (1);
	\draw[1c, out=45,in=-165] (0) to (m1);
	\draw[1c, out=-30,in=-90] (0) to (m1);
	\draw[1c, out=30,in=120] (m1) to (1);
	\draw[1c, out=15,in=-135] (m2) to (1);
	\draw[1c, out=90,in=150] (m2) to (1);
	\draw[1c, out=-60,in=-150] (0) to (m2);
	\draw[2c] (-.9,-.2) to node[auto,swap,pos=.7] {$a$} (-.9,.6);
	\draw[2c] (.9,-.6) to node[auto,pos=.3] {$b$} (.9,.2);
	\draw[2c] (-.2,-1) to (-.2,0);
	\draw[2c] (.2,-0) to (.2,1);
	\draw[2c] (-.6,.6) to (-.6,1.4);
	\draw[2c] (.6,-1.4) to (.6,-.6);
	\draw[3c1] (1.75,0) to node[above=6pt] {$p_1'^*a$} (3,0);
	\draw[3c2] (1.75,0) to (3,0);
	\draw[3c3] (1.75,0) to (3,0);
\end{scope}
\begin{scope}[shift={(4.5,0)}]
	\path[fill, color=gray!20] (-1.5,0) to [out=-75,in=-105,looseness=1.6] (1.5,0) to [out=135,in=-45, looseness=1.1] (-1.5,0);
	\node[0c] (0) at (-1.5,0) {};
	\node[0c] (1) at (1.5,0) {};
	\node[0c] (m2) at (.375,-.75) {};
	\draw[1c, out=75, in=105, looseness=1.5] (0) to (1);
	\draw[1c, out=-75, in=-105, looseness=1.5] (0) to (1);
	\draw[1c, out=-45, in=135] (0) to (1);
	\draw[1c, out=15,in=-135] (m2) to (1);
	\draw[1c, out=90,in=150] (m2) to (1);
	\draw[1c, out=-60,in=-150] (0) to (m2);
	\draw[2c] (0,.1) to node[auto,swap] {$\,a$} (0,1.2);
	\draw[2c] (.9,-.6) to node[auto,pos=.3] {$b$} (.9,.2);
	\draw[2c] (-.2,-1) to (-.2,0);
	\draw[2c] (.6,-1.4) to (.6,-.6);
	\draw[3c1] (1.75,0) to node[above=6pt] {$p_2'^*b$} (3,0);
	\draw[3c2] (1.75,0) to (3,0);
	\draw[3c3] (1.75,0) to (3,0);
\end{scope}
\begin{scope}[shift={(9,0)}]
	\node[0c] (0) at (-1.5,0) {};
	\node[0c] (1) at (1.5,0) {};
	\draw[1c, out=75, in=105, looseness=1.5] (0) to (1);
	\draw[1c, out=-75, in=-105, looseness=1.5] (0) to (1);
	\draw[1c, out=-45, in=135] (0) to (1);
	\draw[2c] (0,.1) to node[auto,swap] {$\,a$} (0,1.2);
	\draw[2c] (0,-1.2) to node[auto,swap] {$\,b$} (0,-.1);
	\node[scale=1.25] at (1.75,-1) {.};
\end{scope}
\end{tikzpicture}
\end{equation*}
Putting together the three 3-diagrams, we obtain a spherical diagram $h: a \cp{1} b \celto b \cp{1} a$, which works as a braiding of $a$ and $b$. Starting from $p_1^*b$ and $p_2^*a$, instead, we can construct a second, inequivalent braiding, where $a$ is moved to the right of $b$. 
\end{exm}

\section{Homotopy hypothesis} \label{sec:homotopy}

\subsection{Simplices as molecules} \label{sec:simplices}

Throughout this part, we will assume some knowledge of simplicial sets and their homotopy theory; see \cite[Chapter 4]{fritsch1990cellular} or \cite{goerss2009simplicial} as references. We also borrow most of our notation from the latter.

Here, we treat the combinatorics that we will need in the rest of the article. First of all, we look at how simplices can be seen as particular molecules, and simplicial sets as particular diagrammatic sets. 

\begin{dfn}
Let $\deltacat$ be the simplex category, and $\deltain$ its subcategory of monomorphisms. A \emph{simplicial set} is a presheaf on $\deltacat$, and a \emph{semi-simplicial set} is a presheaf on $\deltain$. With morphisms of presheaves, these form categories $\sset$ and $\semisset$, respectively. 
\end{dfn}

\begin{cons}
For all $n \geq 0$, the $(n+1)$\nbd fold join $\Delta^n := \underbrace{1\join\ldots\join 1}_{n+1}$ can be identified with Street's oriented $n$\nbd simplex \cite{street1987algebra}: for $0 \leq k \leq n$, 
\begin{itemize}
	\item the $k$\nbd th co-face map $d^k: \Delta^{n-1} \incl \Delta^{n}$ is the inclusion 
\begin{equation*}
\begin{tikzpicture}[baseline={([yshift=-.5ex]current bounding box.center)}]
	\node[scale=1.25] (0) at (-1.5,0) {$\underbrace{1\join\ldots\join 1}_{k} \join \emptyset \join \underbrace{1\join\ldots\join 1}_{n-k}$};
	\node[scale=1.25] (1) at (4,0) {$\underbrace{1\join\ldots\join 1}_{n+1},$};
	\draw[1cinc] (0.east |- 0,.15) to node[auto] {$\mathrm{id}\join \imath \join \mathrm{id}$} (1.west  |- 0,.15);
\end{tikzpicture}
\end{equation*} 
where $\imath: \emptyset \to 1$ is the unique inclusion of the initial object;
\item the $k$\nbd th co-degeneracy map $s^k: \Delta^{n+1} \surj \Delta^{n}$ is the map
\begin{equation*}
\begin{tikzpicture}[baseline={([yshift=-.5ex]current bounding box.center)}]
	\node[scale=1.25] (0) at (-3,0) {$\underbrace{1\join\ldots\join 1}_{k} \join 1 \join 1 \join \underbrace{1\join\ldots\join 1}_{n-k}$};
	\node[scale=1.25] (1) at (4,0) {$\underbrace{1\join\ldots\join 1}_{k} \join 1 \join \underbrace{1\join\ldots\join 1}_{n-k}$,};
	\draw[1cinc] (0.east |- 0,.15) to node[auto] {$\mathrm{id}\join p \join \mathrm{id}$} (1.west  |- 0,.15);
\end{tikzpicture}
\end{equation*} 
where $p: 1\join 1 \to 1$ is the unique map to the terminal object.
\end{itemize}
The co-faces $d^k$ with $k$ odd correspond to elements of the output boundary, and those with $k$ even to elements of the input boundary of $\Delta^n$. 

These maps satisfy the cosimplicial identities, so they define inclusions of subcategories $\deltain \incl \atomin$ and $\deltacat \incl \atom$. Consequently, there are restriction functors $\delres{-}: \dgmset \to \sset$ and $\rpol \to \semisset$, with left adjoints $\imath_\Delta: \sset \to \dgmset$ and $\semisset \to \rpol$.

With the characterisation of $P \join Q$ by $(P \join Q)_\bot \simeq P_\bot \tensor Q_\bot$, if $1_\bot = \{\bot < \top\}$, every element $x \in \Delta^n$ is uniquely of the form
\begin{equation*}
	\top^{j_1}\bot^{k_1}\ldots\top^{j_{m}}\bot^{k_m} := \underbrace{\top \tensor \ldots \tensor \top}_{j_1} \tensor \underbrace{\bot \tensor \ldots \tensor \bot}_{k_1} \tensor \ldots \tensor \underbrace{\top \tensor \ldots \tensor \top}_{j_m} \tensor \underbrace{\bot \tensor \ldots \tensor \bot}_{k_m},
\end{equation*}
where $m \geq 1$, $j_i \neq 0$ if $i \neq 0$ and $k_i \neq 0$ if $i \neq m$, $\sum_i j_i = \dmn{x} + 1 > 0$, and $\sum_i (j_i + k_i) = n+1$. 
\end{cons}

\begin{prop}
The inclusions $\deltain \incl \atomin$ and $\deltacat \incl \atom$ are full.
\end{prop}
\begin{proof}
The atom $\Delta^n$ has exactly $n+1$ elements of dimension $n-1$, so the maps $d^k: \Delta^{n-1} \incl \Delta^n$ exhaust the inclusions of $\Delta^{n-1}$ into $\Delta^n$. Since any inclusion $\Delta^k \incl \Delta^n$ for $k < n-1$ factors through $\Delta^{n-1}$, this suffices to prove that $\deltain$ is a full subcategory of $\atomin$. 

Let $p: \Delta^n \surj \Delta^k$ be a surjective map; if $n = k$ it must be the identity, so suppose $n > k$. Recall that a morphism $\Delta^n \to \Delta^m$ in $\deltacat$ corresponds uniquely to an order-preserving map $f: [n] \to [m]$, where $[n]$ is the linear order $\{0 < \ldots < n\}$: in particular, the $i$\nbd th co-face $d^i: [n-1] \to [n]$ is the inclusion of $\{0 < \ldots < i-1 < i+1 < \ldots < n\}$.

By induction, the restrictions $d^i;p: \Delta^{n-1} \to \Delta^k$ to the faces of $\Delta^n$, for $0 \leq i \leq n$, are all in $\deltacat$, hence they correspond to a system of order-preserving maps $f_i: [n-1] \to [k]$, with the property that $d^i;f_j = d^{j-1};f_i$ for $i < j$. 

We may suppose $n > 1$, handling the few remaning cases explicitly: then any pair of elements of $[n]$ is in the image of $d^i$ for some $i$, and any pair $f_i$, $f_j$ agrees on the overlaps. The system therefore extends uniquely to an order-preserving map $f: [n] \to [k]$ such that $f_i = d^i;f$ for $0 \leq i \leq n$, corresponding to a map $f: \Delta^n \to \Delta^k$ in $\deltacat$ which is equal to $p$ on its boundary. Both $f$ and $p$ must map the greatest element of $\Delta^n$ to the greatest element of $\Delta^k$, and we conclude that $f = p$. 
\end{proof}

\begin{cor}
The functors $\imath_\Delta: \sset \to \dgmset$ and $\imath_\Delta: \semisset \to \rpol$ are full and faithful.
\end{cor}
\begin{proof}
Follows from the proposition and Lemma \ref{lem:leftadj_full}.
\end{proof}

\begin{remark}
Every simplicial set $K$, seen as a diagrammatic set $\imath_\Delta K$, satisfies the Eilenberg-Zilber property: $K$ can always be built as the colimit of the sequence of pushouts 
\begin{equation*}
\begin{tikzpicture}[baseline={([yshift=-.5ex]current bounding box.center)}]
	\node[scale=1.25] (0) at (0,2) {$\displaystyle \coprod_{x\in\nondeg{n}{K}} \bord{}{} \Delta^n$};
	\node[scale=1.25] (1) at (3.5,2) {$\displaystyle\coprod_{x\in\nondeg{n}{K}} \Delta^n$};
	\node[scale=1.25] (2) at (0,0) {$\imath^\deltacat_{n-1}\sigma^\deltacat_{\leq n-1}K$};
	\node[scale=1.25] (3) at (3.5,0) {$\imath^\deltacat_n\sigma^\deltacat_{\leq n}K$};
	\draw[1c] (0) to node[auto,swap] {$(\bord{}{}x)_{x\in\nondeg{n}{K}}$} (2);
	\draw[1c] (1) to node[auto] {$(x)_{x\in\nondeg{n}{K}}$} (3);
	\draw[1cinc] (0) to (1);
	\draw[1cinc] (2) to (3);
	\draw[edge] (2.5,0.2) to (2.5,0.8) to (3.3,0.8);
\end{tikzpicture}
\end{equation*}
where $\nondeg{n}{K}$ is the set of non-degenerate $n$\nbd simplices of $K$. Because $\imath_\Delta$ is left adjoint, it preserves pushouts: the claim then follows from Proposition \ref{prop:ez_polygraph}, after observing that the irreducible elements of $\imath_\Delta K$ correspond to the non-degenerate simplices of $K$, and that $\imath_\Delta$ commutes with skeleta.
\end{remark}

We will now define a sequence of maps $a_n: \Delta^n \surj O^n$ relating the $n$\nbd simplices to the $n$\nbd globes, which we will use in the next section to transform $n$\nbd cells of a diagrammatic set $X$ into $n$\nbd simplices of $X_\Delta$. Recall Construction \ref{cons:ou} and Construction \ref{cons:fattening}.

\begin{dfn} \label{dfn:an_maps}
Let $a_0 := \idcat{}: \Delta^0 \surj O^0$. For $n \geq 0$, if $a_n$ is defined, we define $a_{n+1}: \Delta^{n+1} \surj O^{n+1}$ to be the composite
\begin{equation*}
	\begin{tikzpicture}[baseline={([yshift=-.5ex]current bounding box.center)}]
	\node[scale=1.25] (0) at (-3,0) {$\Delta^{n+1}$};
	\node[scale=1.25] (1) at (0,0) {$O(\Delta^n)$};
	\node[scale=1.25] (2) at (3,0) {$O^{n+1}.$};
	\draw[1csurj] (0) to node[auto] {$s^0_\prec$} (1);
	\draw[1csurj] (1) to node[auto] {$O(a_n)$} (2);
\end{tikzpicture}
\end{equation*}
Unravelling the recursion, we see that $a_n = s^0_\prec;O(s^0_\prec);\ldots;O^{n-1}(s^0_\prec)$. 

These maps make the following diagrams commute, for $n \geq 0$:
\begin{equation} \label{eq:an_faces}
\begin{tikzpicture}[baseline={([yshift=-.5ex]current bounding box.center)}]
	\node[scale=1.25] (0) at (0,1.5) {$\Delta^n$};
	\node[scale=1.25] (1) at (2.5,0) {$O^{n+1}$,};
	\node[scale=1.25] (2) at (0,0) {$\Delta^{n+1}$};
	\node[scale=1.25] (3) at (2.5,1.5) {$O^n$};
	\draw[1csurj] (0) to node[auto] {$a_n$} (3);
	\draw[1cinc] (0) to node[auto,swap] {$d^0$} (2);
	\draw[1csurj] (2) to node[auto,swap] {$a_{n+1}$} (1);
	\draw[1cinc] (3) to node[auto] {$\imath^+$} (1);
\end{tikzpicture} \quad \quad 
\begin{tikzpicture}[baseline={([yshift=-.5ex]current bounding box.center)}]
	\node[scale=1.25] (0) at (0,1.5) {$\Delta^n$};
	\node[scale=1.25] (1) at (2.5,0) {$O^{n+1}$.};
	\node[scale=1.25] (2) at (0,0) {$\Delta^{n+1}$};
	\node[scale=1.25] (3) at (2.5,1.5) {$O^n$};
	\draw[1csurj] (0) to node[auto] {$a_n$} (3);
	\draw[1cinc] (0) to node[auto,swap] {$d^1$} (2);
	\draw[1csurj] (2) to node[auto,swap] {$a_{n+1}$} (1);
	\draw[1cinc] (3) to node[auto] {$\imath^-$} (1);
\end{tikzpicture}
\end{equation} 
\end{dfn}

\begin{remark} \label{rmk:an_explicit}
The maps $a_n$ can be given a fairly simple explicit description: they are defined by
\begin{align*}
	\top^{n+1} & \; \mapsto \; \fnct{n}, \\
	\bot^k\top^j & \; \mapsto \; \fnct{(j-1)}^+, \\
	(\ldots)\top\bot^k\top^j & \; \mapsto \; \fnct{j}^-,
\end{align*}
where $k > 0$. Notice that any element ending with $\bot$ is sent to $\fnct{0}^-$.
\end{remark}

We will also need a modified sequence, mapping an $n$\nbd simplex to an $n$\nbd atom with \emph{two} $(n-1)$\nbd dimensional elements in its output boundary.

\begin{cons}
Let $\compglob{n+1}$ be the $(n+1)$\nbd atom $O^n \celto O^n \cp{n-1} O^n$. This is equal to $O^{n+1}$ with its output boundary split into two $n$\nbd globes, the images of inclusions $\imath^+_1, \imath^+_2: O^n \incl \compglob{n+1}$: where $O^{n+1}$ has a unique $n$\nbd dimensional element $\fnct{n}^+$, $\compglob{n+1}$ has two $n$\nbd dimensional elements $\fnct{n}_1^+$ and $\fnct{n}_2^+$, and an additional $(n-1)$\nbd dimensional element $\fnct{n-1}^0 \in \sbord{}{+}\fnct{n}_1^+ \cap \sbord{}{-}\fnct{n}_2^+$. 

There are two surjective maps $p_1, p_2: \compglob{n+1} \surj O^{n+1}$, the first sending $\fnct{n}_2^+$ to $\fnct{n}^+$, hence sending $\fnct{n-1}^0$ and $\fnct{n}_1^+$ to $\fnct{n-1}^-$; the other sending $\fnct{n}_1^+$ to $\fnct{n}^+$, hence sending $\fnct{n-1}^0$ and $\fnct{n}_2^+$ to $\fnct{n-1}^+$. 
\end{cons}

\begin{dfn}
We define a sequence of surjective maps $c_{n+1}: \Delta^{n+1} \surj \compglob{n+1}$ for $n > 0$. These are defined by
\begin{align*}
	\top^{n+2} & \; \mapsto \; \fnct{n+1}, \\
	(\bot\top^2)\top^{n-1} & \; \mapsto \; \fnct{n}_2^+, \\
	(\top^2\bot)\top^{n-1} & \; \mapsto \; \fnct{n}_1^+, \\
	(\bot\top\bot)\top^{n-1} & \; \mapsto \; \fnct{n-1}^0,
\end{align*}
and, for $k > 0$,
\begin{align*} \bot^k\top^j & \; \mapsto \; \fnct{(j-1)}^+, \\
	(\ldots)\top\bot^k\top^j & \; \mapsto \; \fnct{j}^-,
\end{align*}
for all elements not covered by the previous cases.

Using the explicit description of the $a_n$ from Remark \ref{rmk:an_explicit}, one can check that the following diagrams commute:
\begin{equation} \label{eq:cn_faces}
\begin{tikzpicture}[baseline={([yshift=-.5ex]current bounding box.center)}]
	\node[scale=1.25] (0) at (0,1.5) {$\Delta^n$};
	\node[scale=1.25] (1) at (2.5,0) {$\compglob{n+1}$,};
	\node[scale=1.25] (2) at (0,0) {$\Delta^{n+1}$};
	\node[scale=1.25] (3) at (2.5,1.5) {$O^n$};
	\draw[1csurj] (0) to node[auto] {$a_n$} (3);
	\draw[1cinc] (0) to node[auto,swap] {$d^0$} (2);
	\draw[1csurj] (2) to node[auto,swap] {$c_{n+1}$} (1);
	\draw[1cinc] (3) to node[auto] {$\imath_2^+$} (1);
\end{tikzpicture} \quad \quad 
\begin{tikzpicture}[baseline={([yshift=-.5ex]current bounding box.center)}]
	\node[scale=1.25] (0) at (0,1.5) {$\Delta^n$};
	\node[scale=1.25] (1) at (2.5,0) {$\compglob{n+1}$,};
	\node[scale=1.25] (2) at (0,0) {$\Delta^{n+1}$};
	\node[scale=1.25] (3) at (2.5,1.5) {$O^n$};
	\draw[1csurj] (0) to node[auto] {$a_n$} (3);
	\draw[1cinc] (0) to node[auto,swap] {$d^1$} (2);
	\draw[1csurj] (2) to node[auto,swap] {$c_{n+1}$} (1);
	\draw[1cinc] (3) to node[auto] {$\imath^-$} (1);
\end{tikzpicture} \quad \quad 
\begin{tikzpicture}[baseline={([yshift=-.5ex]current bounding box.center)}]
	\node[scale=1.25] (0) at (0,1.5) {$\Delta^n$};
	\node[scale=1.25] (1) at (2.5,0) {$\compglob{n+1}$,};
	\node[scale=1.25] (2) at (0,0) {$\Delta^{n+1}$};
	\node[scale=1.25] (3) at (2.5,1.5) {$O^n$};
	\draw[1csurj] (0) to node[auto] {$a_n$} (3);
	\draw[1cinc] (0) to node[auto,swap] {$d^2$} (2);
	\draw[1csurj] (2) to node[auto,swap] {$c_{n+1}$} (1);
	\draw[1cinc] (3) to node[auto] {$\imath_1^+$} (1);
\end{tikzpicture}
\end{equation} 
and 
\begin{equation} \label{eq:cn_projection}
\begin{tikzpicture}[baseline={([yshift=-.5ex]current bounding box.center)}]
	\node[scale=1.25] (0) at (-1.25,1.25) {$\Delta^{n+1}$};
	\node[scale=1.25] (1) at (0,0) {$\compglob{n+1}$};
	\node[scale=1.25] (2) at (1.25,1.25) {$O^{n+1}$};
	\draw[1csurj] (0) to node[auto] {$a_{n+1}$} (2);
	\draw[1csurj] (0) to node[auto,swap] {$c_{n+1}$} (1);
	\draw[1csurj] (1) to node[auto,swap] {$p_1$} (2);
	\node[scale=1.25] at (1.5,0) {.};
\end{tikzpicture}
\end{equation}
Notice also that $c_2: \Delta^2 \surj \compglob{2}$ is an isomorphism.
\end{dfn}

Finally, the following construction is crucial in the proof of Theorem \ref{thm:homotopy_group_iso}.

\begin{cons} \label{cons:delta_extract}
For $k \geq 0, n \geq 2$, we define regular $(k+n)$\nbd molecules $\extr{k}{n}$ with spherical boundary, with the property that
\begin{enumerate}
	\item for $\alpha \in \{+,-\}$, $\bord{}{\alpha}\extr{k}{n}$ is isomorphic to $\bord{}{\alpha}O^k(\Delta^n)$, and
	\item $O^{k+1}(\Delta^{n-1}) \subsph \extr{k}{n}$.
\end{enumerate}
For all $n \geq 2$, let $\extr{0}{n}$ be the pushout
\begin{equation*}
\begin{tikzpicture}[baseline={([yshift=-.5ex]current bounding box.center)}]
	\node[scale=1.25] (0) at (0,1.5) {$\Delta^{n-1}$};
	\node[scale=1.25] (1) at (2.5,0) {$\extr{0}{n};$};
	\node[scale=1.25] (2) at (0,0) {$\Delta^n$};
	\node[scale=1.25] (3) at (2.5,1.5) {$O(\Delta^{n-1})$};
	\draw[1cinc] (0) to node[auto] {$\imath^-$} (3);
	\draw[1cincl] (0) to node[auto,swap] {$d^0$} (2);
	\draw[1cinc] (2) to (1);
	\draw[1cincl] (3) to (1);
	\draw[edge] (1.6,0.2) to (1.6,0.7) to (2.3,0.7);
\end{tikzpicture}
\end{equation*}
this splits into $\Delta^n \cup O(\Delta^{n-1})$, and its boundary is isomorphic to the boundary of $\Delta^n = O^0(\Delta^n)$. 

Supposing that $\extr{k}{n}$ has been defined, and satisfies the two properties, we define $\extr{k+1}{n}$ to be the colimit of
\begin{equation*} 
\begin{tikzpicture}[baseline={([yshift=-.5ex]current bounding box.center)}]
	\node[scale=1.25] (0) at (-2,-1.5) {$O^k(\Delta^n) \celto \extr{k}{n}$};
	\node[scale=1.25] (1) at (0,0) {$O^{k+1}(\Delta^{n-1})$};
	\node[scale=1.25] (2) at (2,-1.5) {$O^{k+2}(\Delta^{n-1})$};
	\node[scale=1.25] (3) at (4,0) {$O^{k+1}(\Delta^{n-1})$};
	\node[scale=1.25] (4) at (6,-1.5) {$\extr{k}{n} \celto O^k(\Delta^n)$,};
	\draw[1cincl] (1) to node[auto,swap] {$j_1$} (0);
	\draw[1cinc] (1) to node[auto] {$\imath^-$} (2);
	\draw[1cincl] (3) to node[auto,swap] {$\imath^+$} (2);
	\draw[1cinc] (3) to node[auto] {$j_2$} (4);
\end{tikzpicture}
\end{equation*}
where $j_1$ and $j_2$ are the inclusion of $O^{k+1}(\Delta^{n-1})$ as a submolecule of $\extr{k}{n}$. This is a regular molecule with spherical boundary and three maximal elements, isomorphic to $O^k(\Delta^n) \celto \extr{k}{n}$, $O^{k+2}(\Delta^{n-1})$, and $\extr{k}{n} \celto O^k(\Delta^n)$, respectively.

Now, we define a sequence of retractions $r_{k,n}: \extr{k}{n} \surj O^{k+1}(\Delta^{n-1})$, with the property that
\begin{equation} \label{eq:retract_rkn}
\begin{tikzpicture}[baseline={([yshift=-.5ex]current bounding box.center)}]
	\node[scale=1.25] (0) at (0,1.5) {$\bord{}{}\extr{k}{n}$};
	\node[scale=1.25] (1) at (3,0) {$O^{k+1}(\Delta^{n-1})$};
	\node[scale=1.25] (2) at (0,0) {$\extr{k}{n}$};
	\node[scale=1.25] (3) at (3,1.5) {$O^{k}(\Delta^n)$};
	\draw[1cinc] (0) to (3);
	\draw[1cincl] (0) to (2);
	\draw[1csurj] (2) to node[auto,swap] {$r_{k,n}$} (1);
	\draw[1csurj] (3) to node[auto] {$O^k(s^0_\prec)$} (1);
\end{tikzpicture}
\end{equation}
commutes, where $\bord{}{}\extr{k}{n}$ is included isomorphically into $\bord{}{}O^k(\Delta^n)$:
\begin{itemize}
	\item $r_{0,n}: \extr{0}{n} \surj O(\Delta^{n-1})$ is equal to $s^0$ on $\Delta^n$, and to the identity on $O(\Delta^{n-1})$;
	\item if $r_{k,n}$ is defined, we let $r_{k+1,n}: \extr{k+1}{n} \surj O^{k+2}(\Delta^{n-1})$ be 
	\begin{enumerate}[leftmargin=5pt]
		\item equal to $r_{k,n}$ from the output boundary of $O^k(\Delta^n) \celto \extr{k}{n}$ to the input boundary of $O^{k+2}(\Delta^{n-1})$, and from the input boundary of $\extr{k}{n} \celto O^k(\Delta^n)$ to the output boundary of $O^{k+2}(\Delta^{n-1})$;
		\item equal to $O^k(s^0_\prec)$ from the input boundary of $O^k(\Delta^n) \celto \extr{k}{n}$ to the input boundary of $O^{k+2}(\Delta^{n-1})$, and from the output boundary of $\extr{k}{n} \celto O^k(\Delta^n)$ to the output boundary of $O^{k+2}(\Delta^{n-1})$;
		\item sending the greatest elements of $O^k(\Delta^n) \celto \extr{k}{n}$ and $\extr{k}{n} \celto O^k(\Delta^n)$ to the greatest elements of $\bord{}{-}O^{k+2}(\Delta^{n-1})$ and $\bord{}{+}O^{k+2}(\Delta^{n-1})$, respectively;
		\item equal to the identity on $O^{k+2}(\Delta^{n-1})$. 
	\end{enumerate}
\end{itemize}
Assuming (\ref{eq:retract_rkn}) commutes, the retraction $r_{k+1,n}$ is well-defined, and the commutativity of the respective diagram holds by construction.

Next, we define a modified sequence $\extrtil{k}{n}$ by
\begin{align*}
	\extrtil{k}{2} & := \extr{k}{2}, \\
	\extrtil{k}{n} & := \extr{k}{n}[\extrtil{k+1}{n-1}/O^{k+1}(\Delta^{n-1})] \quad \text{for $n > 2$}.
\end{align*}
This has the property that
\begin{equation*}
	\extrtil{k}{n} \supsph \extrtil{k+1}{n-1} \supsph \ldots \supsph \extrtil{k+n-2}{2} \supsph O^{k+n-1}(\Delta^1) = O^{k+n},
\end{equation*}
and the retractions $r_{k,n}$ induce a sequence of retractions
\begin{equation*}
	\extrtil{k}{n} \surj \extrtil{k+1}{n-1} \surj \ldots \surj \extrtil{k+n-2}{2} \surj O^{k+n},
\end{equation*}
whose composite we call $\tilde{r}_{k,n}: \extrtil{k}{n} \surj O^{k+n}$. 

In particular, $O^n \subsph \extrtil{0}{n}$ for all $n \geq 2$, and by the commutativity of (\ref{eq:retract_rkn}) and the definition of $a_n: \Delta^n \surj O^n$, the diagram
\begin{equation} \label{eq:retract_r0n}
\begin{tikzpicture}[baseline={([yshift=-.5ex]current bounding box.center)}]
	\node[scale=1.25] (0) at (0,1.5) {$\bord{}{}\Delta^n$};
	\node[scale=1.25] (1) at (2.5,0) {$O^n$};
	\node[scale=1.25] (2) at (0,0) {$\extrtil{0}{n}$};
	\node[scale=1.25] (3) at (2.5,1.5) {$\Delta^n$};
	\draw[1cinc] (0) to (3);
	\draw[1cincl] (0) to (2);
	\draw[1csurj] (2) to node[auto,swap] {$\tilde{r}_{0,n}$} (1);
	\draw[1csurj] (3) to node[auto] {$a_n$} (1);
\end{tikzpicture}
\end{equation}
commutes; here $\bord{}{}\Delta^n \incl \extrtil{0}{n}$ is the isomorphic inclusion into $\bord{}{}\extrtil{0}{n}$. 
\end{cons}

\subsection{Combinatorial homotopy groups} \label{sec:combinatorial}

In this section, we define a notion of fibration of diagrammatic sets, by analogy with Kan fibrations of simplicial sets, and then define combinatorial homotopy groups for a Kan (fibrant) diagrammatic set, in such a way that 
\begin{enumerate}
	\item $\delres{-}$ sends fibrations of diagrammatic sets to fibrations of simplicial sets, and in particular Kan diagrammatic sets to Kan complexes, and
	\item the combinatorial homotopy groups of a Kan diagrammatic set $X$ are isomorphic to the simplicial combinatorial homotopy groups of $\delres{X}$. 
\end{enumerate}
This will allow us to define a weak equivalence of Kan diagrammatic sets as a morphism which induces isomorphisms of homotopy groups, and know that $\delres{-}$ will send it to a weak equivalence of Kan complexes. 

\begin{dfn}
Let $f: X \to Y$ be a morphism of diagrammatic sets. We say that $f$ is a \emph{fibration} if it satisfies the right lifting property with respect to all horns $\Lambda \incl U$, that is, for all commutative squares
\begin{equation*}
\begin{tikzpicture}[baseline={([yshift=-.5ex]current bounding box.center)}]
	\node[scale=1.25] (0) at (0,1.5) {$\Lambda$};
	\node[scale=1.25] (1) at (1.5,1.5) {$X$};
	\node[scale=1.25] (2) at (0,0) {$U$};
	\node[scale=1.25] (3) at (1.5,0) {$Y$};
	\draw[1c] (0) to (1);
	\draw[1c] (2) to (3);
	\draw[1cinc] (0) to (2);
	\draw[1c] (1) to node[auto] {$f$} (3);
	\draw[1cdash] (2) to (1);
\end{tikzpicture}
\end{equation*}
in $\dgmset$, there exists a morphism $U \to X$ as pictured that makes both triangles commute.

A diagrammatic set $X$ is \emph{Kan} if the unique morphism $X \to 1$ is a fibration.
\end{dfn}

\begin{prop} \label{prop:kan_allequiv}
Let $X$ be a Kan diagrammatic set. Then $X$ is representable and all cells of $X$ are equivalences.
\end{prop}
\begin{proof}
Since all division horns in $X$ have fillers, the set $S$ of all cells of $X$ satisfies $S \subseteq \mathcal{F}(S)$, where $\mathcal{F}$ is the functor of Remark \ref{rmk:coinduction}, so by coinduction $S = \equi{}{X}$. Moreover, all composition horns in $X$ have fillers, which are necessarily equivalences, so $X$ is representable.
\end{proof}

\begin{remark}
It follows that, for two parallel spherical $n$\nbd diagrams $x, y$ in a Kan diagrammatic set, we have $x \simeq y$ if and only if there is a cell $e: x \celto y$.
\end{remark}

Recall that a Kan fibration is a map of simplicial sets with the right lifting property with respect to all horns $\Lambda_k^n \incl \Delta^n$, where the image of $\Lambda_k^n$ is $\Delta^n$ minus its $n$\nbd simplex and the $k$\nbd th face, for $0 \leq k \leq n$. These correspond biunivocally to the horns of $\Delta^n$ seen as an $n$\nbd atom; moreover, any morphism $\Delta^n \to X$ or $\Lambda_k^n \to X$ factors through $\imath_\Delta \delres{X} \to X$. Thus we have the following.

\begin{prop} 
Let $f: X \to Y$ be a fibration of diagrammatic sets. Then the map $\delres{f}: \delres{X} \to \delres{Y}$ is a Kan fibration.
\end{prop}

\begin{cor}
Let $X$ be a Kan diagrammatic set. Then $\delres{X}$ is a Kan complex.
\end{cor}

\begin{remark} \label{rmk:modelstruct}
We expect fibrations of diagrammatic sets to be those of a cofibrantly generated model structure on $\dgmset$, Quillen equivalent via $\imath_\Delta \dashv \delres{-}$ to the classical model structure on $\sset$: as generating cofibrations one can take the inclusions $\{\bord{}{}U \incl U\}$, and as generating trivial cofibrations the horns $\{\Lambda \incl U\}$, for all atoms $U$. We note that, unlike the situation in $\sset$, only the diagrammatic sets with the Eilenberg-Zilber property would be cofibrant. We postpone this development to future work.
\end{remark}

\begin{remark} \label{rmk:stratified}
A \emph{stratified simplicial set} is a simplicial set $K$ together with a set of \emph{thin} simplices which contains no 0-simplices, and contains all degenerate simplices. Stratified simplicial sets are the underlying structure of the complicial model of weak $\omega$\nbd categories \cite{verity2008weak}: a complicial set is a stratified simplicial set with the property that certain horns have thin fillers, conditional on some of their faces being thin. Kan complexes, as the ``$\omega$\nbd groupoids'' of this model, correspond to the case of all $n$\nbd simplices for $n > 0$ being marked thin.

By our results in Section \ref{sec:closure}, if $X$ is a representable diagrammatic set, then $\delres{X}$ admits a structure of stratified simplicial set whose thin simplices are the equivalences of $X$. If $X$ is Kan, by Proposition \ref{prop:kan_allequiv}, this corresponds to marking all $n$\nbd simplices of $\delres{X}$ as thin: hence, from an $\omega$\nbd groupoid in our model we retrieve an $\omega$\nbd groupoid in the complicial model. We conjecture that this extends to representable diagrammatic sets: that is, if $X$ is a representable diagrammatic set, then $\delres{X}$ with the equivalences of $X$ as thin simplices is a complicial set.
\end{remark}

Next, we define the combinatorial homotopy groups of a Kan diagrammatic set. 

\begin{dfn}
Let $X$ be a Kan diagrammatic set. We define $\pin{0}{}{X}$ to be the set of $\simeq$-equivalence classes of 0-cells $v: 1 \to X$.

Let $v: 1 \to X$ be a 0-cell of $X$. We define $\pin{n}{}{X,v}$ to be the set of $\simeq$-equivalence classes of $n$\nbd cells $x: O^n \to X$ such that
\begin{equation*}
\begin{tikzpicture}[baseline={([yshift=-.5ex]current bounding box.center)}]
	\node[scale=1.25] (0) at (0,1.5) {$\bord{}{}O^n$};
	\node[scale=1.25] (1) at (1.5,1.5) {$1$};
	\node[scale=1.25] (2) at (0,0) {$O^n$};
	\node[scale=1.25] (3) at (1.5,0) {$X$};
	\draw[1c] (0) to node[auto] {$!$} (1);
	\draw[1c] (2) to node[auto,swap] {$x$} (3);
	\draw[1cinc] (0) to (2);
	\draw[1c] (1) to node[auto] {$v$} (3);
\end{tikzpicture}
\end{equation*}
commutes. This becomes a group with the following structure.
\begin{itemize}
	\item Let $[x], [y]$ be two equivalence classes with representatives $x, y$. Then $x \cp{n-1} y$ is a spherical $n$\nbd diagram, which by representability has a weak composite $\compos{x \cp{n-1} y}$. We define $[x] * [y] := [\compos{x \cp{n-1} y}]$.
	\item The unit is the equivalence class $[!^*v]$, where $!$ is the unique morphism $O^n \to 1$. 
	\item Let $[x]$ be an equivalence class with representative $x$. Because every cell of $X$ is an equivalence, $x$ has a weak inverse $x^*$, and we define $\invrs{[x]} := [x^*]$.
\end{itemize}

If $f: X \to Y$ is a morphism of Kan diagrammatic sets and $v: 1 \to X$ a 0-cell of $X$, we also define functions
\begin{equation*}
	\pin{0}{}{f}: \pin{0}{}{X} \to \pin{0}{}{Y}, \quad \quad \pin{n}{}{f}: \pin{n}{}{X,v} \to \pin{n}{}{Y,f(v)}
\end{equation*}
for $n \geq 1$, by $[x] \mapsto [f(x)]$. These are all well-defined since morphisms respect the equivalence relation.
\end{dfn}

\begin{prop}
For all $n \geq 1$, $\pin{n}{}{X,v}$ is well-defined as a group. If $f: X \to Y$ is a morphism of Kan diagrammatic sets, then $\pin{n}{}{f}$ is a homomorphism of groups.
\end{prop}
\begin{proof}
There are several things to check.
\begin{enumerate}
	\item \emph{$[x]*[y]$ is independent of the representatives $x, y$.} Suppose $x \simeq x'$, $y \simeq y'$, and let $z$, $z'$ be weak composites of $x\cp{n-1}y$ and $x'\cp{n-1}y'$, respectively. By Proposition \ref{prop:equi-subst}, we have 
	\begin{equation*}
		x \cp{n-1} y \simeq (x \cp{n-1} y)[y'/y][x'/x] = x' \cp{n-1} y',
	\end{equation*}
	hence $z \simeq z'$. 
	
	\item \emph{Multiplication is associative.} Using Proposition \ref{prop:equi-subst}, we have
	\begin{equation*}
		\compos{\compos{x \cp{n-1} y}\cp{n-1}z} \simeq x \cp{n-1} y \cp{n-1} z \simeq \compos{x \cp{n-1} \compos{y\cp{n-1}z}}.
	\end{equation*}
	
	\item \emph{$[!^*v]$ is a unit.} For all $x$, precomposing $p_{O^n}^*x: O^{n+1} \to X$ with $p_1: \compglob{n+1} \surj O^{n+1}$ yields a cell $l_x: x \celto\, !^*v\cp{n-1}x$, and precomposing with $p_2: \compglob{n+1} \surj O^{n+1}$ a cell $r_x: x \celto x \cp{n-1} !^*v$, exhibiting $[x] = [!^*v] * [x] = [x]*[!^*v]$.
	
	\item \emph{Each element has a two-sided inverse.} Since $x^*$ is a weak inverse of $x$, there are cells $e: \,!^*v \celto x \cp{n-1} x^* $ and $e': \,!^*v \celto x^* \cp{n-1} x$, exhibiting $\invrs{[x]}$ as an inverse of $[x]$.
\end{enumerate}
This completes the proof that $\pin{n}{}{X,v}$ is a group. 

If $[x]*[y] = [z]$, exhibited by $k: z \celto x \cp{n-1} y$, then $f(k): f(z) \celto f(x) \cp{n-1} f(y)$ exhibits $[f(x)]*[f(y)] = [f(z)]$. Moreover, $f(!^*v) = \,!^*f(v)$. This proves that $\pin{n}{}{f}$ is a homomorphism of groups.
\end{proof}

\begin{remark}
The groups $\pin{n}{}{X}$ are abelian for $n \geq 2$. While it does not seem too hard to show this directly, generalising the construction of braidings from Example \ref{exm:braiding}, it will also follow from the isomorphism between the homotopy groups of $X$ and those of the Kan complex $\delres{X}$, that we will soon prove.
\end{remark}

Next, we recall the definition of the combinatorial homotopy groups of a Kan complex. We adopt the opposite convention with respect to \cite[Section I.7]{goerss2009simplicial}, by having the ``relevant'' faces of simplices be the first few, as opposed to the last few, to avoid that they oscillate between input and output boundaries; the two conventions are easily determined to be equivalent. 

If $x: \Delta^n \to K$ is an $n$\nbd simplex in a simplicial set $K$, we write $d_k x := d^k;x$ and $\bord{}{}x = (d_0 x,d_1 x,\ldots,d_n x)$. 

\begin{dfn}
Let $K$ be a Kan complex. We define $\pin{0}{\Delta}{K}$ to be the set of equivalence classes of 0-simplices $v: \Delta^0 \to K$, where $v \sim w$ if there exists a 1-simplex $x$ with $\bord{}{}x = (w,v)$. 

Let $v: \Delta^0 \to K$ be a 0-simplex of $K$. We define $\pin{n}{\Delta}{K,v}$ to be the set of equivalence classes of $n$\nbd simplices $x: \Delta^n \to K$ such that
\begin{equation*}
\begin{tikzpicture}[baseline={([yshift=-.5ex]current bounding box.center)}]
	\node[scale=1.25] (0) at (0,1.5) {$\bord{}{}\Delta^n$};
	\node[scale=1.25] (1) at (1.5,1.5) {$\Delta^0$};
	\node[scale=1.25] (2) at (0,0) {$\Delta^n$};
	\node[scale=1.25] (3) at (1.5,0) {$K$};
	\draw[1c] (0) to node[auto] {$!$} (1);
	\draw[1c] (2) to node[auto,swap] {$x$} (3);
	\draw[1cinc] (0) to (2);
	\draw[1c] (1) to node[auto] {$v$} (3);
\end{tikzpicture}
\end{equation*}
commutes, where $x \sim y$ if there exists a simplicial homotopy from $x$ to $y$ relative to $\bord{}{}\Delta^n$, that is, a map $h: \Delta^0 \times \Delta^n \to K$ such that the diagrams
\begin{equation*}
\begin{tikzpicture}[baseline={([yshift=-.5ex]current bounding box.center)}]
	\node[scale=1.25] (0) at (0,1.5) {$\Delta^0 \times \bord{}{}\Delta^n$};
	\node[scale=1.25] (1) at (2.5,1.5) {$\Delta^0$};
	\node[scale=1.25] (2) at (0,0) {$\Delta^0 \times \Delta^n$};
	\node[scale=1.25] (3) at (2.5,0) {$K$,};
	\draw[1c] (0) to node[auto] {$!$} (1);
	\draw[1c] (2) to node[auto,swap] {$h$} (3);
	\draw[1cinc] (0) to (2);
	\draw[1c] (1) to node[auto] {$v$} (3);
\end{tikzpicture} \quad \quad \quad
\begin{tikzpicture}[baseline={([yshift=-.5ex]current bounding box.center)}]
	\node[scale=1.25] (0) at (0,1.5) {$\Delta^n$};
	\node[scale=1.25] (1) at (2.5,0) {$K$};
	\node[scale=1.25] (2) at (2.5,1.5) {$\Delta^1 \times \Delta^n$};
	\node[scale=1.25] (3) at (5,1.5) {$\Delta^n$};
	\draw[1cinc] (0) to node[auto] {$(d^1,\idcat{})$} (2);
	\draw[1cincl] (3) to node[auto,swap] {$(d^0,\idcat{})$} (2);
	\draw[1c] (0) to node[auto,swap] {$x$} (1);
	\draw[1c] (2) to node[auto] {$h$} (1);
	\draw[1c] (3) to node[auto] {$y$} (1);
\end{tikzpicture}
\end{equation*}
commute.  This becomes a group as follows.
\begin{itemize}
	\item Let $[x], [y]$ be two equivalence classes with representatives $x, y$. There is a well-defined horn $\Lambda^{n+1}_1 \to K$ equal to $y$ on $d^0$, to $x$ on $d^2$, and to $!^*v$ on $d^i$ for $i > 2$. This has a filler $h: \Delta^{n+1} \to K$, and we define $[x]*[y] := [d_1 h]$. 
	\item The unit is the equivalence class $[!^*v]$.
	\item Let $[x]$ be an equivalence class with representative $x$. There is a well-defined horn $\Lambda^{n+1}_0 \to K$ equal to $x$ on $d^2$ and to $!^*v$ on all other faces. This has a filler $h: \Delta^{n+1} \to K$, and we define $\invrs{[x]} := [d_0 h]$. 
\end{itemize}
If $f: K \to L$ is a map between Kan complexes, there are also functions
\begin{equation*}
	\pin{0}{\Delta}{f}: \pin{0}{\Delta}{K} \to \pin{0}{\Delta}{L}, \quad \quad \pin{n}{\Delta}{f}: \pin{n}{\Delta}{K,v} \to \pin{n}{\Delta}{L,f(v)}
\end{equation*}
for $n \geq 1$, defined by $[x] \mapsto [f(x)]$.
\end{dfn}

\begin{prop}\emph{\cite[Theorem I.7.2]{goerss2009simplicial}}
For all $n \geq 1$, $\pin{n}{\Delta}{X,v}$ is well-defined as a group, and abelian for $n \geq 2$. If $f: K \to L$ is a map between Kan complexes, then $\pin{n}{\Delta}{f}$ is a group homomorphism.
\end{prop}

The following is trivial.
\begin{prop} \label{prop:pi0_simplicial}
Let $X$ be a Kan diagrammatic set. Then $\pin{0}{}{X}$ and $\pin{0}{\Delta}{\delres{X}}$ are naturally isomorphic. 
\end{prop}
\begin{proof}
The $n$\nbd cells of $X$ are in bijection with the $n$\nbd simplices of $\delres{X}$ for $n = 0,1$. 
\end{proof}

Now, we connect the homotopy groups of a Kan diagrammatic set $X$ with those of the Kan complex $\delres{X}$ by a sequence of homomorphisms.

\begin{cons}
Let $X$ be a Kan diagrammatic set, and $v: 1 \to X$; we still write $v$ for $\delres{v}: \Delta^0 \equiv \delres{1} \to \delres{X}$. For $n \geq 1$, we define a function
\begin{align*}
	\alpha_n: \pin{n}{}{X,v} & \to \pin{n}{\Delta}{\delres{X},v}, \\
	[x] & \mapsto [a_n^*x],
\end{align*}
where $a_n: \Delta^n \surj O^n$ is the map of Definition \ref{dfn:an_maps}. 

First of all, $\alpha_n$ is well-defined, that is, independent of the representative of $[x]$: if $x \simeq y$, that is, there exists an $(n+1)$\nbd cell $h: x \celto y$, then $a_{n+1}^*h$ is an $(n+1)$\nbd simplex of $\delres{X}$ with
\begin{equation*}
	\bord{}{}(a_{n+1}^*h) = (a_n^*y, a_n^*x, !^*v, \ldots, !^*v)
\end{equation*}
by the commutativity of (\ref{eq:an_faces}). This exhibits 
\begin{equation*}
	[a_n^*x] = [!^*v]*[a_n^*y] = [a_n^*y]
\end{equation*}
in $\pin{n}{\Delta}{\delres{X},v}$.

Moreover, $\alpha_n$ is a homomorphism of groups. We have $!^*v = a_n^*(!^*v): \Delta^n \to X$, so $\alpha_n$ preserves the unit. If $k: z \celto x \cp{n-1} y$ is an $(n+1)$\nbd cell of shape $\compglob{n+1}$ exhibiting $[x]*[y] = [z]$, then $c_{n+1}^*k$ is an $(n+1)$\nbd simplex with
\begin{equation*}
	\bord{}{}(c_{n+1}^*k) = (a_n^*y, a_n^*z, a_n^*x, !^*v, \ldots, !^*v)
\end{equation*}
by the commutativity of (\ref{eq:cn_faces}). This exhibits $\alpha_n[x]*\alpha_n[y] = \alpha_n[z]$ in $\pin{n}{\Delta}{\delres{X},v}$.

Finally, the $\alpha_n$ are natural in the sense that, if $f: X \to Y$ is a morphism of Kan diagrammatic sets, then the square
\begin{equation} \label{eq:naturality_alphan}
\begin{tikzpicture}[baseline={([yshift=-.5ex]current bounding box.center)}]
	\node[scale=1.25] (0) at (0,1.5) {$\pin{n}{}{X,v}$};
	\node[scale=1.25] (1) at (4,1.5) {$\pin{n}{}{Y,f(v)}$};
	\node[scale=1.25] (2) at (0,0) {$\pin{n}{\Delta}{\delres{X},v}$};
	\node[scale=1.25] (3) at (4,0) {$\pin{n}{\Delta}{\delres{Y},\delres{f}(v)}$};
	\draw[1c] (0) to node[auto] {$\pin{n}{}{f}$} (1);
	\draw[1c] (2) to node[auto,swap] {$\pin{n}{\Delta}{\delres{f}}$} (3);
	\draw[1c] (0) to node[auto,swap] {$\alpha_n$} (2);
	\draw[1c] (1) to node[auto] {$\alpha_n$} (3);
\end{tikzpicture}
\end{equation}
commutes in the category of groups and homomorphisms.
\end{cons}

\begin{thm} \label{thm:homotopy_group_iso}
For all $n \geq 1$, $\alpha_n: \pin{n}{}{X,v} \to \pin{n}{\Delta}{\delres{X},v}$ is an isomorphism.
\end{thm}
\begin{proof}
To prove surjectivity, we need to show that, for all elements $[x]$ of $\pin{n}{\Delta}{\delres{X},v}$, represented by $x: \Delta^n \to X$, there exists an $n$\nbd cell $\tilde{x}: O^n \to X$ such that $[x] = [a_n^*\tilde{x}]$. For $n = 1$, there is nothing to prove, so suppose $n > 1$. 

Recall the molecules $\extrtil{k}{n}$ of Construction \ref{cons:delta_extract}, and let $U$ be the $(n+1)$\nbd atom $\Delta^n \celto \extrtil{0}{n}$, and $\Lambda \incl U$ its horn with the greatest element of $O^n \subsph \extrtil{0}{n} = \bord{}{+}U$ missing. Because $\bord{}{\alpha}x =\ !;v$, there is a well-defined morphism $\Lambda \to X$ equal to $x$ on $\bord{}{-}U$, and equal to $!;v$ everywhere else. Filling the horn, we obtain a morphism $h: U \to X$, and we define $\tilde{x} := \restr{\bord{}{+}h}{O^n}: O^n \to X$.

Now, let $p: (\extrtil{0}{n} \celto \Delta^n) \surj O^n$ be the surjective map 
\begin{enumerate}
	\item equal to $\tilde{r}_{0,n}$ on the input boundary $\extrtil{0}{n}$, 
	\item to $a_n$ on the output boundary $\Delta^n$, and 
	\item sending the greatest element of $(\extrtil{0}{n} \celto \Delta^n)$ to the greatest element of $O^n$;
\end{enumerate}
this is well-defined by the commutativity of (\ref{eq:retract_r0n}). Then, consider the $(n+1)$\nbd cell $p^*\tilde{x}$: because $\tilde{r}_{0,n}$ is a retraction of $\extrtil{0}{n}$ onto its submolecule $O^n$, we have that $\bord{}{-}(p^*\tilde{x})$ is equal to $\tilde{x}$ on $O^n$ and to $!;v$ everywhere else, that is, it is equal to $\bord{}{+}h$; whereas $\bord{}{+}(p^*\tilde{x})$ is, by definition, $a_n^*\tilde{x}$. 

It follows that the $(n+1)$\nbd diagram $h\cp{n}(p^*\tilde{x}): x \celto a_n^*\tilde{x}$ is well-defined and regular, therefore it has a weak composite $k: x \celto a_n^*\tilde{x}$ of shape $O(\Delta^n)$. Then $(s^0_\prec)^*k: \Delta^{n+1} \to X$ is an $(n+1)$\nbd simplex with
\begin{equation*}
	\bord{}{}(s^0_\prec)^*k = (a_n^*\tilde{x},x,!^*v,\ldots,!^*v),
\end{equation*}
exhibiting $[a_n^*\tilde{x}] = [!^*v] * [a_n^*\tilde{x}] = [x]$. This proves surjectivity.

To prove injectivity, because the $\alpha_n$ are homomorphisms, it suffices to show that they have a trivial kernel, that is, if $[a_n^*x] = [!^*v]$ in $\pin{n}{\Delta}{\delres{X},v}$, then $[x] = [!^*v]$ in $\pin{n}{}{X,v}$. By \cite[Lemma I.7.4]{goerss2009simplicial}, the premise is equivalent to the existence of an $(n+1)$\nbd simplex $h$ with $\bord{}{}h = (a_n^*x,!^*v,\ldots,!^*v)$. 

First, let $U := \Delta^{n+1} \celto \extr{0}{n+1}$, and let $\Lambda \incl U$ be the horn with the greatest element of $O(\Delta^n) \subsph \extr{0}{n+1}$ missing. There is a morphism $\Lambda \to X$ defined as $h$ on $\bord{}{-}U$, and as $!^*v$ on $\Delta^{n+1} \subsph \bord{}{+}U$; filling the horn, we obtain an $(n+1)$\nbd cell $h':\ !^*v \celto a_n^*x$ of shape $O(\Delta^n)$. 

Next, let $V$ be the colimit of
\begin{equation*} 
\begin{tikzpicture}[baseline={([yshift=-.5ex]current bounding box.center)}]
	\node[scale=1.25] (0) at (-2,-1.5) {$\Delta^n \celto \extrtil{0}{n}$};
	\node[scale=1.25] (1) at (0,0) {$O^n$};
	\node[scale=1.25] (2) at (2,-1.5) {$O^{n+1}$};
	\node[scale=1.25] (3) at (4,0) {$O^n$};
	\node[scale=1.25] (4) at (6,-1.5) {$\extrtil{0}{n} \celto \Delta^n$,};
	\draw[1cincl] (1) to node[auto,swap] {$j_1$} (0);
	\draw[1cinc] (1) to node[auto] {$\imath^-$} (2);
	\draw[1cincl] (3) to node[auto,swap] {$\imath^+$} (2);
	\draw[1cinc] (3) to node[auto] {$j_2$} (4);
\end{tikzpicture}
\end{equation*}
where $j_1$ and $j_2$ are the inclusion of $O^n$ as a submolecule of $\extrtil{0}{n}$. This is a well-defined regular $(n+1)$\nbd molecule with spherical boundary, with three maximal atoms isomorphic to $\Delta^n \celto \extrtil{0}{n}$, $O^{n+1}$, and $\extrtil{0}{n} \celto \Delta^n$, respectively. 

Moreover, $W := O(\Delta^n) \celto V$ is also well-defined as an $(n+2)$\nbd atom; we let $\Lambda' \incl W$ be the horn with $O^{n+1} \subsph \bord{}{+}W$ missing. There is a morphism $\Lambda' \to X$ 
\begin{enumerate}
	\item equal to $h'$ on $\bord{}{-}W$,
	\item equal to $!;v$ on $(\Delta^n \celto \extrtil{0}{n}) \subsph \bord{}{+}W$, and
	\item equal to $p^*x$ on $(\extrtil{0}{n} \celto \Delta^n) \subsph \bord{}{+}W$, where $p: (\extrtil{0}{n} \celto \Delta^n) \surj O^n$ is the surjective map we defined earlier.
\end{enumerate}
This is well-defined, since $\bord{}{+}p^*x = a_n^*x = \bord{}{+}h'$ and $\bord{}{-}(!;v) =\ !;v = \bord{}{-}h'$. Filling the horn, we obtain an $(n+2)$\nbd cell of shape $W$, whose restriction to $O^{n+1} \subsph \bord{}{+}W$ is an $(n+1)$\nbd cell $k:\ !^*v \celto x$. Therefore, $[!^*v] = [x]$ in $\pin{n}{}{X,v}$. This proves injectivity.
\end{proof}

\begin{dfn}
Let $f: X \to Y$ be a morphism of Kan diagrammatic sets. We say that $f$ is a \emph{weak equivalence} if $\pin{0}{}{f}: \pin{0}{}{X} \to \pin{0}{}{Y}$ and $\pin{n}{}{f}: \pin{n}{}{X,v} \to \pin{n}{}{Y,f(v)}$ are isomorphisms for all $n > 0$ and $v: 1 \to X$. 
\end{dfn}

Recall that, in the classical model structure on $\sset$, a map $f: K \to L$ of Kan complexes is a weak equivalence if and only if $\pin{n}{\Delta}{f}$ is an isomorphism for all $n \geq 0$ and all choices of 0-simplices $v: \Delta^0 \to K$.

\begin{cor}
If $f: X \to Y$ is a weak equivalence of Kan diagrammatic sets, then $\delres{f}: \delres{X} \to \delres{Y}$ is a weak equivalence of Kan complexes.
\end{cor}
\begin{proof}
Follows from Proposition \ref{prop:pi0_simplicial}, Theorem \ref{thm:homotopy_group_iso} and naturality, that is, commutativity of (\ref{eq:naturality_alphan}).
\end{proof}

\subsection{Geometric realisation} \label{sec:realisation}

In the previous section, we connected the homotopy theory of diagrammatic sets to that of simplicial sets; the aim of this section is to connect it to the usual homotopy theory of spaces. We will show that it is possible to associate a Kan diagrammatic set $SX$ to each space $X$, in such a way that the combinatorial homotopy groups of $SX$ are naturally isomorphic to the homotopy groups of $X$. 

First, we need to construct a second adjunction between $\dgmset$ and $\sset$.

\begin{dfn}
Let $P$ be a poset. The \emph{nerve} of $P$ is the simplicial set $NP$ whose
\begin{itemize}
	\item $n$\nbd simplices are chains $(x_0 \leq \ldots \leq x_n)$ of length $(n+1)$ in $P$,
	\item the $k$\nbd th face map $d_k: NP_n \to NP_{n-1}$ is defined by 
	\begin{equation*}
		(x_0 \leq \ldots \leq x_n) \mapsto (x_0 \leq \ldots \leq x_{k-1} \leq x_{k+1} \leq \ldots \leq x_n),
	\end{equation*}
	\item the $k$\nbd th degeneracy map $s_k: NP_n \to NP_{n+1}$ is defined by 
	\begin{equation*}
		(x_0 \leq \ldots \leq x_n) \mapsto (x_0 \leq \ldots \leq x_{k} \leq x_{k} \leq \ldots \leq x_n),
	\end{equation*} 
\end{itemize}
for $0 \leq k \leq n$. 
\end{dfn}

The nerve extends to a functor $N: \pos \to \sset$. Precomposing with the forgetful functor $\atom \to \pos$, we obtain a functor $k: \atom \to \sset$. We then take its left Kan extension $k: \dgmset \to \sset$ along the Yoneda embedding of $\atom$ into $\dgmset$, defined by the coend
\begin{equation*}
	kX := \int^{U \in \atom} kU \times X(U).
\end{equation*}
This has a right adjoint $p: \sset \to \dgmset$, defined on a simplicial set $K$ by
\begin{equation*}
	pK(-) := \homset{\sset}(k-, K).
\end{equation*}
Thus, we have a pair of adjunctions
\begin{equation*}
\begin{tikzpicture}[baseline={([yshift=-.5ex]current bounding box.center)}]
	\node[scale=1.25] (0) at (0,0) {$\sset$};
	\node[scale=1.25] (1) at (3,0) {$\dgmset$};
	\node[scale=1.25] (2) at (6,0) {$\sset.$};
	\draw[1c, out=30,in=150] (0.east |- 0,.15) to node[auto] {$\imath_\Delta$} (1.west |- 0,.15);
	\draw[1c, out=-150,in=-30] (1.west |- 0,-.15) to node[auto] {$\delres{-}$} (0.east |- 0,-.15);
	\node[scale=1.25] at (1.4,0) {$\bot$};
	\draw[1c, out=30,in=150] (1.east |- 0,.15) to node[auto] {$k$} (2.west |- 0,.15);
	\draw[1c, out=-150,in=-30] (2.west |- 0,-.15) to node[auto] {$p$} (1.east |- 0,-.15);
	\node[scale=1.25] at (4.6,0) {$\bot$};
\end{tikzpicture}
\end{equation*}
We claim that, up to natural isomorphism, $\imath_\Delta;k$ is equal to the \emph{barycentric subdivision} endofunctor $\subdiv$, and $p;\delres{-}$ to its right adjoint $\exfun$; see \cite[Section 4.6]{fritsch1990cellular} for a review.

\begin{enumerate}
	\item The restriction of $\atom \to \pos$ to $\deltacat \incl \atom$ is precisely the functor sending the $n$\nbd simplex to its poset of non-degenerate simplices, ordered by inclusion. By definition, its post-composition with $N: \pos \to \sset$ is the barycentric subdivision functor $\subdiv: \deltacat \to \sset$.
	
	\item The general barycentric subdivision functor $\subdiv: \sset \to \sset$ is defined as the left Kan extension of $\subdiv: \deltacat \to \sset$ along the Yoneda embedding $\deltacat \incl \sset$. Now, the following diagram of functors commutes up to natural isomorphism:
\begin{equation*}
\begin{tikzpicture}[baseline={([yshift=-.5ex]current bounding box.center)}]
	\node[scale=1.25] (0) at (-1.5,1.5) {$\sset$};
	\node[scale=1.25] (1) at (1,1.5) {$\dgmset$};
	\node[scale=1.25] (0b) at (-1.5,0) {$\deltacat$};
	\node[scale=1.25] (1b) at (1,0) {$\atom$};
	\node[scale=1.25] (2b) at (3.5,0) {$\sset$,};
	\draw[1c] (0) to node[auto] {$\imath_\Delta$} (1);
	\draw[1c,out=0,in=105] (1) to node[auto] {$k$} (2b);
	\draw[1cinc] (0b) to (1b);
	\draw[1c] (1b) to node[auto] {$k$} (2b);
	\draw[1cinc] (0b) to (0);
	\draw[1cinc] (1b) to (1);
	\draw[1c,out=-30,in=-150] (0b) to node[auto,swap] {$\subdiv$} (2b);
\end{tikzpicture}
\end{equation*}
	and $\imath_\Delta$ is the left Kan extension of $\deltacat \incl \dgmset$ along $\deltacat \incl \sset$. Being a left adjoint, $k$ preserves left Kan extensions; their essential uniqueness then guarantees that $\imath_\Delta;k$ and $\subdiv$ are naturally isomorphic. 
	
	\item The endofunctor $\exfun$ of $\sset$ is the right adjoint of $\subdiv$. Since $\imath_\Delta;k \dashv p;\delres{-}$, it follows that $p;\delres{-}$ and $\exfun$ are naturally isomorphic.
\end{enumerate}

\begin{remark}
We can also use the ``semi-simplicial nerve'' $N: \pos \to \semisset$ to obtain a pair of functors $k': \rpol \to \semisset$ and $p': \semisset \to \rpol$, and adjunctions
\begin{equation*}
\begin{tikzpicture}[baseline={([yshift=-.5ex]current bounding box.center)}]
	\node[scale=1.25] (0) at (0,0) {$\semisset$};
	\node[scale=1.25] (1) at (3,0) {$\rpol$};
	\node[scale=1.25] (2) at (6,0) {$\semisset.$};
	\draw[1c, out=30,in=150] (0.east |- 0,.15) to node[auto] {$\imath_\Delta$} (1.west |- 0,.15);
	\draw[1c, out=-150,in=-30] (1.west |- 0,-.15) to node[auto] {$\delres{-}$} (0.east |- 0,-.15);
	\node[scale=1.25] at (1.555,0) {$\bot$};
	\draw[1c, out=30,in=150] (1.east |- 0,.15) to node[auto] {$k'$} (2.west |- 0,.15);
	\draw[1c, out=-150,in=-30] (2.west |- 0,-.15) to node[auto] {$p'$} (1.east |- 0,-.15);
	\node[scale=1.25] at (4.45,0) {$\bot$};
\end{tikzpicture}
\end{equation*}
The composites $\imath_\Delta;k'$ and $p';\delres{-}$ are equal, up to natural isomorphism, to the semi-simplicial versions of $\subdiv$ and $\exfun$. 
\end{remark}

We need the following facts about $\subdiv$ and $\exfun$.
\begin{cons}
For all $n \geq 0$, let $[n]$ be the linear order $\{0 < \ldots < n\}$. There is a map of posets $\gamma_n: \Delta^n \to [n]$, defined as 
\begin{equation*}
	\top^{j_1}\bot^{k_1}\ldots\top^{j_{m}}\bot^{k_m} \quad \mapsto \quad n - k_m.
\end{equation*}
Let $d_{\Delta^n} := N(\gamma_n): \subdiv \Delta^n \to \Delta^n$; this is called the \emph{last vertex map}. This family of maps is natural on $\deltacat$, hence extends uniquely to a natural transformation $d: \subdiv \to \idcat{\sset}$. By adjointness, we also obtain a natural transformation $e: \idcat{\sset} \to \exfun$.
\end{cons}
\begin{prop}\emph{\cite[Corollary 4.6.21]{fritsch1990cellular}} \label{prop:exfun_weakequiv}
Let $K$ be a Kan complex. Then $\exfun K$ is a Kan complex, and $e_K: K \to \exfun K$ is a weak equivalence.
\end{prop}

Let $\cghaus$ be the ``convenient'' category of compactly generated Hausdorff spaces. The usual geometric realisation of simplicial sets is a functor $\realis{-}_\Delta: \sset \to \cghaus$, with a right adjoint $S_\Delta: \cghaus \to \sset$, defined on a space $X$ by 
\begin{equation*}
	S_\Delta X(-) := \homset{\cghaus}(\realis{-}_\Delta, X).
\end{equation*}
We define 
\begin{align*} 
	\realis{-} & := k;\realis{-}_\Delta: \dgmset \to \cghaus, \\
	S & := S_\Delta; p: \cghaus \to \dgmset.
\end{align*}
We call $\realis{X}$ the \emph{geometric realisation} of the diagrammatic set $X$, and $SY$ the \emph{singular diagrammatic set} of the space $Y$. 

Let $D^n$ be the topological closed $n$\nbd ball, $S^{n-1} = \bord{}{}D^n$ the topological $(n-1)$\nbd sphere, and $D^{n-1} \incl D^n$ the inclusion of $D^n$ as a hemisphere of the boundary of $D^n$. In the following proof, we implicitly use some basic facts of combinatorial topology: that the realisation of the nerve of a poset is homeomorphic to the realisation of its \emph{order complex} (an ordered simplicial complex), and that it is compatible with boundaries, unions, and intersections of closed subsets. We refer to \cite{bjorner1995topological} as a general reference.

\begin{remark}
A similar result is stated, for the restricted case of loop-free pasting schemes, as \cite[Theorem 2.2]{kapranov1991combinatorial} without proof; from the authors' few words on the matter, we believe that the proof is essentially the same.
\end{remark}

\begin{prop} \label{prop:globe_realis}
Let $U$ be a regular $n$\nbd molecule with spherical boundary. Then $\realis{U}$ is homeomorphic to $D^n$ and $\realis{\bord{}{}U}$ is homeomorphic to $S^{n-1}$. If $U$ is an atom and $\Lambda \incl U$ is a horn of $U$, then $\realis{\Lambda} \incl \realis{U}$ is equal to the inclusion $D^{n-1} \incl D^n$ up to homeomorphism.
\end{prop}
\begin{proof} For $n = 0$, this is obvious, so let $n > 0$. By the inductive hypothesis, $\realis{\bord{}{+}U}$ and $\realis{\bord{}{-}U}$ are $(n-1)$\nbd balls, and their intersection $\realis{\bord{n-2}{}U}$ is an $(n-2)$\nbd sphere; it follows from \cite[Theorem 2]{zeeman1966seminar} that their union $\realis{\bord{}{}U}$ is an $(n-1)$\nbd sphere. If $U$ is an atom, this suffices to prove that $\realis{U}$ is an $n$\nbd ball, for example by \cite[Proposition 3.1]{bjorner1984posets}. 

Suppose $U$ is not an atom; then $U = U_1 \cp{n-1} \ldots \cp{n-1} U_m$ as in Lemma \ref{lem:composition_form}, where each of the $U_i$ contains a single $n$\nbd dimensional element $x_i$. Moreover, $U$ is pure, hence equal to the union of the $\clos\{x_i\}$. Now, let 
\begin{equation*}
	V := (\bord{}{-}U \celto \bord{}{-}U) \cp{n-1} U,
\end{equation*}
and then, by recursion on $i = 0,\ldots,m$,
\begin{equation*}
	\tilde{U}_0 := \bord{}{-}U \celto \bord{}{-}U, \quad \quad \tilde{U}_i := \tilde{U}_{i-1} \cup \clos\{x_i\}.
\end{equation*}
We have that $\realis{\tilde{U}_0}$ is a closed $n$\nbd ball. Assuming $\realis{\tilde{U}_{i-1}}$ is a closed $n$\nbd ball, by Proposition \ref{prop:boundary_submol} and the inductive hypothesis $\realis{\tilde{U}_i}$ is the union of the closed $n$\nbd balls $\realis{\clos\{x_i\}}$ and $\realis{\tilde{U}_{i-1}}$ along a closed $(n-1)$\nbd ball in their boundaries; it follows from \cite[Theorem 2]{zeeman1966seminar} that $\realis{\tilde{U}_i}$ is a closed $n$\nbd ball. Thus $\realis{V}$ is a closed $n$\nbd ball. Finally, as in the first part, given
\begin{equation*}
	W := (\bord{}{-}U \celto \bord{}{+}U) \celto V,
\end{equation*}
we have that $\realis{\bord{}{}W}$ is an $n$\nbd sphere. Since $U$ has spherical boundary,
\begin{equation*}
	\realis{(\bord{}{-}U \celto \bord{}{-}U) \cup (\bord{}{-}U \celto \bord{}{+}U)} \incl \realis{\bord{}{}W}
\end{equation*}
is the embedding of a closed $n$\nbd ball in an $n$\nbd sphere, so by \cite[Theorem 3]{zeeman1966seminar} its closed complement $\realis{U}$ is a closed $n$\nbd ball. The statement about realisations of horns follows immediately from the same result.
\end{proof}

\begin{dfn}
Let $X$ be a CW complex. The \emph{face poset} $\face{X}$ of $X$ is the poset whose elements are the generating cells of $X$, and for any pair of generating cells $x: D^k \to X$ and $y: D^n \to X$ we have $x \leq y$ if and only if $x(D^k) \subseteq y(D^n)$.

We say that $X$ is \emph{regular} if each generating cell $x: D^n \to X$ is a homeomorphism onto its image. 
\end{dfn}

By \cite[Theorem 1.7]{lundell1969topology}, regular CW complexes are determined up to homeomorphism by their face poset: that is, if $X$ is a regular CW complex, then $\realis{N(\face{X})}$ is homeomorphic to $X$. The following result justifies our use of the adjective ``regular''.

\begin{prop}
Let $P$ be a regular directed complex. Then the underlying poset of $P$ is the face poset of a regular CW complex.
\end{prop}
\begin{proof}
Follows from \cite[Proposition 3.1]{bjorner1984posets} together with Proposition \ref{prop:globe_realis}.
\end{proof}

Recall that a \emph{Serre fibration} is a map of spaces which has the right lifting property with respect to the inclusions $D^{n-1} \incl D^n$ for $n \geq 0$.

\begin{prop}
Let $f: X \to Y$ be a Serre fibration. Then $Sf: SX \to SY$ is a fibration of diagrammatic sets.
\end{prop}
\begin{proof}
Follows immediately from Proposition \ref{prop:globe_realis} and adjointness.
\end{proof}

\begin{cor}
For all spaces $X$, the diagrammatic set $SX$ is Kan.
\end{cor}

We conjecture that $\realis{-} \dashv S$ form a Quillen equivalence between the model category structure on $\dgmset$ sketched in Remark \ref{rmk:modelstruct} and the classical model structure on $\cghaus$; this would follow from $k \dashv p$ being a Quillen equivalence. However, we are not yet in the position to prove this, as it seems to require a theory of ``oriented simplicial approximations'' of regular atoms. 

Instead, we will use what we already know about the relation between a Kan diagrammatic set $X$ and the Kan complex $\delres{X}$, and use $\realis{\delres{X}}_\Delta$ as a realisation of $X$. We call this the \emph{simplicial geometric realisation} of a Kan diagrammatic set.

\begin{thm} \label{prop:iso_pin_fibrant}
Let $X$ be a Kan diagrammatic set, $v: 1 \to X$. For all $n > 0$, there are natural isomorphisms
\begin{equation*}
	\pin{0}{}{X} \to \pin{0}{}{\realis{\delres{X}}_\Delta}, \quad \quad \pin{n}{}{X,v} \to \pin{n}{}{\realis{\delres{X}}_\Delta,\realis{v}_\Delta}.
\end{equation*}
\end{thm}
\begin{proof}
It suffices to compose the natural isomorphisms 
\begin{equation*}
	\pin{0}{}{X} \to \pin{0}{\Delta}{\delres{X}}, \quad \quad \pin{n}{}{X,v} \to \pin{n}{\Delta}{\delres{X},v}
\end{equation*}
from the previous section with the natural isomorphisms 
\begin{equation*}
	\pin{0}{\Delta}{K} \to \pin{0}{}{\realis{K}_\Delta}, \quad \quad \pin{n}{\Delta}{K,v} \to \pin{n}{}{\realis{K}_\Delta,\realis{v}_\Delta}
\end{equation*}
defined, for instance, after \cite[Proposition I.11.1]{goerss2009simplicial}.
\end{proof}

\begin{cor} \label{cor:weakequivpres}
If $f: X \to Y$ is a weak equivalence of Kan diagrammatic sets, then $\realis{\delres{f}}_\Delta: \realis{\delres{X}}_\Delta \to \realis{\delres{Y}}_\Delta$ is a weak equivalence of spaces.
\end{cor}

Finally, we prove that every space is naturally weakly equivalent to the simplicial geometric realisation of a Kan diagrammatic set. 

\begin{thm} \label{thm:span_weakeq}
Let $X$ be a space. Then $X$ is weakly equivalent to $\realis{\delres{(SX)}}_\Delta$ via a natural span of weak equivalences.
\end{thm}
\begin{proof}
Let $\varepsilon^\Delta$ be the counit of the adjunction $\realis{-}_\Delta \dashv S_\Delta$; since this is a Quillen equivalence, all the components of $\varepsilon^\Delta$ are weak equivalences. 

By construction, $\exfun (S_\Delta X)$ and $(SX)_\Delta$ are naturally isomorphic for all spaces $X$: therefore, by Proposition \ref{prop:exfun_weakequiv}, we have a family $e_{S_\Delta X}: S_\Delta X \to (SX)_\Delta$ of weak equivalences of Kan complexes, natural in $X$. It follows that
\begin{equation} \label{eq:weakeq_span}
	\begin{tikzpicture}[baseline={([yshift=-.5ex]current bounding box.center)}]
	\node[scale=1.25] (0) at (-2,-1.5) {$X$};
	\node[scale=1.25] (1) at (0,0) {$\realis{S_\Delta X}_\Delta$};
	\node[scale=1.25] (2) at (2,-1.5) {$\realis{\delres{(SX)}}_\Delta$};
	\draw[1c] (1) to node[auto,swap] {$\varepsilon^\Delta_X$} (0);
	\draw[1c] (1) to node[auto] {$\realis{e_{S_\Delta X}}_\Delta$} (2);
	\end{tikzpicture}
\end{equation}
is a span of weak equivalences in $\cghaus$, natural in $X$. 
\end{proof}

\begin{cor} 
Let $X$ be a space, $v: \{*\} \to X$. For all $n > 0$, there are natural isomorphisms
\begin{equation*}
	\pin{0}{}{X} \to \pin{0}{}{SX}, \quad \quad \pin{n}{}{X,v} \to \pin{n}{}{SX,Sv}.
\end{equation*}
\end{cor}
\begin{proof}
Compose together the natural isomorphisms of Theorem \ref{prop:iso_pin_fibrant} with those induced by the natural span of weak equivalences of Theorem \ref{thm:span_weakeq}.
\end{proof}

\begin{cor} \label{cor:homotopyinv}
If $f: X \to Y$ is a weak equivalence of spaces, then $Sf: SX \to SY$ is a weak equivalence of Kan diagrammatic sets.
\end{cor}

Let $\kandgmset$ be the full subcategory of $\dgmset$ (or $\rdgmset$) on the Kan diagrammatic sets. It follows from Corollary \ref{cor:weakequivpres} and \ref{cor:homotopyinv} that the functors 
\begin{equation*}
	\realis{\delres{-}}_\Delta: \kandgmset \to \cghaus, \quad S: \cghaus \to \kandgmset
\end{equation*}
descend to functors between the homotopy categories of $\kandgmset$ and $\cghaus$, that is, their localisations at the weak equivalences, and by Theorem \ref{thm:span_weakeq}, $S$ is homotopically right inverse to $\realis{\delres{-}}_\Delta$. In this sense, Kan diagrammatic sets are ``sufficient'' for modelling homotopy types. We believe that $S$ is in fact a homotopical two-sided inverse, but we did not prove it at this time.

By imitating the simplicial theory, it is easy to prove the stronger statement that $(n+1)$\nbd coskeletal Kan diagrammatic sets model all homotopy $n$\nbd types.

\begin{prop}
Let $X$ be a Kan diagrammatic set. Then $\coskel{n+1}{X}$ is Kan, and for all $v: 1 \to X$,
\begin{equation*}
\begin{aligned}
	& \pin{k}{}{\coskel{n+1}{X},v} \simeq \pin{k}{}{X,v}, & k \leq n, \\
	& \pin{k}{}{\coskel{n+1}{X},v} \simeq \{*\}, & k > n.
\end{aligned}
\end{equation*}
where the isomorphisms for $k \leq n$ are induced by the counit $X \to \coskel{n+1}{X}$.
\end{prop}
\begin{proof}
Same as the proof for Kan complexes, for example \cite[Proposition 8.8]{may1992simplicial}.
\end{proof}

It follows that the sequence of coskeleta (\ref{eq:coskel-tower}) is a Postnikov tower for a Kan diagrammatic set. Combining Proposition \ref{prop:iso_pin_fibrant}, Theorem \ref{thm:span_weakeq}, and their corollaries, we can see that, given a space $X$, the simplicial geometric realisation of the sequence of coskeleta of $SX$ is a Postnikov tower for $X$ up to weak equivalence.

\bibliographystyle{alpha}
\small \bibliography{main}

\end{document}